\UseRawInputEncoding
\documentclass[11pt]{article}
\usepackage[hidelinks]{hyperref}
\usepackage{amsmath,amssymb,mathrsfs,hyperref,color,amsthm}
\usepackage{mathtools}
\usepackage[x11names]{xcolor}
\usepackage{paralist}
\usepackage{lineno}
\usepackage{subfigure}
\usepackage{float}
\usepackage{tikz}
\usepackage{caption}
\usepackage{cases}
\usetikzlibrary{intersections}
\usepackage{xcolor}
\usepackage[latin1]{inputenc}
\usepackage{bbm}
\usepackage{dsfont}
\usepackage[T1]{fontenc}
\usepackage{appendix}
\usepackage{tikz}

\usetikzlibrary{positioning}
\usetikzlibrary{arrows}

\usetikzlibrary{shadings}
\usetikzlibrary[intersections]
\usetikzlibrary{arrows,shapes}

\makeatletter
\def\author#1{\gdef\autrun{\def\and{\unskip, }#1}\gdef\@author{#1}}

\makeatother
\def\email#1{e-mail: #1}

\def\keywords#1{\par\medskip
\noindent\textbf{Keywords.} #1}

\newtheorem{lemma}{Lemma}[section]
\newtheorem{theorem}{Theorem}
\newtheorem{definition}{Definition}[section]
\newtheorem{proposition}{Proposition}[section]
\newtheorem{remarks}{Remark}[section]

\newcommand{\K}{\mathcal{K}}

\newcommand{\NN}{\mathcal{N}}
\newcommand{\TT}{\mathcal{T}}
\newcommand{\Lip}{\mathrm{Lip}}
\newcommand{\mmod}{\mathrm{mod}}
\newcommand{\diam}{\mathrm{diam}}
\newcommand{\cl}{\mathrm{cl}}
\newcommand{\T}{\mathbb{T}}
\newcommand{\R}{\mathbb{R}}
\newcommand{\Z}{\mathbb{Z}}
\newcommand{\N}{\mathbb{N}}

\newcommand{\PP}{\mathcal{P}}
\newcommand{\A}{\mathcal{A}}
\newcommand{\CC}{\mathcal{C}}
 
\newcommand{\MMM}{\mathcal{M}}

\newcommand{\LL}{\mathcal{L}}

\DeclareMathOperator*{\argmin}{argmin} 
\DeclareMathOperator*{\argmax}{argmax}
\DeclareMathOperator*{\supp}{spt}

\DeclareMathOperator*{\ddiv}{div}

\numberwithin{equation}{section}

\begin{document}

%%%%% To ease editing, add:

\baselineskip=15pt

%%%%%%%%%%%%%%%%

%% In the running head, give an abbreviation of the title. 
\title{Discretization and Vanishing Discount Problems for First-order Mean Field Games}

\author{
Renato Iturriaga \thanks{Centro de Investigaci\'on en Math\'ematicas, Valenciana Guanajuato 36000, Mexico; \email{\tt renato@cimat.mx}} \and	
Cristian Mendico \thanks{Institut de Math\'ematique de Bourgogne, UMR 5584 CNRS, Universit\'e Bourgogne, Dijon 21000, France; \email{\tt cristian.mendico@u-bourgogne.fr}} \and
Kaizhi Wang \thanks{School of Mathematical Sciences, CMA-Shanghai, Shanghai Jiao Tong University, Shanghai 200240, China; \email{\tt kzwang@sjtu.edu.cn}} \and  Yuchen Xu \thanks{Corresponding author. School of Mathematical Sciences, CMA-Shanghai, Shanghai Jiao Tong University, Shanghai 200240, China; \email{\tt math\_rain@sjtu.edu.cn}} 
}

%\date{\today}

\maketitle

\begin{abstract}
	This article focuses two issues related to the first-order discounted mean field games system. The first is the time discretization problem. The time discretization approach enables us to prove the existence of solutions $(u,m)$ of the system, where $u$ is a viscosity solution of the discounted Hamilton-Jacobi equation and $m$ is a  projected minimizing measure satisfying the continuity equation in the sense of distributions.
	The second is the vanishing discount problems for  both the discounted mean field games system and its discretized system. The methods we use primarily derive from weak KAM theory. Moreover, we provide an example demonstrating the non-uniqueness of solutions to the discounted mean field games system.
\end{abstract}

\keywords{Discounted mean field games; discretization; vanishing discount problem; weak KAM theory}

\noindent\textbf{MSC (2020).} 35Q89, 37J51, 49N80.
%\end{abstract}

\tableofcontents
%\linenumbers

\bigskip

%%%%%%%%%%%%%%%%%%%%%%%%%%%%%%%%%%%%%
%Section 1
\section{Introduction} \label{introduction}
\setcounter{equation}{0}

\subsection{Purpose of this work}

The theory of mean field games was developed independently and almost simultaneously by Lasry and Lions \cite{MR2269875, MR2271747, MR2295621} and by Huang, Malham\'e and Caines \cite{MR2352434, MR2346927}.
%{\color{red} Huang M, Caines PE, Malham\'e RP (2007) Large-population cost-coupled LQG problems with nonuniform agents: individual-mass behavior and decentralized $\epsilon$-Nash equilibria. IEEE Trans Autom Control  52:1560-1571}
In essence, mean field games provide a powerful mathematical lens through which to view and solve the problem of complexity arising from vast numbers of individual decisions. It is a framework that makes the impossible--modeling millions of interacting entities--possible.

In the present work we focus on two first-order stationary mean field games systems 
\begin{equation}\label{DMFG}
	\begin{cases}
		\lambda u + H(x, Du)= F(x,m), & x \in \T^d \quad \quad \quad  (\ref{DMFG}a)
		\\
		\ddiv \left( m \frac{\partial H}{\partial p}(x, Du) \right) = 0, & x \in \T^d \quad \quad \quad  (\ref{DMFG}b)
		\\
		\int_{\T^d} m \, dx =1, & \quad \quad \quad\quad\quad \quad  (\ref{DMFG}c)
	\end{cases}
\end{equation}
where $\lambda>0$ is the discounted rate, and 
\begin{equation}\label{MFG}
	\begin{cases}
		H(x, Du)= F(x,m) + c(m), & x \in \T^d \quad \quad \quad  (\ref{MFG}a)
		\\
		\ddiv \left( m \frac{\partial H}{\partial p}(x, Du) \right) = 0, & x \in \T^d \quad \quad \quad  (\ref{MFG}b)
		\\
		\int_{\T^d} m \, dx =1, & \quad \quad \quad\quad\quad \quad  (\ref{MFG}c)
	\end{cases}
\end{equation}
where $c(m)$ denotes the Ma\~n\'e's critical value. 
To distinguish between \eqref{DMFG} and \eqref{MFG}, we refer to them respectively as the discounted MFG system (DMFGs) and the  MFG system (MFGs).  The first equation in the above systems is a first-order Hamilton-Jacobi equation, and the second one is a continuity equation. DMFGs \eqref{DMFG} arises in certain free-market economy models, see for instance \cite{MR3363751}, while MFGs \eqref{MFG} usually arises in the study of the large time behavior of first-order mean field games with finite horizon, see \cite{MR3127145,MR4092693,MR4772553} for instance.
%{\color{red} P. Cardaliaguet, Long time average of first order mean field games and weak KAM theory, Dyn. Games Appl., 3 (2013), pp. 473--488.} 
%{\color{red} Long-time behavior of first-order mean field games on Euclidean space
%Cannarsa, Piermarco; Cheng, Wei; Mendico, Cristian; Wang, Kaizhi
%Dyn. Games Appl. 10 (2020), no. 2, 361-390.}

This paper has two objectives: First, to discretize DMFGs \eqref{DMFG} in time, prove the existence of solutions to the discretized system, and further establish the existence of solutions to \eqref{DMFG} by demonstrating the convergence of certain sequence of solutions to the discretized system as the time step approaches zero. Second, to investigate the vanishing discount problem for  \eqref{DMFG} and its discretized counterpart, which will establish a connection between \eqref{DMFG} and \eqref{MFG}.

\tikzset{
	arrow1/.style = {
		draw =black, semithick%, -{Stealth[length = 3mm, width = 2mm]},
	}	
}
\begin{figure}

\begin{tikzpicture}[]
	%\draw[help lines] (-7 ,-7) grid ( 7, 7); 
	\node [draw, inner sep=0.5 cm,on grid] at (-4,7){Time discretization of \eqref{DMFG}};
	\node [draw, inner sep=0.5 cm,on grid] at (5,7){Time discretization of \eqref{MFG}};
	\node [draw, inner sep=0.5 cm,on grid] at (-4,4){Solutions to \eqref{DMFG}};
	\node [draw, inner sep=0.5 cm,on grid] at (5,4){Solutions to \eqref{MFG}};

	\draw[red, -latex] (-1,7)--(2,7) node [midway,above, black]{$\lambda \rightarrow 0$};
	\draw[blue, -latex] (5,6)--(5,5)  node [midway,right, black]{$\tau \rightarrow 0$ };
	\draw[-latex, red] (-4,6)--(-4,5)  node [midway,right, black]{$\tau \rightarrow 0$};
	\draw[-latex, red] (-1.8,4)--(2.9,4)  node [midway,above, black]{$\lambda \rightarrow 0$};
\end{tikzpicture}
\caption{The schematic of the research framework}
\caption*{\small In the above diagram, $\tau$ denotes the time step, and $\lambda$ denotes the discounted rate. The red arrows represent the convergence processes to be studied in this paper, while the convergence process indicated by the blue arrow was discussed in \cite{MR4605206}.} 
\label{fig}
\end{figure}
We use Fig. \ref{fig} to further illustrate the research framework of this paper.
Discretize DMFGs \eqref{DMFG} in time first.
Then we study the asymptotic behavior of solutions to the discretized system of \eqref{DMFG} as $\tau\to 0$ and  the asymptotic behavior of solutions to \eqref{DMFG} as $\lambda\to 0$.
At last we take care of the asymptotic behavior of solutions to the discretized system of \eqref{DMFG} as $\lambda\to 0$.
More precisely, for any given $\lambda > 0$, we prove the existence of solutions $(u_{\tau, m_\tau^\lambda}^\lambda, m_\tau^\lambda)$ to the discretized system of \eqref{DMFG}, where $u_{\tau, m_\tau^\lambda}^\lambda$ is the solution to the discrete Lax-Oleinik equation for (\ref{DMFG}a) and $m_\tau^\lambda$ is the minimizing $\tau$-holonomic $\lambda$-measure for DMFGs \eqref{DMFG}. And then we show that there is a subsequence  $\{( u_{\tau_i, m_{\tau_i}^\lambda}^\lambda, m_{\tau_i}^\lambda)\}_{i\in\mathbb{N}}$ converges to a solution $( u_0^\lambda, m_0^\lambda)$ to DMFGs \eqref{DMFG} as  $i \to \infty$. Next, we prove that there is a subsequence  $\{(u_0^{\lambda_i}, m_0^{\lambda_i})\}_{i\in\mathbb{N}}$ converges to the solution to MFGs \eqref{MFG} as $i \to \infty$.
At last, for any given $\tau > 0$, there is a subsequence $\{(u_{\tau, m_\tau^{\lambda_i}}^{\lambda_i}, m_\tau^{\lambda_i} )\}_{i\in\mathbb{N}}$ converges to  $(u_0^\tau, m_0^\tau )$ as $i \to \infty$, where  $(u_0^\tau, m_0^\tau )$ is a solution to the discretized system of \eqref{MFG}.

\subsection{A brief review of relevant works}
Discrete approximation schemes for first-order evolutionary mean field games systems with finite horizon were developed in \cite{MR2928379, MR3148086}. A discrete weak KAM method for first-order stationary mean field games system \eqref{MFG} was provided in \cite{MR4605206}, where 
the limit process represented by the blue arrow in Fig. \ref{fig} has been studied. To date, the discretization problem for DMFGs \eqref{DMFG} remains largely unexplored.

In addition to the discretization problem, another focus of this paper is the vanishing discount problem for DMFGs \eqref{DMFG} and its discretized counterpart. There exists a substantial body of in-depth research on the vanishing discount problem and, more broadly, the vanishing contact problem for Hamilton-Jacobi equations. See for instance  \cite{MR3952779, MR3556524,MR3670619,MR3581314, MR4275748} and the references therein.
%{\color{red} Chen, Q., Cheng, W., Ishii, H., Zhao, K.: Vanishing contact structure problem and
%convergence of the viscosity solutions. Commun. Partial Differ. Equ. 44(9), 801-836,
%2019.}
%{\color{red} Convergence of viscosity solutions of generalized contact Hamilton-Jacobi equations
%Wang, Ya-Nan; Yan, Jun; Zhang, Jianlu
%Arch. Ration. Mech. Anal. 241 (2021), no. 2, 885-902.}
%{\color{red} more clear}
For first-order stationary MFG systems with local coupling terms $F$, the vanishing discount problem was addressed in \cite{MR4175148, MR4567771} where different concepts of weak solutions from us were used.
They considered a solution pair $(u,m)$, where $u$ is a viscosity subsolution to the Hamilton-Jacobi equation. \cite{MR4175148} dealt with the case where the Hamiltonian has the mechanical form $H(x,p) = \frac{1}{2} \vert p \vert^2 + V(x)$ and the definition of weak solutions comes  from \cite{MR3882950}. \cite{MR4567771} utilized the concept of weak solutions from \cite{MR3408214} to solve the vanishing discount problem for more general Hamiltonians. For second-order stationary MFG systems with non-local coupling terms, the vanishing discount problem was investigated in \cite{MR3921309}.

See for instance \cite{MR4688694,MR4273184,MR3846236,MR4097936,MR4683963}
and the references therein for more recent progresses on first-order mean field games.
%{\color{red} remove}
%Moreover, several works \cite{MR4066802, ma2024continuousdependencemckeanvlasovsdes, MR4608382} considered the continuous dependence of solutions and invariant measures for Mckean-Vlasov stochastic differential equations on initial values and parameters, which provides relevant technical insights for the asymptotic analysis of DMFGs.

\subsection{Assumptions}
Let $\T^d := \R^d / \Z^d$ denote the standard flat torus. Let $H: \T^d \times \R^d \rightarrow \R$ be a Hamiltonian satisfying
\begin{enumerate}[\bfseries (H1)]
	\item \label{MFG_H1} Regularity: $H$ is of class $\CC^2$.
	\item \label{MFG_H2} Strict convexity: for each $(x,p) \in \T^d \times \R^d$, $\frac{\partial ^2 H}{\partial p^2} (x,p)$ is positive definite.
	\item \label{MFG_H3} Superlinearity: for each $K>0$, there exists $C(K) \in \R$ such that
	$$
	H(x,p) \geq K \vert p \vert + C(K), \quad \forall (x,p) \in \T^d \times \R^d.
	$$
\end{enumerate}
We call such a Hamiltonian $H$ a Tonelli Hamiltonian. Denote by $L: \T^d \times \R^d \rightarrow \R$ the associated Lagrangian, defined by
$$
L(x,v) = \sup_{p \in \R^d} \left( \left\langle p,v \right\rangle - H(x,p) \right) , \quad \forall (x,v) \in \T^d \times \R^d.
$$
It is clear that $L$ satisfies
\begin{enumerate}[\bfseries (L1)]
	\item \label{MFG_L1} Regularity: $L$ is of class $\CC^2$.
	\item \label{MFG_L2} Strict convexity: for each $(x,v) \in \T^d \times \R^d$, $\frac{\partial ^2 L}{\partial v^2} (x,v)$ is positive definite.
	\item \label{MFG_L3} Superlinearity: for each $K>0$, there exists $C(K) \in \R$ such that
	$$
	L(x,v) \geq K \vert v \vert + C(K), \quad \forall (x,v) \in \T^d \times \R^d.
	$$
\end{enumerate}

Let $\PP ( \T^d )$ denote the space of Borel probability measures on $\T^d$ endowed with the weak-* convergence. It is convenient to put a metric, i.e., the Kantorovich-Rubinstein distance $d_1$ on $\PP ( \T^d )$, which metricizes the weak-* topology. See Sect. \ref{wasserstein space} for details.
Assume the non-local coupling term $F: \T^d \times \PP ( \T^d ) \rightarrow \R$ satisfies
\begin{enumerate}[\bfseries (F1)]
	\item \label{MFG_F1} For each $m \in \PP ( \T^d )$, the function $x \mapsto F(x,m)$ is of class $\CC^2$, and there is a constant $F_{\infty} >0$ such that
	$$
	\left\Vert F ( \cdot, m ) \right\Vert_{\infty}, \left\Vert D_x F ( \cdot, m ) \right\Vert_{\infty} \leq F_{\infty}, \quad \forall m \in \PP ( \T^d ),
	$$
	where the uniform norm is defined by
	$$
	\left\Vert F ( \cdot, m ) \right\Vert_{\infty} := \sup_{x \in \T^d} \left\vert F(x,m) \right\vert.
	$$
	
	\item \label{MFG_F2} $F( \cdot, \cdot)$ and $D_x F( \cdot, \cdot)$ are continuous on $\T^d \times \PP ( \T^d )$.
	
	\item \label{MFG_F3} There is a constant $\Lip(F) >0$ such that
	$$
	\left\vert F(x,m_1) - F(x,m_2) \right\vert \leq \Lip(F) d_1(m_1,m_2), \quad \forall x \in \T^d,\,\, \forall m_1,m_2 \in \PP ( \T^d ).
	$$
\end{enumerate}

\subsection{Main results}
\begin{definition}
	A solution to DMFGs \eqref{DMFG} is a pair $(u,m)\in \CC ( \T^d ) \times \PP ( \T^d )$ such that (\ref{DMFG}a) is satisfied in the viscosity sense, $D u(x)$ exists for $m$-a.e. $x \in \T^d$, and (\ref{DMFG}b) is satisfied in the sense of distributions.
\end{definition}
\begin{remarks}
The concept of viscosity solutions comes from \cite{MR690039}.
When $D u(x)$ exists for $m$-a.e. $x \in \T^d$, we say that $m \in \PP ( \T^d )$ satisfies the continuity equation in the sense of distributions if
$$
\int_{\T^d} \left\langle D f(x), \frac{\partial H}{\partial p}(x, Du(x)) \right\rangle dm = 0,\quad \forall f \in \CC^\infty ( \T^d ).
$$
\end{remarks}

For any $m \in \PP ( \T^d )$, define
$$
H_m (x,p) := H(x,p) - F(x,m), \quad \forall (x,p) \in \T^d \times \R^d,
$$
$$
L_m(x,v) := L(x,v) + F(x,m), \quad \forall (x,v) \in \T^d \times \R^d.
$$

For any $\lambda > 0$ and any $m \in \PP ( \T^d )$, a $\Z^d$-periodic function $u_m^\lambda \in \CC ( \R^d )$ is the viscosity solution to (\ref{DMFG}a) if and only if it satisfies the Lax-Oleinik equation
$$
u_m^\lambda(y) = \inf_{x \in \R^d} \left( e^{-\lambda t} u_m^\lambda(x) + h_{t, \lambda}^m(x,y) \right), \quad \forall y \in \R^d,\,\, \forall t > 0,
$$
where the minimal action $h_{t, \lambda}^m(x,y)$ is defined by
$$
h_{t, \lambda}^m(x,y) = \inf_{\gamma} \int_{-t}^0 e^{\lambda s} L_m \left( \gamma(s), \dot{\gamma}(s) \right) ds, \quad \forall x,y \in \R^d,
$$
where the infimum is taken over all the absolutely continuous curves $\gamma : [-t, 0] \rightarrow \R^d$ with $\gamma(-t)=x$ and $\gamma(0) =y$.
In fact, the solution is unique, and can be represented as
\begin{equation} \label{equation of solution of HJE_MFG}
u_m^\lambda (x) = \inf_{\sigma} \int_{-\infty}^0 e^{\lambda s} L_m \left( \sigma(s), \dot{\sigma}(s) \right) ds, \quad \forall x \in \R^d,
\end{equation}
where the infimum is taken over all the absolutely continuous curves $\sigma:(-\infty, 0] \rightarrow \R^d$ with $\sigma(0)=x$. See, for instance, \cite{MR3556524, MR3927084} for the proofs of the above facts.

Let $\Phi^t_{L_m , \lambda}$ denote the discounted Euler-Lagrange flow generated by $L_m$ as defined in \eqref{lagrangian flow of DMFG} below.
For every $(x,v) \in \T^d \times \R^d$, let $\gamma_{(x,v)} (t) := \pi \left( \Phi^t_{L_m, \lambda} (x,v) \right)$ for all $t \in \R$, where $\pi: \T^d \times \R^d \rightarrow \T^d$ is the canonical projection.
Define the set
$$
\tilde{\Sigma}_{L_m, \lambda} := \left\{ (x,v) \in \T^d \times \R^d \,\, \middle\vert\,\, \text{the curve $\gamma_{(x,v)}$ is $(u_m^\lambda , L_m)$-$\lambda$-calibrated on $(-\infty, 0]$}  \right\}.
$$
Similar to \cite{MR3663623}, define the discounted Aubry set by
\begin{equation} \label{discounted aubry set}
\tilde{\A}_{L_m, \lambda} := \bigcap_{t \geq 0} \Phi^{-t}_{L_m, \lambda} \left( \tilde{\Sigma}_{L_m, \lambda} \right).
\end{equation}
Denote by $\A_{L_m, \lambda} = \pi (\tilde{\A}_{L_m, \lambda})$ the projected discounted Aubry set. 
Then, $\tilde{\A}_{L_m, \lambda}$ is non-empty, compact and invariant under the discounted Euler-Lagrange flow \eqref{lagrangian flow of DMFG}. Moreover, for any $x \in \A_{L_m, \lambda}$, the function $u_m^\lambda$ is differentiable at $x$.

Define the set of closed measures by
$$
\K ( \T^d \times \R^d ) := \left\{\mu \in \PP ( \T^d \times \R^d ) \, \middle\vert \,\, \int_{\T^d \times \R^d} \vert v \vert d \mu <+\infty, \,\, \int_{\T^d \times \R^d} v \cdot D\varphi(x) d\mu =0,\,\, \forall \varphi \in \CC^1 ( \T^d ) \right\}.
$$

\begin{definition} \label{lax-oleinik and solution discounted case}
	For any $\lambda \in (0, 1]$ and any $m \in \PP ( \T^d )$, we call $\mu \in \K ( \T^d \times \R^d )$ with $\supp (\mu) \subset \tilde{\A}_{L_m, \lambda}$ a minimizing $\lambda$-measure for $L_m$, if it satisfies
	$$
	\int_{\T^d \times \R^d} \left( L_m(x,v) - \lambda u_m^\lambda (x) \right) d\mu =0.
	$$	
\end{definition}
For any $\lambda \in (0, 1]$ and any $m \in \PP ( \T^d )$, Proposition \ref{existence of minimizing lambda measure} below establishes the existence of a minimizing $\lambda$-measure for Lagrangian $L_m$.

For any $\tau>0$ and any $m \in \PP ( \T^d )$, define the discrete action function by
$$
\LL_{\tau,m} (x,y) := \tau \left( L \left(x, \frac{y-x}{\tau} \right) + F(x,m) \right), \quad \forall x,y \in \R^d.
$$
For any $\tau \in (0,1)$, any $\lambda \in (0,1]$ and any $m \in \PP ( \T^d )$, define the discrete Lax-Oleinik equation for (\ref{DMFG}a) by
\begin{equation} \label{discrete_laxoleinik}
		u_{\tau,m}^\lambda (y) = \inf_{x \in \R^d} \left((1-\tau \lambda) u_{\tau,m}^\lambda(x) + \LL_{\tau,m} (x,y) \right) , \quad \forall y \in \R^d.		
\end{equation}
By Proposition \ref{existence of solution of discounted HJE} below, for any $\tau \in (0,1)$, any $\lambda \in (0,1]$ and any $m \in \PP ( \T^d )$, there exists a unique $\Z^d$-periodic continuous function $u_{\tau,m}^\lambda$ satisfies (\ref{discrete_laxoleinik}), and $\left\Vert u_{\tau,m}^\lambda \right\Vert_{\infty} \leq \frac{C_0}{\lambda}$, where
\begin{equation} \label{def of C_0}
	C_0 := \max \left\{ \sup_{\substack{\tau \in (0,1),\  x \in \R^d,\\ m \in \PP ( \T^d )}} \frac{\LL_{\tau,m}(x,x)}{\tau}, \,\,\,- \inf_{\substack{\tau \in (0,1),\ x,y\in \R^d, \\ m \in \PP ( \T^d )}} \frac{\LL_{\tau,m} (x,y)}{\tau} \right\}.
\end{equation}
Note that $C_0 < +\infty$ by (L\ref{MFG_L1}), (L\ref{MFG_L3}) and (F\ref{MFG_F1}).

By Lemma \ref{existence of calibrated configuration} below, for any $x \in \R^d$, any $m \in \PP ( \T^d )$, any $\tau \in (0,1)$ and any $\lambda \in (0,1]$, there exists $x_{-1} \in \R^d$ such that
$$
u_{\tau, m}^\lambda \left( x \right) = (1- \tau \lambda) u_{\tau,m}^\lambda (x_{-1}) + \LL_{\tau, m}(x_{-1} ,x).
$$
Moreover, there exists $x_{-2} \in \R^d$ such that $u_{\tau, m}^\lambda \left( x_{-1} \right) = (1- \tau \lambda) u_{\tau,m}^\lambda (x_{-2}) + \LL_{\tau, m}(x_{-2} ,x_{-1})$. Thus, one can inductively construct a sequence $\left\{x_{-k}\right\}_{k=0}^{+\infty}$ with $x_0 = x$ such that for any $k \geq 0$, 
$$
u_{\tau, m}^\lambda \left( x_{-k} \right) = (1- \tau \lambda) u_{\tau,m}^\lambda (x_{-k-1}) + \LL_{\tau, m}(x_{-k-1} ,x_{-k}).
$$
We call such a sequence a \textit{discounted calibrated configuration} of $u_{\tau, m}^\lambda (x)$.

%\medskip
\medskip

For any $x = \left( x_1, \cdots, x_d \right) \in \R^d$, define a standard universal covering projection $\Pi: \R^d \rightarrow \T^d$ by
$$
\Pi (x) = \left( x_1 \,\, \mmod\,\,1, \,\,x_2 \,\, \mmod\,\,1,\,\cdots, \,x_d \,\, \mmod\,\,1 \right).
$$
For any $\tau \in (0,1)$, any $\lambda \in (0,1]$ and any $m \in \PP ( \T^d )$, define
\begin{align*}
	\tilde{\Sigma}_{L_m, \lambda}^\tau := \left\{ ([x],v) \in \T^d \times \R^d \middle\vert \right. 
	& \,\,\exists x\in \R^d \,\,s.t.\,\, \Pi(x) = [x],\,\, \exists\,\, \text{a discounted calibrated configuration}   \\
	& \left. \left\{x_{-k}\right\}_{k=0}^{+\infty}\,\,\text{of}\,\, u_{\tau,m}^\lambda (x +\tau v)\,\, \text{satisfying}\,\,x_{-1} =x, \,\, \frac{x_0 - x_{-1}}{\tau} =v \right\}.
\end{align*}
%By the regularity of $u^\lambda_{\tau,m}$ and $L_m$, it is clear that $\tilde{\Sigma}_{L_m, \lambda}^\tau$ is non-empty and compact.
For any $n \in \N$ and any $([x],v) \in \T^d \times \R^d$, we define a map $\Psi^n_{L_m, \lambda, \tau}: \T^d \times \R^d \rightarrow \T^d \times \R^d$ by
\begin{align*}
	\Psi^n_{L_m, \lambda, \tau} ([x], v):= \left\{ \left( [x_{-n-1}], \frac{x_{-n} - x_{-n-1}}{\tau} \right) \middle\vert  \right. 
	& \,\, \exists \,\,x\in \R^d \,\,s.t.\,\, \Pi(x) = [x],\,\,  \\
	& \exists\,\, \text{a discounted calibrated configuration}\,\,\left\{x_{-k}\right\}_{k=0}^{+\infty} \\
	& \left. \text{of}\,\, u_{\tau,m}^\lambda (x +\tau v)\,\, \text{satisfying}\,\, x_{-1} =x,\,\, v = \frac{x_0 -x_{-1}}{\tau} \right\}.
\end{align*}
%It is clear that for any $n \in \N$, $\Psi^n_{L_m, \lambda, \tau} \left( \tilde{\Sigma}_{L_m, \lambda}^\tau \right) \subset \tilde{\Sigma}_{L_m, \lambda}^\tau$. More generally, for any $n_1 > n_2$,
%$$
%\Psi^{n_1}_{L_m, \lambda, \tau} \left( \tilde{\Sigma}_{L_m, \lambda}^\tau \right) \subset \Psi^{n_2}_{L_m, \lambda, \tau} \left( \tilde{\Sigma}_{L_m, \lambda}^\tau \right).
%$$
Thus we can define the discrete version of the discounted Aubry set as follows.
\begin{definition} \label{def of discrete aubry set}
	For any $\tau \in (0,1)$, any $\lambda \in (0,1]$ and any $m \in \PP ( \T^d )$, define the discrete discounted Aubry set by
	$$
	\tilde{\A}_{L_m, \lambda}^\tau := \bigcap_{n \in \N} \Psi^n_{L_m, \lambda, \tau} \left( \tilde{\Sigma}_{L_m, \lambda}^\tau \right).
	$$
\end{definition}
By Lemma \ref{compact of discrete aubry set}, the set $\tilde{\A}_{L_m, \lambda}^\tau$ is compact for any $\tau \in (0,1)$, any $\lambda \in (0,1]$ and any $m \in \PP ( \T^d )$.
%Inspired from the definition of discounted Aubry set, one can define the discrete discounted Aubry set $\tilde{\A}_{L_m, \lambda}^\tau$ as in Definition \ref{def of discrete aubry set} below. 
For any $\tau>0$, define the set of $\tau$-holonomic measures by
$$
\PP_\tau ( \T^d \times \R^d ) := \left\{ \mu \in \PP ( \T^d \times \R^d ) \,\, \middle\vert \,\, \int_{\T^d \times \R^d} \varphi(x+\tau v)d\mu = \int_{\T^d \times \R^d} \varphi(x) d\mu,\,\, \forall \varphi \in \CC ( \T^d ) \right\}.
$$

\begin{definition} \label{def of minimizng holonomic measure}
For any $\tau \in (0,1)$, any $\lambda \in (0,1]$ and any $m \in \PP ( \T^d )$, we call $\mu \in \PP_\tau ( \T^d \times \R^d )$ with $\supp(\mu) \subset \tilde{\A}_{L_m, \lambda}^\tau$ a minimizing $\tau$-holonomic $\lambda$-measure for $L_m$, if it satisfies
$$
\int_{\T^d \times \R^d} \left( L_m(x,v) - \lambda u_{\tau,m}^\lambda (x) \right) d\mu =0.
$$
\end{definition}
For any $\tau \in (0,1)$, any $\lambda \in (0,1]$ and any $m \in \PP ( \T^d )$, Proposition \ref{non-empty of minimizing holonomic measure} below establishes the existence of a minimizing $\tau$-holonomic $\lambda$-measure for Lagrangian $L_m$.

The first main result of this paper is stated as follows.
\begin{theorem} \label{theorem 1}
	Assume (L\ref{MFG_L1})-(L\ref{MFG_L3}) and (F\ref{MFG_F1})-(F\ref{MFG_F3}). Then we have the following:
	\begin{enumerate}[(1)]
		\item Fix $\lambda \in (0,1]$. \begin{enumerate}[(i)]
		\item For any $\tau \in (0, 1)$, there is $m \in \PP ( \T^d )$ such that there exists a minimizing $\tau$-holonomic $\lambda$-measure $\mu_{\tau,m}^\lambda$ for the Lagrangian $L_m$ with
		$$
		m = \pi \sharp \mu_{\tau,m}^\lambda,
		$$
		where the push-forward of $\mu_{\tau,m}^\lambda$ through the canonical projection $\pi:\T^d\times\mathbb{R}^d\to\T^d$ is defined by 
		$$
		\pi \sharp \mu_{\tau,m}^\lambda (B):= \mu_{\tau,m}^\lambda \left( \pi^{-1} (B) \right), \quad \text{for any Borel subset $B$ of} \,\,\, \T^d.
		$$
		Such a measure $m$ is called a minimizing $\tau$-holonomic $\lambda$-measure for DMFGs \eqref{DMFG} and denote it by $m_\tau^\lambda$.
		
		\item There is a subsequence $\tau_i \rightarrow 0$, a subsequence $m_{\tau_i}^\lambda \stackrel{w^*}{\longrightarrow} m_0^\lambda$, and a subsequence $u_{\tau_i, m_{\tau_i}^\lambda}^\lambda$ solutions to \eqref{discrete_laxoleinik} such that $u_{\tau_i, m_{\tau_i}^\lambda}^\lambda$ converges to $u_0^\lambda$ uniformly and $(u_0^\lambda, m_0^\lambda)$ is a solution to DMFGs \eqref{DMFG}.
		\end{enumerate}
		
		\item There exists a sequence $\lambda_j \rightarrow 0$ such that $m_0^{\lambda_j} \stackrel{w^*}{\longrightarrow} m_0$ and $u_0^{\lambda_j} + \frac{c \left( m_0^{\lambda_j} \right)}{\lambda_j} \rightarrow u_0$ uniformly, and $(u_0,m_0)$ is a solution to MFGs \eqref{MFG}. 
	\end{enumerate}
\end{theorem}

\medskip
\medskip
\begin{remarks} \rm \label{remark 1.2}
\begin{enumerate}[(i)]
%\item Fix Tonelli Hamiltonian $H_m$ and associated Lagrangian $L_m$. %The Hamiltonian flow of $H$, denoted by $\Phi^t_{H}$, is defined by
%\begin{equation} \tag{HS} \label{hamiltonian flow of MFG}
%	\begin{cases}
%		& \dot{x} = \frac{\partial H}{\partial p}, \\ 
%		& \dot{p} = - \frac{\partial H}{\partial x}.
%	\end{cases}
%\end{equation}
%The Euler-Lagrange flow of $L_m$, denoted by $\Phi^t_{L_m}$, is defined by
%\begin{equation} \tag{EL} \label{lagrangian flow of MFG}
%	\begin{cases}
%		& \dot{x} = v, \\ 
%		& \frac{d}{dt} \left( \frac{\partial L_m}{\partial v} \right) =  \frac{\partial L_m}{\partial x}.
%	\end{cases}
%\end{equation}

%\item Denote by $\Phi^t_{L}$ the Euler-Lagrange flow of the Lagrangian $L$ defined as in \eqref{lagrangian flow of MFG}, by $\Phi^t_{L, \lambda}$ the discounted Euler-Lagrange flow of the Lagrangian $L$ defined as in \eqref{lagrangian flow of DMFG}, by $\Phi^t_{L_m}$ and $\Phi^t_{L_m, \lambda}$ the Euler-Lagrange flow and the discounted Euler-Lagrange flow of the Lagrangian $L_m$ respectively.
	
\item We use $c(m)$ to denote the Ma\~n\'e critical value of $L_m$ \cite{MR1479499}. %{\color{red} R. Mane, Lagrangian flows: The dynamics of globally minimizing orbits, Bol. Soc. Brasil. Mat. (N.S.),
%28 (1997), pp. 141--153.}
It is well known that for any given $m \in \PP ( \T^d )$, $c(m)$ is the unique real number $k$ such that the equation $H_m \left( x, Du(x) \right) = k$ has viscosity solutions. A measure $\mu \in \PP ( \T^d \times \R^d )$ is called a Mather measure for $L_m$ if it satisfies
$$
\int_{\T^d \times \R^d} L_m(x,v) d \mu = \min_\mu \int_{\T^d \times \R^d} L_m(x,v) d \mu = -c(m),
$$
where the minimum is taken over all Borel probability measures on $\T^d \times \R^d$ invariant under the Euler-Lagrange flow $\Phi^t_{L_m}$ generated by $L_m$ as defined in \eqref{lagrangian flow of MFG} below. 
Besides, we have the conclusion that a closed measure $\mu$ satisfying
$$
\int_{\T^d \times \R^d} L_m(x,v) d \mu = -c(m)
$$
is a Mather measure.

%\item For any solution $(u^\lambda, m^\lambda)$ of DMFGs \eqref{DMFG}, there exists a sequence $\lambda_j \rightarrow 0$ such that $m^{\lambda_j} \stackrel{w^*}{\longrightarrow} m$ and $u^{\lambda_j} + \frac{c \left( m^{\lambda_j} \right)}{\lambda_j} \rightarrow u$ uniformly for some measure $m \in \PP ( \T^d )$ and Lipschitz function $u$.
%It is clear that $u$ is a viscosity solution to
%$$
%H_{m} (x, Du) = - c(m),\quad \forall x \in \T^d.
%$$ 
%But the measure $m$ maynot satisfies the continuity equation (\ref{MFG}b). %since we cannot construct a Mather measure $\mu$ such that $m = \pi \sharp \mu$.

\item 
In \cite{MR4455763}, Hu and Wang showed the 
existence of solutions to
\begin{equation*}
	\begin{cases}
		H_1(x,u, Du)= F_1(x,m), & x \in \T^d
		\\
		\ddiv \left( m \frac{\partial H_1}{\partial p}(x,u, Du) \right) = 0, & x \in \T^d 
		\\
		\int_{\T^d} m \, dx =1. 
	\end{cases}
\end{equation*}
There, besides satisfying (H\ref{MFG_H1})-(H\ref{MFG_H3}), the Hamiltonian $H_1$ is also required to be strictly monotonically increasing with respect to $u$ and reversible with respect to $p$, i.e., $H_1(x,u,p)=H_1(x,u,-p)$ for all $(x,u,p)\in\mathbb{T}^d\times\mathbb{R}\times\mathbb{R}^d$. By Theorem \ref{theorem 1}, we get the existence of solutions to DMFGs \eqref{DMFG} under (H\ref{MFG_H1})-(H\ref{MFG_H3}) without the reversibility condition.
\item In \cite{MR3556524}, a full convergence result is provided for vanishing discount  problem for Hamilton-Jacobi equations. One might naturally wonder why, for the DMFGs \eqref{DMFG}, we only established subsequential convergence (Theorem \ref{theorem 1} (2)) rather than full convergence. Based on our current understanding of the problem,
the lack of uniqueness of solutions of \eqref{DMFG} appears to be a critical factor.
An example demonstrating non-unique solutions will be provided in  Appendix B.

\end{enumerate}
\end{remarks}

Before stating the second main result, we first recall the discretization framework and key conclusions in \cite{MR4605206}.
A $\Z^d$-periodic function $u \in \CC ( \R^d )$ is a viscosity solution to (\ref{MFG}a) if and only if $u$ satisfies the Lax-Oleinik equation:
\begin{equation} \label{classical_laxoleinik}
	u(y) - c(m)t = \inf_{x \in \R^d} \left( u(x) + h_t^m (x,y) \right), \quad \forall y \in \R^d, \,\,\forall t >0,
\end{equation}
where the minimal action $h_t^m (x,y)$ is defined by
$$
h_t^m (x,y) := \inf_{\gamma} \int_0^t L_m \left( \gamma(s), \dot{\gamma}(s) \right) ds,
$$
where the infimum is taken over all absolutely continuous curves $\gamma: [0,t] \rightarrow \R^d$ with $\gamma(0) = x$ and $\gamma(t) = y$.

For each $\tau >0$ and each $m \in \PP ( \T^d )$, there is a unique constant $\bar{L} (\tau,m) \in \R$ such that the discrete Lax-Oleinik equation for (\ref{MFG}a)
	\begin{equation} \label{discrete_laxoleinik_classical}
		u_{\tau,m} (y) + \tau \bar{L} (\tau,m) = \inf_{x \in \R^d} \left( u_{\tau,m} (x) + \LL_{\tau,m} (x,y)\right) , \quad \forall y \in \R^d
	\end{equation}
	has a continuous $\Z^d$-periodic solution $u_{\tau,m}$ and $\bar{L} (\tau,m) \rightarrow - c(m)$ as $\tau \rightarrow 0$. More precisely, 
	\begin{equation} \label{need0805}
	\bar{L} (\tau,m) = \min_\mu \int_{\T^d \times \R^d} L_m (x,v) d \mu,
	\end{equation}
	where the minimum is taken over all $\tau$-holonomic measures $\PP_\tau ( \T^d \times \R^d )$. A measure $\mu$ attaining the minimum is called a minimizing $\tau$-holonomic measure for $L_m$.

There exists a constant $\tau_0>0$ such that for each $\tau \in (0,\tau_0)$, there is $m \in \PP ( \T^d )$ such that there exists a minimizing $\tau$-holonomic measure $\mu_{\tau,m}$ for the Lagrangian $L_m$ with $m = \pi \sharp \mu_{\tau,m}$. Such a measure $m$ is denoted by $m_\tau$, and we call $m_\tau$ a minimizing $\tau$-holonomic measure for MFGs \eqref{MFG}.

\medskip
\medskip

The second main result of this paper is the following. We point out that Theorem \ref{theorem 2} (2) is a direct consequence of \cite[Theorem 1 (2)]{MR4605206}.
\begin{theorem} \label{theorem 2}
Assume (L\ref{MFG_L1})-(L\ref{MFG_L3}) and (F\ref{MFG_F1})-(F\ref{MFG_F3}). Then we have the following:
\begin{enumerate} [(1)]
\item For any $\tau \in (0,\tau_0)$, there exists a sequence $\lambda_i \rightarrow 0$ such that $m_\tau^{\lambda_i} \stackrel{w^*}{\longrightarrow} m_0^\tau$ and $u^{\lambda_i}_{\tau, m_\tau^{\lambda_i}} - \frac{\bar{L} \left( \tau, m_\tau^{\lambda_i} \right)}{\lambda_i}$ converges to $u_0^{\tau}$ uniformly, where $m_0^\tau$ is a minimizing $\tau$-holonomic measure for MFGs \eqref{MFG} and $u_0^\tau$ is a solution to \eqref{discrete_laxoleinik_classical} with respect to the measure $m_0^\tau$.

\item There exists a sequence $\tau_j \rightarrow 0$ such that $m_0^{\tau_j} \stackrel{w^*}{\longrightarrow} m_0$, $u_0^{\tau_j}$ converges to $u_0$ uniformly, and $(u_0, m_0)$ is a solution to MFGs \eqref{MFG}.
\end{enumerate}
\end{theorem}

The rest of this paper is organized as follows. Sect. \ref{priliminaries} is devoted to preliminaries. Sect. \ref{Section 3} is devoted to show the time discretization of DMFGs \eqref{DMFG}. We prove Theorem \ref{theorem 1} and Theorem \ref{theorem 2} in Sect. \ref{section 4} and Sect. \ref{section5} respectively. We leave the proof of Lemma \ref{convergence of aubry set} in Appendix A since it is quite long, and give a specific example of DMFGs \eqref{DMFG} where the solution is not unique in Appendix B.

\section{Priliminaries} \label{priliminaries}
Fix $\lambda >0$. In Sect. \ref{viscosity solutions}-\ref{aubry set and minimizing measures for DMFG}, we consider the discounted Hamilton-Jacobi equation
\begin{equation} \label{HJE_1}
	\lambda u + H(x,Du) = 0, \quad \forall x \in \T^d.
\end{equation}
Recall that $L$ denotes the associated Lagrangian.

\subsection{Viscosity solutions} \label{viscosity solutions}
\begin{definition}[Viscosity solutions] \label{viscosity solution}
Let $U$ be an open subset of $\T^d$.
\begin{itemize}
\item A function $u: U \rightarrow \R$ is called a viscosity subsolution to (\ref{HJE_1}) if for every $\CC^1$ function $\varphi : U \rightarrow \R$ and every point $x_0 \in U$ such that $u - \varphi$ attains a local maximum at $x_0$, we have
$$
\lambda u(x_0) + H \left( x_0, D \varphi(x_0) \right) \leq 0.
$$

\item A function $u: U \rightarrow \R$ is called a viscosity supersolution to (\ref{HJE_1}) if for every $\CC^1$ function $\psi : U \rightarrow \R$ and every point $y_0 \in U$ such that $u - \psi$ attains a local minimum at $y_0$, we have
$$
\lambda u(y_0) + H \left( y_0, D \psi(y_0) \right) \geq 0.
$$

\item A function $u: U \rightarrow \R$ is called a viscosity solution to (\ref{HJE_1}) if it is both a viscosity subsolution and a viscosity supersolution.
\end{itemize}
\end{definition}

\begin{proposition} \cite[Proposition 2.6, Proposition 3.5, Theorem 3.8]{MR3556524}
Equation (\ref{HJE_1}) admits a unique viscosity solution
\begin{equation} \label{representation of viscosity solution}
u_\lambda (x) = \inf_{\gamma} \int_{-\infty}^0 e^{\lambda s} L \left( \gamma(s), \dot{\gamma}(s) \right) ds, \quad \forall x \in \T^d,
\end{equation}
where the infimum is taken over all absolutely continuous curves $\gamma: \left(-\infty, 0 \right] \rightarrow \T^d$ with $\gamma(0) =x$.

Moreover, the solutions $\left\{ u_\lambda \, \middle \vert \,\, \lambda>0 \right\}$ are equi-Lipschitz and equi-bounded.
There exists a function $u_0$ such that $u_\lambda$ converges to $u_0$ as $\lambda \rightarrow 0$ on $\T^d$ uniformly, and $u_0$ is a viscosity solution to the Hamilton-Jacobi equation, i.e.,
\begin{equation} \label{classical_HJE}
H \left( x, D u \right) =0, \quad \forall x \in \T^d.
\end{equation}
\end{proposition}

\subsection{Weak KAM solutions} \label{weak KAM solutions}
This section briefly introduces several key definitions from weak KAM theory for \eqref{HJE_1}. See \cite{MR3663623} for details.
\begin{definition}
Let $U$ be an open set of $\T^d$. A function $u: U \rightarrow \R$ is $\lambda$-dominated by $L$ and denoted by $u \prec_{\lambda} L$ if for any absolutely continuous curve $\gamma: [a,b] \rightarrow U$ with $-\infty < a < b < +\infty$, there holds
$$
e^{\lambda b} u \left( \gamma(b) \right) - e^{\lambda a} u \left( \gamma(a) \right) \leq \int_a^b e^{\lambda t} L \left( \gamma(t), \dot{\gamma}(t) \right) dt.
$$
\end{definition}

\begin{definition}
An absolutely continuous curve $\gamma : I \rightarrow \T^d$ defined on the interval $I \subset \R$ is called $(u,L)$-$\lambda$-calibrated curve, if for any $t, t^\prime \in I$,
$$
e^{\lambda t} u \left( \gamma(t) \right) - e^{\lambda t^\prime} u \left( \gamma(t^\prime) \right) = \int_{t^\prime}^t e^{\lambda s} L \left( \gamma(s), \dot{\gamma}(s) \right) ds.
$$
\end{definition}

\begin{definition} [Weak KAM solution]
	Let $U$ be an open set of $\T^d$. A function $u: U \rightarrow \R$ is a weak KAM solution to \eqref{HJE_1} if
	\begin{enumerate} [(i)]
		\item $u \prec_\lambda L$,
		\item for any $x \in U$, there exists $\gamma: (-\infty , 0] \rightarrow U$ with $\gamma(0)=x$, which is $(u,L)$-$\lambda$-calibrated.
	\end{enumerate}
\end{definition}

\subsection{Lax-Oleinik semi-groups} \label{lax-oleinik semigroup}
\begin{definition}
Let $\gamma: [a,b] \rightarrow \T^d$ be an absolutely continuous curve with $- \infty <a<b< + \infty$. Define its discounted action by
$$
A_{L,\lambda}(\gamma) := \int_a^b e^{\lambda t} L \left( \gamma(t), \dot{\gamma} (t) \right) dt.
$$
We say that an absolutely continuous curve $\gamma: [a,b] \rightarrow \T^d$ is a minimizing curve for $A_{L,\lambda}$, if for any absolutely continuous curve $\sigma: [a,b] \rightarrow \T^d$ with $\gamma(a) = \sigma(a)$ and $\gamma(b) = \sigma(b)$,
$$
A_{L,\lambda}(\gamma) \leq A_{L,\lambda}(\sigma).
$$
\end{definition}
By \cite[Remark 4]{MR3663623}, every minimizing curve for $A_{L,\lambda}$ is of class $\CC^2$.

\begin{definition}
For any $t \geq 0$, the minimal action $h_{t,\lambda}(x,y)$ is defined by
$$
h_{t,\lambda}(x,y) := \inf_{\gamma} \int_{-t}^0 e^{\lambda s} L \left( \gamma(s), \dot{\gamma}(s) \right) ds, \quad \forall x,y \in \T^d,
$$
where the infimum is taken over all the absolutely continuous curves $\gamma : [-t,0]  \rightarrow \T^d$ with $\gamma(-t) =x$ and $\gamma(0)=y$.

Define the Lax-Oleinik semigroup $\left\{ T_{t, \lambda} : \CC ( \T^d ) \rightarrow \CC ( \T^d ) \middle \vert \,\, t \geq 0 \right\}$ by $T_{0, \lambda} \varphi = \varphi$ for any $\varphi \in \CC ( \T^d )$ and
$$
T_{t, \lambda} \varphi(x) := \min_{y \in \T^d} \left( e^{- \lambda t} \varphi(y) + h_{t,\lambda}(y,x) \right) ,\quad \forall t > 0,\,\, \forall x \in \T^d,\,\, \forall \varphi \in \CC ( \T^d ).
$$
\end{definition}

\begin{proposition} \cite[Lemma 4.1, Proposition 4.2, 4.3]{MR3912650}
For any $s, t \geq 0$ and any $u,v \in \CC ( \T^d )$,
\begin{enumerate} [(i)]
\item $T_{t+s, \lambda} = T_{t, \lambda} \circ T_{s, \lambda}$;
\item if $u \leq v$, then $T_{t, \lambda} u \leq T_{t, \lambda} v$;
\item $\left \Vert T_{t, \lambda} u - T_{t, \lambda} v \right\Vert_\infty \leq e^{- \lambda t} \left\Vert u-v \right\Vert_\infty$.
\end{enumerate}
\end{proposition}

The next Proposition reveals the relations between viscosity solutions, weak KAM solutions and Lax-Oleinik operators. See, for instance, \cite{MR3927084}.
\begin{proposition}
	Let $U$ be an open set of $\T^d$. A function $u: U \rightarrow \R$ is a weak KAM solution to \eqref{HJE_1} if and only if it is a viscosity solution to \eqref{HJE_1}, and if and only if $u$ is a fixed point of the Lax-Oleinik operators, i.e., 
	$$
	T_{t, \lambda} u = u , \quad \forall t \geq 0.
	$$
\end{proposition}

\subsection{Discounted Aubry sets and minimizing $\lambda$-measures} \label{aubry set and minimizing measures for DMFG}

In contrast to \cite{MR3556524}, we propose a more precise definition of the discounted Mather measure, ensuring that it converges to the classical Mather measure, which is defined as in Remark \ref{remark 1.2} (i), as the discounted rate $\lambda$ tends to 0. Furthermore, compared with \cite{MR3663623}, our definition of the discounted Mather measure is based exclusively on closed measures and does not involve the discounted Euler-Lagrange flow. This choice significantly simplifies the discretization of the discounted Mather measure later.

We first introduce both the Euler-Lagrange flow generated by $L$ and its discounted counterpart.
%Before defining the discounted Aubry sets, we first introduce both the Euler-Lagrange flow generated by $L$ and its discounted counterpart.
Denote by $\Phi_L^t$ the flow of
\begin{equation} \tag{EL} \label{lagrangian flow of MFG}
	\begin{cases}
			& \dot{x} = v, \\ 
			& \frac{d}{dt} \left( \frac{\partial L}{\partial v} \right) =  \frac{\partial L}{\partial x}.
		\end{cases}
\end{equation}
We call it the Euler-Lagrange flow generated by $L$.
Denote by $\Phi^t_{L,\lambda}$ the flow of
\begin{equation} \tag{EL$_\lambda$} \label{lagrangian flow of DMFG}
	\begin{cases}
		& \dot{x} = v, \\ 
		& \frac{d}{dt} \left( e^{\lambda t} \frac{\partial L}{\partial v} \right) = e^{\lambda t} \frac{\partial L}{\partial x}.
	\end{cases}
\end{equation}
We call it the discounted Euler-Lagrange flow generated by $L$.

%For every $(x,v) \in \T^d \times \R^d$, let $\gamma_{(x,v)} (t) := \pi \left( \Phi^t_{L, \lambda} (x,v) \right)$ for all $t \in \R$.
%Define the set
%$$
%\tilde{\Sigma}_{L, \lambda} := \left\{ (x,v) \in \T^d \times \R^d \,\, \middle\vert\,\, \text{the curve $\gamma_{(x,v)}$ is $(u_\lambda , L)$-$\lambda$-calibrated on $(-\infty, 0]$}  \right\}.
%$$
%It is clear that $\tilde{\Sigma}_{L, \lambda}$ is non-empty, bounded, and for any $t>0$, the set $\Phi^{-t}_{L, \lambda} \left( \tilde{\Sigma}_{L, \lambda} \right)$ is compact.
%Define the discounted Aubry set by
%$$
%\tilde{\A}_{L, \lambda} := \bigcap_{t \geq 0} \Phi^{-t}_{L, \lambda} \left( \tilde{\Sigma}_{L, \lambda} \right).
%$$
%Denote by $\A_{L, \lambda} = \pi (\tilde{\A}_{L, \lambda})$ the projected discounted Aubry set. Then, $\tilde{\A}_{L, \lambda}$ is non-empty, compact and invariant under the Euler-Lagrange flow \eqref{lagrangian flow of DMFG}. Moreover, for any $x \in \A_{L, \lambda}$, the function $u_\lambda$ is differentiable at $x$.

Inspired by \cite{MR3663623}, we define minimizing $\lambda$-measures (discounted Mather measures) with the help of closed measures and discounted Aubry sets, without relying on the Euler-Lagrange flow.
First of all, we establish in the following proposition that for any closed measure $\mu$, the functional $\int_{\T^d \times \R^d} \left( L (x,v) - \lambda u_\lambda (x) \right) d\mu$ is non-negative.

\begin{proposition} \label{need0818}
	For any $\mu \in \K ( \T^d \times \R^d )$, there is
	$$
	\int_{\T^d \times \R^d} \left(  L (x,v) - \lambda u_\lambda (x) \right)  d\mu \geq 0.
	$$
\end{proposition}
\begin{proof}
	By \cite[Lemma 2.2]{MR3556524}, for any $\varepsilon >0$, there exists $u_\varepsilon \in \CC^\infty ( \T^d )$ such that $\left\Vert u_\lambda - u_\varepsilon \right\Vert_{\infty} \leq \varepsilon$ and
	$$
	\lambda u_\lambda (x) + H \left( x, D u_\varepsilon (x) \right) \leq \varepsilon, \quad \forall x \in \T^d.
	$$
	For any $(x,v) \in \T^d \times \R^d$, Fenchel's inequality implies that
	$$
	v \cdot Du_\varepsilon (x) \leq L (x,v)+H (x, Du_\varepsilon (x)) = L (x,v) - \lambda u_\lambda(x) + \lambda u_\lambda(x) + H(x, Du_\varepsilon (x)).
	$$
	Integrating both sides with respect to $\mu$, we have
	$$
	\int_{\T^d \times \R^d} \left( L(x,v) - \lambda u_\lambda (x) \right)  d \mu \geq  \int_{\T^d \times \R^d} v \cdot Du_\varepsilon (x) d\mu - \int_{\T^d \times \R^d} \left( \lambda u_\lambda(x) +H(x, Du_\varepsilon (x)) \right)  d\mu \geq -\varepsilon.
	$$
	Let $\varepsilon$ tend to 0. The proof is finished.
\end{proof}

The following two lemmas are utilized in the proof of Proposition \ref{existence of minimizing lambda measure}, establishing the existence of a closed measure $\mu$ such that the functional $\int_{\T^d \times \R^d}\left(  L(x,v) - \lambda u_\lambda (x) \right) d\mu$ attains the minimum and the support of it belongs the discounted Aubry set.
\begin{lemma} \label{lemma1,1}
	For any $t>0$, there exists a constant $C_t < + \infty$ such that for any $x,y \in \T^d$, there is a $\CC^\infty$ curve $\gamma: [-t, 0] \rightarrow \T^d$ with $\gamma(-t) = x$, $\gamma(0)=y$ satisfying that
	$$
	A_{L, \lambda} \left( \gamma \right) = \int_{-t}^0 e^{\lambda s} L \left( \gamma(s), \dot{\gamma}(s) \right) ds \leq C_t.
	$$
\end{lemma}
\begin{proof}
	Define $\gamma: [-t, 0] \rightarrow \T^d$ by a segment connecting $x$ and $y$ with $\vert \dot{\gamma} \vert$ is a constant.
	Then for any $s \in  [-t,0]$, we have $\left\vert \dot{\gamma}(s) \right\vert \leq \frac{\diam ( \T^d )}{t}$.
	Define a compact set $A_t$ by
	$$
	A_t := \left\{ (x,v) \in \T^d \times \R^d \middle\vert \,\, \vert v \vert \leq \frac{\diam ( \T^d )}{t} \right\},
	$$
	which contains the graph of $\gamma$. By the compactness of $A_t$, there exists a constant $\tilde{C}_t >0$ such that 
	$$
	L(x,v) \leq \tilde{C}_t \leq \frac{\tilde{C}_t}{e^{\lambda s}}, \quad \forall (x,v) \in A_t,\,\, \forall s \in [-t,0].
	$$
	Let $C_t := t \tilde{C}_t$. Then $A_{L, \lambda} \left( \gamma \right) \leq C_t$.
\end{proof}

\begin{lemma} \label{lemma1,2}
	Fix $t>0$. There exists a compact set $K_t \subset \T^d \times \R^d$ such that for any minimizing curve $\gamma: [a, b] \rightarrow \T^d$ for $A_{L, \lambda}$ with $b-a>t$, we have
	$$
	\left( \gamma(s), \dot{\gamma}(s) \right) \in K_t, \quad \forall s \in [a, b ].
	$$
\end{lemma}
\begin{proof}
	We only need to consider the case that $[a,b]=[-t,0]$, since for any $t_0 \in [a,b]$, one can always find an interval of length $t$ containing $t_0$.
	Without loss of generality, assume that $L\geq 0$. Define the constant $C_t$ as in Lemma \ref{lemma1,1}. Since $\gamma$ is a minimizing curve for $A_{L, \lambda}$, we obtain that $A_{L, \lambda} \left( \gamma \right) \leq C_t$. Since the function $s \mapsto e^{\lambda s} L \left( \gamma(s), \dot{\gamma}(s) \right)$ is continuous, by the mean value theorem, there exists $s_0 \in [-t,0]$ such that
	$$
	e^{\lambda s_0} L \left( \gamma(s_0), \dot{\gamma}(s_0) \right) \leq \frac{C_t}{t}.
	$$
	Define a compact set $B$ by
	$$
	B := \left\{ (x,v) \in \T^d \times \R^d \middle\vert \,\, L \left( x,v \right) \leq \frac{C_t}{t e^{-t}} \right\}.
	$$
	By the continuity of the Lagrangian flow $\Phi^t_{L, \lambda}$, the set $K_t := \bigcup_{\vert s \vert \leq t} \Phi^s_{L, \lambda} (B)$ is compact as well. Moreover, for any $s \in [-t, 0 ]$,
	$$
	\left( \gamma(s), \dot{\gamma}(s) \right) \in \Phi^{s- s_0}_{L, \lambda} (B) \subset K_t.
	$$
\end{proof}

\begin{proposition} \label{existence of minimizing lambda measure}
	There exists $\mu \in \K ( \T^d \times \R^d )$ such that $\supp (\mu) \subset \tilde{\A}_{L, \lambda}$ and
	$$
	\int_{\T^d \times \R^d} \left( L(x,v) - \lambda u_\lambda (x) \right)  d\mu = 0,
	$$
	where $\tilde{\A}_{L, \lambda}$ is defined as in \eqref{discounted aubry set}.
\end{proposition}
\begin{proof}
	By the definition of $u_\lambda$, for any $(x,v) \in \tilde{\A}_{L,\lambda}$, there exists a curve $\gamma_0 : (-\infty, 0] \rightarrow \T^d$ with $\gamma_0(0) = x$ and $\dot{\gamma}(0) = v$ such that $\gamma_0$ is $\left( u_\lambda, L \right)$-$\lambda$-calibrated. More precisely, for any $t>0$,
	$$
	u_\lambda \left( \gamma_0(0) \right) - e^{-\lambda t} u_\lambda \left( \gamma_0(-t) \right) = \int_{-t}^0 e^{\lambda s} L \left( \gamma_0(s), \dot{\gamma}_0(s) \right) ds.
	$$
	Since $\gamma_0$ is a minimizing curve for action, we have
	$$
	\Phi^s_{L, \lambda} \left( \gamma_0(0), \dot{\gamma}_0(0) \right) = \left( \gamma_0(s), \dot{\gamma}_0(s) \right), \quad \forall s 
	\leq 0.
	$$
	
	For $t \geq 1$, define a Borel probability measure $\mu_t$ on $\T^d \times \R^d$ by
	$$
	\mu_t \left( \theta \right) := \frac{1}{t} \int_{-t}^0 \theta \left( \Phi^s_{L,\lambda} \left( x_0, v_0 \right) \right) ds,
	$$
	for all continuous functions $\theta : \T^d \times \R^d \rightarrow \R$.
	By Lemma \ref{lemma1,2}, there exists a compact set $K_1$ such that $\supp (\mu_t) \subset K_1$ for any $t \geq 1$. Thus, there exists a sequence $\left\{ t_n \right\}_{n=1}^\infty$ with $\lim_{n \rightarrow \infty} t_n = +\infty$ such that $\mu_{t_n} \stackrel{w^*}{\longrightarrow} \mu$ for some measure $\mu$. The measure $\mu$ is compactly supported and invariant under the Lagrangian flow $\Phi^t_{L,\lambda}$ obviously.
	Since $\gamma_0$ is absolutely continuous, it is clear that
	$$
	\int_{\T^d \times \R^d} \vert v \vert d \mu = \lim_{n \rightarrow \infty} \frac{1}{t_n} \int_{-t_n}^0 \left\vert \dot{\gamma}_0(s) \right\vert ds = 0.
	$$
	Moreover, for any $\varphi \in \CC^1 ( \T^d )$,
	\begin{align*}
		\left\vert \int_{\T^d \times \R^d} v \cdot D\varphi(x) d\mu \right\vert
		& = \left\vert \lim_{n \rightarrow \infty} \frac{1}{t_n} \int_{-t_n}^0 \dot{\gamma}_0(s) \cdot D \varphi \left( \gamma_0(s) \right) ds \right\vert \\
		& = \left\vert \lim_{n \rightarrow \infty} \frac{\varphi \left( \gamma_0(0) \right) - \varphi \left( \gamma_0(- t_n) \right)}{t_n} \right\vert \\
		& \leq \lim_{n \rightarrow \infty} \frac{2 \left\Vert \varphi \right \Vert_\infty}{t_n} = 0.
	\end{align*}
	Thus, $\mu \in \K ( \T^d \times \R^d )$.
	Furthermore, since $\mu$ is invariant under \eqref{lagrangian flow of DMFG},
	\begin{align*}
		\int_{\T^d \times \R^d} \lambda u_\lambda (x) d\mu
		& = \lambda \int_{\T^d \times \R^d} \int_{-\infty}^0 e^{\lambda s} L \left( \Phi^s_{L, \lambda} (x,v) \right) ds d\mu \\
		& = \lambda \int_{-\infty}^0 e^{\lambda s} \int_{\T^d \times \R^d} L \left( \Phi^s_{L, \lambda} (x,v) \right) d\mu ds \\
		& = \lambda \int_{-\infty}^0 e^{\lambda s} \int_{\T^d \times \R^d} L \left( x, v \right) d\mu ds \\
		& = \left( \int_{\T^d \times \R^d} L \left( x, v \right) d\mu \right) \left( \int_{-\infty}^0 \lambda e^{\lambda s} ds \right) = \int_{\T^d \times \R^d} L \left( x, v \right) d\mu.
	\end{align*}
	Thus, $\int_{\T^d \times \R^d} \left(  L \left( x, v \right) - \lambda u_\lambda (x) \right)  d\mu = 0$.
\end{proof}

We now can define minimizing $\lambda$-measures for the Lagrangian $L$ as follows, and it is well-defined with the help of Proposition \ref{need0818} and Proposition \ref{existence of minimizing lambda measure}.
\begin{definition} 
	For any $\lambda \in (0,1]$, we call $\mu \in \K ( \T^d \times \R^d )$ with $\supp (\mu) \subset \tilde{\A}_{L, \lambda}$ a minimizing $\lambda$-measure for $L$, if it satisfies
	$$
	\int_{\T^d \times \R^d} \left( L(x,v) - \lambda u_\lambda (x) \right)  d\mu =0.
	$$
\end{definition}
\begin{remarks}
	The word ``minimizing'' in the definition of minimizing $\lambda$-measures for $L$ originates from the property that these measures minimize the functional $\int_{\T^d \times \R^d} \left(  L (x,v) - \lambda u_\lambda (x) \right) d\mu$ with the minimum value being precisely 0.
\end{remarks}

\subsection{Wasserstein space} \label{wasserstein space}
In this part, we recall some results from measure theory that is useful in this paper.
Throughout this part, the space $\left( M, \rho \right)$ is a separable metric space.

\medskip

For a sequence $\left\{ \mu_n \right\}_{n =1}^{\infty} \subset \PP (M)$, $\mu_n$ weakly converges to some $\mu \in \PP(M)$ as $n \rightarrow \infty$, denoted by $\mu_n \stackrel{w^*}{\longrightarrow} \mu$, if 
$$
\lim_{n \rightarrow \infty} \int_M f(x) d \mu_n = \int_M f(x) d \mu, \quad \forall f \in \CC_b ( M ).
$$

For each $p \in [1, +\infty )$, the Wasserstein space of order $p$ is defined by
$$
\PP_p (M) := \left\{ \mu \in \PP(M)\, \middle \vert \,\, \int_M \rho^p ( x_0 , x )  d \mu < +\infty \right\},
$$
where $x_0 \in M$ is an arbitrary point.
The Monge-Kantorovich distance on $\PP_p(M)$ is defined by
$$
d_p\left(\mu_1, \mu_2 \right) :=\inf _{\eta \in \tilde{\Pi}(\mu_1, \mu_2)} \left( \int_{M^2} \rho^p(x, y)  d \eta \right)^{1 / p}, \quad \forall \mu_1,\mu_2 \in \PP_p(M),
$$
where $\tilde{\Pi}(\mu_1, \mu_2)$ is the set of Borel probability measures on $M^2$ such that $\eta \left(A \times M\right)=\mu_1(A)$ and $\eta \left(M \times A\right)=\mu_2(A)$ for any Borel set $A$ of $M$.

As for the distance $d_1$, which is often called Kantorovich-Rubinstein distance, can be characterized by a useful duality formula (see, for instance, \cite{2013}) as
$$
d_1 (\mu_1, \mu_2) = \sup \left\{ \int_M g(x) d \mu_1 - \int_M g(x)  d \mu_2 \right\}, \quad \forall \mu_1, \mu_2 \in \PP_1 (M), 
$$
where the supremum is taken over all 1-Lipschitz functions $g: M \rightarrow \R$.

\medskip
\medskip

We now recall the relation between weakly convergence and $d_p$-convergence. See \cite[Theorem 7.1.5]{MR2401600} for examples.
\begin{proposition}
	If a sequence of measures $\left\{ \mu_i \right\}_{i =1}^{\infty} \subset \PP_p (M)$ converges to some $\mu \in \PP_p (M)$ in $d_p$-topology, then $\mu_i$ converges to $\mu$ weakly.
	Conversely, if $\bigcup_{i=1}^\infty \supp \{\mu_i\}$ is contained in a compact subset of $M$ and $\mu_i$ converges to $\mu$ weakly, then $\mu_i$ converges to $\mu$ in $d_p$-topology.
\end{proposition}

\subsection{Results of \cite{MR4605206}: Discretization of mean field games} \label{recall conclusions of W&I paper}
In this section, we recall some results from \cite{MR4605206}, which used the discrete method to analyze the MFGs \eqref{MFG}.
%We also denote by $L : \T^d \times \R^d \rightarrow \R$ the Tonelli Lagrangian that associated with $H$ by the Fenchel inequality.
Assume the Lagrangian $L$ satisfies (L\ref{MFG_L1})-(L\ref{MFG_L3}) and the non-local coupling term $F$ satisfies (F\ref{MFG_F1})-(F\ref{MFG_F3}).

\begin{proposition} \label{equi-lipschitz of classical solution}
	There exists a constant $C>0$ such that if $\tau \in (0,1)$, $m \in \PP ( \T^d )$ and $u_{\tau,m}$ is a solution to (\ref{discrete_laxoleinik_classical}), then $u_{\tau,m}$ is Lipschitz and $\Lip \left( u_{\tau,m} \right) \leq C$, where the Lipschitz constant is defined by
	$$
	\Lip \left( u_{\tau,m} \right) := \sup_{x,y \in \R^d} \frac{\left\vert u_{\tau,m}(x) - u_{\tau,m} (y) \right\vert}{\vert x-y \vert}.
	$$
\end{proposition}

\begin{lemma}
	There exist a compact subset $K \subset \T^d \times \R^d$ and a constant $\tau_0 >0$ such that for any $m \in \PP ( \T^d )$, any $\tau \in (0,\tau_0)$, and any minimizing $\tau$-holonomic measure $\mu$ for $L_m$, we always have $\supp (\mu) \subset K$.
\end{lemma}

The main result of \cite{MR4605206} is stated as follows.
\begin{proposition} \label{main theorem in classical case}
	For each $\tau \in (0,\tau_0)$, there is $m \in \PP ( \T^d )$ such that there exists a minimizing $\tau$-holonomic measure $\mu_{\tau,m}$ for $L_m$ with $m = \pi \sharp \mu_{\tau,m}$. We call such a measure $m$ the minimizing $\tau$-holonomic measure for MFGs \eqref{MFG} and denote it by $m_\tau$.
	%Such a measure $m$ is denoted by $m_\tau$, and we call such a measure the minimizing $\tau$-holonomic measure for MFGs \eqref{MFG}.
	
	There is a subsequence $\tau_i \rightarrow 0$, a subsequence $m_{\tau_i} \stackrel{w^*}{\longrightarrow} m_0$, and a subsequence $u_{\tau_i, m_{\tau_i}}$ solutions to \eqref{discrete_laxoleinik_classical} such that $u_{\tau_i, m_{\tau_i}}$ converges to $u_0$ uniformly on $\T^d$ and $(u_0, m_0)$ is a solution to \eqref{MFG}.
\end{proposition}

\section{Discretization of discounted mean field games} \label{Section 3}
In this section, we provide detailed proofs concerning the discretization of DMFGs \eqref{DMFG}. In Sect. \ref{discrete aubry set and lambda minimizing measure}, we prove that the discrete discounted Aubry set is well-defined and compact. Sect. \ref{Priori estimates} is devoted to establishing several a priori estimates. Finally, with the help of these estimates, we prove that the minimizing $\tau$-holonomic $\lambda$-measure, which is the discretization of the minimizing $\lambda$-measure, is well-defined and one can find a compact set that contains the supports of all minimizing $\tau$-holonomic $\lambda$-measures.

\medskip

\subsection{Discrete discounted Aubry sets} \label{discrete aubry set and lambda minimizing measure}
We first prove the uniqueness of the $\Z^d$-periodic solution to the discrete Lax-Oleinik equation \eqref{discrete_laxoleinik} in Proposition \ref{existence of solution of discounted HJE}. This result ensures that discrete discounted Aubry sets and minimizing $\tau$-holonomic $\lambda$-measures are well-defined.

\begin{proposition} \label{existence of solution of discounted HJE}
For any $\tau \in (0,1)$, any $\lambda \in (0,1]$ and any $m \in \PP ( \T^d )$, there exists a unique $\Z^d$-periodic continuous function $u_{\tau,m}^\lambda$ satisfying the discrete Lax-Oleinik equation (\ref{discrete_laxoleinik}). Moreover, 
$$
\left\Vert u_{\tau,m}^\lambda \right\Vert_{\infty} \leq \frac{C_0}{\lambda},
$$
where $C_0$ is defined as in \eqref{def of C_0}.
\end{proposition}
\begin{proof}
Fix $\tau \in (0,1)$, $\lambda \in (0,1]$ and $m \in \PP ( \T^d )$. Define the discrete Lax-Oleinik operator for (\ref{DMFG}a) by
$$
\TT_{\tau, \lambda}^m u(y) := \inf_{x \in \R^d} \left( (1-\tau \lambda) u(x) + \LL_{\tau,m} (x,y)\right) , \quad \forall y \in \R^d.
$$
Let $u,v \in \CC ( \R^d )$ be $\Z^d$-periodic functions. For any $y \in \R^d$ and any minimizer $x$ of $\TT_{\tau, \lambda}^m v(y)$, we obtain that
\begin{align*}
\TT_{\tau, \lambda}^m u(y) -\TT_{\tau, \lambda}^m v(y)
& \leq (1-\tau \lambda) u(x) + \LL_{\tau,m} (x,y) - (1-\tau \lambda) v(x) - \LL_{\tau,m} (x,y) \\
& = (1-\tau \lambda) u(x) -(1-\tau \lambda) v(x) \leq (1-\tau \lambda) \left\Vert u-v \right\Vert_{\infty}.
\end{align*}
Similarly, we obtain $\TT_{\tau, \lambda}^m v(y) -\TT_{\tau, \lambda}^m u(y) \leq (1-\tau \lambda) \left\Vert u-v \right\Vert_{\infty}$, and conclude that
$$
\left\Vert \TT_{\tau, \lambda}^m u -\TT_{\tau, \lambda}^m v \right\Vert_{\infty} \leq (1-\tau \lambda) \left\Vert u-v \right\Vert_{\infty}.
$$

Moreover, let $u \in \CC ( \R^d )$ be a $\Z^d$-periodic function satisfying $\left\Vert u \right\Vert_\infty \leq \frac{C_0}{\lambda}$. For any $z \in \R^d$, we have
$$
\TT_{\tau, \lambda}^m u (z) \leq (1-\tau \lambda) \max_{x \in \R^d} u(x) + \max_{x \in \R^d} \LL_{\tau,m} (x,x),
$$
$$
\TT_{\tau, \lambda}^m u (z) \geq (1-\tau \lambda) \min_{x \in \R^d} u(x) + \min_{x,y\in \R^d} \LL_{\tau,m} (x,y).
$$
Thus, we obtain that
$$
\left\Vert \TT_{\tau, \lambda}^m u \right\Vert_\infty \leq (1-\tau \lambda) \left\Vert u \right\Vert_\infty + \tau C_0 \leq \frac{C_0}{\lambda}.
$$
In conclusion, there exists a unique $\Z^d$-periodic solution $u_{\tau,m}^\lambda$ satisfying $\left\Vert u_{\tau,m}^\lambda \right\Vert_\infty \leq \frac{C_0}{\lambda}$ and \eqref{discrete_laxoleinik}.

If there exists another $\Z^d$-periodic continuous function $u_0 \neq u_{\tau,m}^\lambda$ satisfying the discrete Lax-Oleinik equation \eqref{discrete_laxoleinik}, then there is a constant $M > \frac{C_0}{\lambda}$ such that $\left\Vert u_0 \right\Vert_\infty \leq M$.
Let $u \in \CC ( \R^d )$ be a $\Z^d$-periodic function with $\left\Vert u \right\Vert_\infty \leq M$. Similarly, we can also obtain that
$$
\left\Vert \TT_{\tau, \lambda}^m u \right\Vert_\infty \leq (1-\tau \lambda) \left\Vert u \right\Vert_\infty + \tau C_0 \leq M,
$$
which indicates that there exists a unique $\Z^d$-periodic function $u_0$ satisfying $\left\Vert u_0 \right\Vert_\infty \leq M$ and \eqref{discrete_laxoleinik}. Thus, $u_0 = u_{\tau,m}^\lambda$. Here comes a contradiction.
Hence, we conclude that $u_{\tau,m}^\lambda$ is the unique $\Z^d$-periodic solution to (\ref{discrete_laxoleinik}).
\end{proof}

\begin{remarks}
	From now on, for any $\tau \in (0,1)$, any $\lambda \in (0,1]$ and any $m \in \PP ( \T^d )$, denote by $u_{\tau,m}^\lambda$ the only solution to \eqref{discrete_laxoleinik}, which is bounded by $\frac{C_0}{\lambda}$.
\end{remarks}

The next lemma guarantees the existence of the discounted calibrated configuration of $u_{\tau,m}^\lambda (x)$. This concept is subsequently utilized in the definition of discrete discounted Aubry sets.

\begin{lemma} \label{existence of calibrated configuration}
	For any $\tau \in (0,1)$, any $\lambda \in (0,1]$, any $m \in \PP ( \T^d)$ and any $x_0 \in \R^d$, there always exists a point $x_{-1} \in \R^d$ such that
	$$
	u_{\tau, m}^\lambda \left( x_0 \right) = (1- \tau \lambda) u_{\tau,m}^\lambda (x_{-1}) + \LL_{\tau, m}(x_{-1} ,x_0).
	$$
\end{lemma}
\begin{proof} 
	By (L\ref{MFG_L3}), there exists a constant $R_\tau > \frac{1}{\tau}$ such that
	$$
	\inf_{\left\vert v \right\vert \geq R_\tau} \inf_{x \in \R^d} L(x,v) > \sup_{\left\vert v \right\vert \leq \frac{1}{\tau}} \sup_{x \in \R^d} L(x,v) + 2F_\infty,
	$$
	indicating that
	$$
	\inf_{\left\vert v \right\vert \geq R_\tau} \inf_{x \in \R^d} \inf_{m \in \PP ( \T^d )} L_m(x,v) > \sup_{\left\vert v \right\vert \leq \frac{1}{\tau}} \sup_{x \in \R^d} \sup_{m \in \PP ( \T^d )} L_m(x,v).
	$$
	Thus, given $x_0 \in \R^d$, there exists $v_0 \in \R^d$ with $\vert v_0 \vert \leq R_\tau$ such that
	$$
	u_{\tau, m}^\lambda \left( x_0 \right) = (1- \tau \lambda) u_{\tau,m}^\lambda (x_0 - \tau v_0) + \tau L_m(x_0-\tau v_0 ,v_0).
	$$
	The proof is concluded by setting $x_{-1} = x_0 - \tau v_0$.
\end{proof}

We then prove that $\Psi^n_{L_m, \lambda, \tau} \left( \tilde{\Sigma}_{L_m, \lambda}^\tau \right)$ is compact. Consequently, the discrete discounted Aubry set, which is an intersection of a decreasing family of compact sets, is compact as well.

\begin{lemma} \label{compact of discrete aubry set}
For any $\tau \in (0,1)$, any $\lambda \in (0,1]$, any $m \in \PP ( \T^d )$ and any $n \in \N$, $\Psi^n_{L_m, \lambda, \tau} \left( \tilde{\Sigma}_{L_m, \lambda}^\tau \right)$ is compact. Consequently, $\tilde{\A}_{L_m, \lambda}^\tau$ is compact as well.
\end{lemma}
\begin{proof}
Suppose there exists a sequence $\left\{ ([x_i], v_i) \right\}_{i=1}^\infty \subset \Psi^n_{L_m, \lambda, \tau} \left( \tilde{\Sigma}_{L_m, \lambda}^\tau \right)$ converging to some point $([x_0], v_0)$.
Thus, for any $i \in \N$, there exists a sequence $\left\{ x^i_{-k} \right\}_{k=0}^{n+1} \subset \R^d$ such that $\Pi \left( x_{-n-1}^i \right) = [x_i]$ and $\frac{x_{-n}^i - x_{-n-1}^i}{\tau} = v_i$, and the sequence satisfies
\begin{equation} \label{need0619}
u^\lambda_{\tau,m} (x_0^i) = \Sigma_{k=0}^{n-1} (1- \tau\lambda)^k \tau L_m \left( x_{-k-1}^i, v_{-k-1}^i \right) + (1-\tau \lambda)^n u_{\tau,m}^\lambda (x_{-n}^i),
\end{equation}
where $v_{-k}^i = \frac{x_{-k+1}^i - x_{-k}^i}{\tau}$. 
Without loss of generality, assume that the set $\left\{ x_{-n-1}^i \right\}_{i=0}^\infty$ is bounded.
Since $\vert v^i_{-k} \vert \leq R_\tau$ for any $i \in \N$ and any integer $1 \leq k \leq n+1$, there exists a compact set containing all sequences $\left\{ x^i_{-k} \right\}_{k=0}^{n+1}$ for every $i \in \N$. 
Thus, for any integer $k \in [0, n+1]$, there exists a point $(x_{-k}^0, v_{-k}^0)$ such that $\left( x_{-k}^i , v_{-k}^i \right) \rightarrow (x_{-k}^0, v_{-k}^0)$ as $i \rightarrow \infty$. Since $v_{-k-1}^i = \frac{x_{-k}^i - x_{-k-1}^i}{\tau}$, it follows that $v_{-k-1}^0 = \frac{x_{-k}^0 - x_{-k-1}^0}{\tau}$. By \eqref{need0619}, we conclude that $([x^0_{-n-1}], v^0_{-n-1}) = ([x_0], v_0) \in \Psi^n_{L_m, \lambda, \tau} \left( \tilde{\Sigma}_{L_m, \lambda}^\tau \right)$.
In conclusion, the set $\Psi^n_{L_m, \lambda, \tau} \left( \tilde{\Sigma}_{L_m, \lambda}^\tau \right)$ is compact, as it is a subset of the non-empty compact set $\tilde{\Sigma}_{L_m, \lambda}^\tau$.
Moreover, since for any $n_1 > n_2$,
$$
\Psi^{n_1}_{L_m, \lambda, \tau} \left( \tilde{\Sigma}_{L_m, \lambda}^\tau \right) \subset \Psi^{n_2}_{L_m, \lambda, \tau} \left( \tilde{\Sigma}_{L_m, \lambda}^\tau \right),
$$
the set $\tilde{\A}_{L_m , \lambda}^\tau$ is also compact.
\end{proof}

\subsection{Priori estimates} \label{Priori estimates}

At the beginning of this section, we prove that the first and second derivatives of the minimizing curve for the action $h_{\tau, \lambda}^m(x,y)$ are bounded. This result is used in the proofs of Lemma \ref{lemma3.4}, Proposition \ref{minimizer x is near} and Proposition \ref{convergence2_prop1}.

\begin{lemma} \label{lemma 1}
For any $D>0$, there is $C(D)>0$ such that for each $\tau \in (0,1)$, each $\lambda \in (0,1]$, each $m \in \PP \left(\T^d \right)$, each $x,y \in \R^d$ with $\vert x-y \vert \leq \tau D$, and each minimizing curve $\gamma_{x,y}^{m, \lambda}: [-\tau, 0] \rightarrow \R^d$ of $L_m$ with the action $h_{\tau, \lambda}^m(x,y)$, there hold
$$
\left\vert \dot{\gamma}_{x,y}^{m, \lambda} (s) \right\vert, \left\vert \ddot{\gamma}_{x,y}^{m,\lambda} (s) \right\vert \leq C(D), \quad \forall s \in [-\tau, 0].
$$
\end{lemma}
\begin{proof}
Fix $D>0$. For each $\lambda \in (0,1]$, each $\tau \in (0,1)$ and each $x,y \in \R^d$ with $\vert x - y \vert \leq \tau D$, define the segment $l_{x,y} :[-\tau, 0] \rightarrow \R^d$ connecting $x$ and $y$ by
$$
l_{x,y} (s) := y + s \frac{y-x}{\tau}, \quad \forall s \in [-\tau, 0].
$$
Then, for each $m \in \PP ( \T^d )$,
\begin{align*}
\int_{-\tau}^0 e^{\lambda s} L_m \left( l_{x,y}(s), \dot{l}_{x,y}(s) \right) ds
& = \int_{-\tau}^0 e^{\lambda s} \left( L \left( l_{x,y} (s), \frac{y-x}{\tau} \right) + F \left( l_{x,y}(s),m \right) \right) ds \\
& \leq \frac{1- e^{- \lambda \tau}}{\lambda} \left( \max_{x \in \T^d, \vert v \vert \leq D} L(x,v) + F_\infty \right).
\end{align*}
Let $C_1(D):=  \max_{x \in \T^d, \vert v \vert \leq D} L(x,v) + F_\infty <+\infty$.

Since $L$ is superlinear in $v$, there is a constant $R>0$ such that for any $\vert v \vert >R$,
$$
L(x,v) + F(x,m) > C_1(D), \quad \forall x \in \T^d, \,\, \forall m \in \PP ( \T^d ).
$$
Let $\Sigma_R := \left\{ (x,v) \in \T^d \times \R^d \middle\vert \,\, \vert v \vert \leq R \right\}$. Obviously, $\Sigma_R$ is a compact subset of $\T^d \times \R^d$. By the compactness of $\Sigma_R$ and $\PP ( \T^d )$, the continuous dependence of the solutions on the initial condition and the parameter, and (F\ref{MFG_F2}), one can deduce that there is a constant $R_1>0$ independent of $\tau$ and $m$ such that
$$
\Phi^s_{L_m, \lambda} \left( \Sigma_R \right) \subset \Sigma_{R_1}:= \left\{ (x,v) \in \T^d \times \R^d \middle\vert \,\, \vert v \vert \leq R_1 \right\}, \quad \forall s \in [-1,1], \,\,\forall m \in \PP ( \T^d ).
$$

For any minimizing curve $\gamma_{x,y}^{m, \lambda} : [-\tau, 0] \rightarrow \R^d$ of $L_m$ with the action $h_{\tau, \lambda}^m (x,y)$, we assert that $\left\vert \dot{\gamma}_{x,y}^{m, \lambda} (s) \right\vert \leq R_1$ for all $s \in [-\tau, 0]$. Otherwise, there would exist some $s_0 \in [-\tau, 0]$ such that $\left\vert \dot{\gamma}_{x,y}^{m, \lambda} (s_0) \right\vert > R_1$. We define a curve $\tilde{\gamma}$ in $\T^d$ by $\tilde{\gamma}(s) := \Pi \left( \gamma_{x,y}^{m, \lambda}(s) \right)$ for any $s \in [-\tau, 0]$. Since $\gamma_{x,y}^{m, \lambda}$ is a minimizing curve, it follows that $\left( \tilde{\gamma}(s), \dot{\tilde{\gamma}}(s) \right) \in \T^d \times \R^d$ is part of the Lagrangian flow $\Phi^t_{L_m, \lambda}$. In view of $\left\vert \dot{\tilde{\gamma}}(s_0)\right\vert = \left\vert \dot{\gamma}_{x,y}^{m, \lambda} (s_0) \right\vert > R_1$, one can deduce that
$$
\left( \tilde{\gamma}(s), \dot{\tilde{\gamma}}(s) \right) \notin \Sigma_R, \quad \forall s \in [-\tau, 0].
$$
Thus, $\left\vert \dot{\tilde{\gamma}}(s)\right\vert = \left\vert \dot{\gamma}_{x,y}^{m, \lambda} (s) \right\vert > R$ for any $s \in [-\tau, 0]$. Thus, we have that
$$
L \left( \gamma_{x,y}^{m, \lambda}(s), \dot{\gamma}_{x,y}^{m, \lambda}(s) \right) + F \left( \gamma_{x,y}^{m, \lambda}(s), m \right) > C_1(D), \quad \forall s \in [-\tau, 0],
$$
implying that
$$
\int_{-\tau}^0 e^{\lambda s} L_m \left( \gamma_{x,y}^{m, \lambda}(s), \dot{\gamma}_{x,y}^{m, \lambda}(s) \right) ds > C_1(D) \frac{1- e^{-\lambda \tau}}{\lambda}  \geq \int_{-\tau}^0 e^{\lambda s} L_m \left( l_{x,y}(s), \dot{l}_{x,y}(s) \right) ds.
$$
Here comes a contradiction.

At last, note that
\begin{align*}
\ddot{\gamma}_{x,y}^{m, \lambda} = \frac{\partial^2 L}{\partial v^2} \left( \gamma_{x,y}^{m, \lambda}, \dot{\gamma}_{x,y}^{m, \lambda} \right) ^{-1} 
& \left( \frac{\partial L}{\partial x} \left( \gamma_{x,y}^{m, \lambda}, \dot{\gamma}_{x,y}^{m, \lambda} \right) + \frac{\partial F}{\partial x} \left( \gamma_{x,y}^{m, \lambda}, m \right) \right. \\
& \quad \left.  - \lambda \frac{\partial L}{\partial v} \left( \gamma_{x,y}^{m, \lambda}, \dot{\gamma}_{x,y}^{m, \lambda} \right) - \frac{\partial^2 L}{\partial x \partial v} \left( \gamma_{x,y}^{m, \lambda}, \dot{\gamma}_{x,y}^{m, \lambda} \right) \cdot \dot{\gamma}_{x,y}^{m, \lambda} \right),
\end{align*}
which completes the proof.
\end{proof}

The following lemma establishes that the minimal action $h_{\tau, \lambda}^m (x,y)$ and the discrete action $\LL_{\tau,m}(x,y)$ are close provided $x$ and $y$ are sufficiently near. This result is employed in the proof of Proposition \ref{convergence to find viscosity solution}.

\begin{lemma} \label{lemma3.4}
For any $D>0$, there is $\tilde{C}(D)>0$ such that for each $\tau \in (0,1)$, each $\lambda \in (0,1]$ and each $x,y \in \R^d$ with $\vert x-y\vert \leq \tau D$, there holds
$$
\left\vert h_{\tau, \lambda}^m (x,y) - \LL_{\tau,m}(x,y) \right\vert \leq \tau^2 \tilde{C}(D), \quad \forall m \in \PP ( \T^d ).
$$
\end{lemma}
\begin{proof}
Fix $D>0$. Let $C(D)$ be the constant given by Lemma \ref{lemma 1}. Let $\lambda \in (0,1]$, $\tau \in (0,1)$ and $x,y \in \R^d$ with $\vert x-y \vert \leq \tau D$. Let $\gamma_{x,y}^{m, \lambda}: [-\tau, 0] \rightarrow \R^d$ be a minimizing curve of $h_{\tau, \lambda}^m (x,y)$. Then, by Lemma \ref{lemma 1}, we get that $\left\vert \dot{\gamma}_{x,y}^{m, \lambda} (s) \right\vert, \left\vert \ddot{\gamma}_{x,y}^{m, \lambda} (s) \right\vert \leq C(D)$ for all $s \in [-\tau, 0]$. For any $s \in [-\tau, 0]$, we have that
\begin{align*}
& \left\vert \gamma_{x,y}^{m, \lambda} (s) - x \right\vert = \left\vert \gamma_{x,y}^{m, \lambda} (s) - \gamma_{x,y}^{m, \lambda} (-\tau) \right\vert \leq \tau C(D), \quad \left\vert \dot{\gamma}_{x,y}^{m, \lambda} (s) - \dot{\gamma}_{x,y}^{m, \lambda} (0) \right\vert \leq \tau C(D), \\
& \left\vert \frac{y-x}{\tau} - \dot{\gamma}_{x,y}^{m, \lambda} (0) \right\vert \leq \tau C(D), \quad \left\vert \dot{\gamma}_{x,y}^{m, \lambda} (s) - \frac{y-x}{\tau} \right\vert \leq 2\tau C(D).
\end{align*}
By (L\ref{MFG_L1}) and (F\ref{MFG_F2}), there holds
\begin{align*}
& \left\vert h_{\tau, \lambda}^m (x,y) - \LL_{\tau,m}(x,y) \right\vert \\
= & \left\vert \int_{-\tau}^0 e^{\lambda s} \left( L \left( \gamma_{x,y}^{m, \lambda}(s), \dot{\gamma}_{x,y}^{m, \lambda} (s) \right) + F\left( \gamma_{x,y}^{m, \lambda}(s),m \right) \right) - L \left( x,\frac{y-x}{\tau} \right) - F(x,m) ds  \right\vert \\
\leq & \int_{-\tau}^0 \left\vert L \left( \gamma_{x,y}^{m, \lambda}(s), \dot{\gamma}_{x,y}^{m, \lambda} (s) \right) - L \left( x,\frac{y-x}{\tau} \right) \right\vert ds + \int_{-\tau}^0 \left\vert F\left( \gamma_{x,y}^{m, \lambda} (s), m \right) - F(x,m) \right\vert ds \\
& + \left\vert \int_{-\tau}^0 (1- e^{\lambda s}) \left( L \left( \gamma_{x,y}^{m, \lambda}(s), \dot{\gamma}_{x,y}^{m, \lambda} (s) \right) + F\left( \gamma_{x,y}^{m, \lambda}(s),m \right) \right) ds \right\vert \\
\leq & \left( \sup_{x \in \T^d, \vert v \vert \leq C(D) + D} \vert DL (x,v) \vert + F_{\infty} \right) 2 \tau^2 C(D) + \left( \sup_{x \in \T^d, \vert v \vert \leq C(D)} \vert L(x,v) \vert + F_\infty \right) \left( \tau - \frac{1-e^{-\lambda \tau}}{\lambda} \right) \\
= & : \tilde{C} (D) \tau^2.
\end{align*}
\end{proof}

The following lemma is a direct consequence of assumptions (L\ref{MFG_L1})-(L\ref{MFG_L3}) and (F\ref{MFG_F1}). We omit the proof here. This result is used in the proofs of Proposition \ref{equi-Lipschitz of discounted HJE} and Proposition \ref{minimizer x is near}.

\begin{lemma} \label{222}
Functions $\LL_{\tau,m}(x,y)$ and $h_{\tau, \lambda}^m (x,y)$ satisfy the following properties.
\begin{enumerate}
\item For each $D>0$,
$$
\inf_{m \in \PP ( \T^d ), \tau \in (0,1), x,y \in \R^d} \frac{\LL_{\tau,m} (x,y)}{\tau} > -\infty,
$$
$$
\sup_{m \in \PP ( \T^d ), \tau \in (0,1), \vert y-x \vert \leq \tau D} \frac{\LL_{\tau,m} (x,y)}{\tau} < +\infty.
$$

\item The next two limits
$$
\lim_{D \rightarrow +\infty} \inf_{\tau \in (0,1), \vert x-y \vert \geq \tau D} \frac{\LL_{\tau,m}(x,y)}{\vert x-y \vert} = +\infty,
$$
$$
\lim_{D \rightarrow +\infty} \inf_{\lambda \in (0, 1], \tau \in (0,1), \vert x-y \vert \geq \tau D} \frac{h_{\tau, \lambda}^m (x,y)}{\vert x-y \vert} = +\infty
$$
are both uniformly on $m \in \PP \left(\T^d \right)$.

\item For each $D>0$, there exists a constant $C(D)>0$ such that for any $\tau \in (0,1)$, any $x,y,z \in \R^d$ and any $m \in \PP \left(\T^d \right)$,
\begin{itemize}
\item if $\vert y-x \vert \leq \tau D$ and $\vert z-x \vert \leq \tau D$, then $\left\vert \LL_{\tau,m}(x,z) - \LL_{\tau,m}(x,y) \right\vert \leq C(D) \vert z-y \vert$.
\item if $\vert z-x \vert \leq \tau D$ and $\vert z-y \vert \leq \tau D$, then $\left\vert \LL_{\tau,m}(x,z) - \LL_{\tau,m}(y,z) \right\vert \leq C(D) \vert y-x \vert$.
\end{itemize}
\end{enumerate}
\end{lemma}

The following proposition establishes that all solutions to the discrete Lax-Oleinik equation (\ref{discrete_laxoleinik}) are equi-Lipschitz. This result is employed in the proofs of Proposition \ref{convergence to find viscosity solution} and Proposition \ref{convergence1_prop1}.

\begin{proposition} \label{equi-Lipschitz of discounted HJE}
There exists a constant $C>0$ such that for any $\tau \in (0,1)$, any $\lambda \in (0,1]$ and any $m \in \PP ( \T^d )$, the function $u_{\tau,m}^\lambda$ is Lipschitz and $\Lip \left( u_{\tau,m}^\lambda \right) \leq C$.
\end{proposition}
\begin{proof}
Let
\begin{align*}
& C_1 := 2 \sup_{\tau \in (0,1), \vert y-x \vert \leq \tau, m \in \PP ( \T^d )} \frac{\LL_{\tau,m}(x,y)}{\tau}, \\
& D := \inf \left\{ D^\prime>1 \middle\vert \inf_{\tau \in (0,1), \vert y-x \vert > \tau D^\prime, m \in \PP ( \T^d )} \frac{\LL_{\tau,m} (x,y) - \tau C_0}{\vert y-x \vert} > C_1 + 2 C_0 \right\}, \\
& C := \max \left\{ C_1 + 2C_0, \sup_{\vert y-x \vert, \vert z-x \vert \leq \tau (D+1), \tau \in (0,1), m \in \PP ( \T^d )} \frac{\LL_{\tau,m} (x,y) - \LL_{\tau,m} (x,z)}{\vert y-z \vert}  \right\}.
\end{align*}
Note that the three constants $C_1$, $D$ and $C$ are well-defined by Lemma \ref{222}.

Firstly, we show that if $\vert x-y \vert >\tau$, then $u_{\tau,m}^\lambda (y) - u_{\tau,m}^\lambda (x) \leq \left( C_1+2C_0 \right) \vert y-x \vert$. Indeed, choose $n \geq 2$ such that $(n-1)\tau < \vert y-x \vert \leq n\tau \leq 2 \vert y-x \vert$ and define $x_i = x +\frac{i}{n} (y-x)$ for $i=0,\cdots, n$. Then we obtain
$$
u_{\tau,m}^\lambda \left( x_{i+1} \right) -u_{\tau,m}^\lambda \left( x_{i} \right) \leq \LL_{\tau,m} \left( x_{i} , x_{i+1} \right) + \tau C_0 \leq \frac{1}{2} \tau C_1 + \tau C_0,
$$
$$
u_{\tau,m}^\lambda (y) - u_{\tau,m}^\lambda (x) \leq n\tau  \left( \frac{1}{2} C_1 + C_0 \right) \leq \left( C_1+2C_0 \right) \vert y-x \vert.
$$

Then, for any $y \in \R^d$ and any $x \in \argmin_{x \in \R^d} \left\{ (1- \tau \lambda) u_{\tau,m}^\lambda (x) + \LL_{\tau,m}(x,y) \right\}$, assume that $\vert y-x \vert > \tau D$. Then we obtain that
$$
\left( C_1+2C_0 \right) \vert y-x \vert \geq u_{\tau,m}^\lambda (y) - u_{\tau,m}^\lambda (x) \geq \LL_{\tau,m}(x,y) - \lambda \tau \left\Vert u_{\tau,m}^\lambda \right\Vert_\infty \geq \LL_{\tau,m}(x,y) - \tau C_0 > \left( C_1+2C_0 \right) \vert y-x \vert.
$$
Here comes a contradiction. Thus, $\vert y-x \vert \leq \tau D$.

Let $y,z \in \R^d$ with $\vert y-z \vert \leq \tau$. Let $x \in \argmin_{x \in \R^d} \left\{ (1- \tau \lambda) u_{\tau,m}^\lambda (x) + \LL_{\tau,m}(x,y) \right\}$. Then $\vert x - y \vert \leq \tau D$, $\vert x - z \vert \leq \tau (D+1)$ and
$$
u_{\tau,m}^\lambda (z) - u_{\tau,m}^\lambda (x) \leq (1- \tau \lambda) u_{\tau,m}^\lambda (x) + \LL_{\tau,m}(x,z) - u_{\tau,m}^\lambda (x) = \LL_{\tau,m}(x,z) - \tau \lambda u_{\tau,m}^\lambda (x).
$$
Thus
\begin{align*}
u_{\tau,m}^\lambda (z) - u_{\tau,m}^\lambda (y)
& \leq \LL_{\tau,m}(x,z) - \tau \lambda u_{\tau,m}^\lambda (x) + u_{\tau,m}^\lambda (x) - u_{\tau,m}^\lambda (y) \\
& = \LL_{\tau,m}(x,z) - \tau \lambda u_{\tau,m}^\lambda (x) + \tau \lambda u_{\tau,m}^\lambda (x) - \LL_{\tau,m}(x,y) \\
& = \LL_{\tau,m}(x,z) - \LL_{\tau,m}(x,y) \leq C \vert y-z \vert.
\end{align*}                                    
By changing the roles of $z$ and $y$, we just proved that $\Lip \left( u_{\tau,m}^\lambda \right) \leq C$.
\end{proof}

For any $\Z^d$-periodic function $\varphi$, any $\tau \in (0,1)$, any $\lambda \in (0,1]$ and any $m \in \PP ( \T^d )$, recall definitions of two Lax-Oleinik operators for (\ref{DMFG}a) by
\begin{align}
& T_{\tau, \lambda}^m \varphi(y) := \inf_{x \in \R^d} \left(  e^{- \lambda\tau} \varphi(x) + h_{\tau, \lambda}^m (x,y) \right) , \quad \forall y \in \R^d, \label{DMFG laxoleinik} \\
& \TT_{\tau, \lambda}^m \varphi(y) := \inf_{x \in \R^d} \left(  (1 - \tau \lambda) \varphi(x) + \LL_{\tau, m} (x,y) \right) , \quad \forall y \in \R^d.\label{DMFG laxoleinik_discrete}
\end{align}
The next proposition demonstrates that for any point $y$, any Lipschitz function $\varphi$ and any minimizer $x$ of $T_{\tau, \lambda}^m \varphi(y)$ or $\TT_{\tau, \lambda}^m \varphi(y)$, the distance between $x$ and $y$ and the uniform norms of $T_{\tau, \lambda}^m \varphi - \varphi$ and $\TT_{\tau, \lambda}^m \varphi - \varphi$ are all small. This result is crucial and is employed in the proofs of Proposition \ref{existence of minimizing tau holonomic lambda measure for MFG}, Proposition \ref{convergence to find viscosity solution}, Lemma \ref{convergence of aubry set}, Proposition \ref{convergence2_prop1} and Proposition \ref{convergence1_prop1}.

\begin{proposition} \label{minimizer x is near}
For any constant $\kappa>0$, any $\tau \in (0,1)$, any $\lambda \in (0,1]$ and any $m \in \PP ( \T^d )$, if there exists a constant $C>0$ such that the function $\varphi$ is a $\Z^d$-periodic Lipschitz function satisfying $\Lip(\varphi)\leq \kappa$ and $\left\Vert \varphi \right\Vert_\infty \leq C / \lambda$, then there exists a constant $D_\kappa>0$, depending only on $\kappa$, such that
\begin{enumerate}
\item for any $y \in \R^d$ and any $x \in \argmin_{x \in \R^d} \left\{ (1- \tau \lambda) \varphi(x) + \LL_{\tau,m}(x,y) \right\}$, it is clear that $\vert x-y \vert \leq \tau D_\kappa$. Moreover, there is a constant $\tilde{C}_\kappa$ such that
$$
\left\Vert \TT_{\tau, \lambda}^m \varphi - \varphi \right\Vert_{\infty} \leq \tau \tilde{C}_\kappa.
$$
\item for any $y \in \R^d$ and any $x \in \argmin_{x \in \R^d} \left\{ e^{-\lambda \tau} \varphi(x) + h_{\tau, \lambda}^m(x,y) \right\}$, it is clear that $\vert x-y \vert \leq \tau D_\kappa$. Moreover, there is a constant $\tilde{C}_\kappa$ such that
$$
\left\Vert T_{\tau, \lambda}^m \varphi - \varphi \right\Vert_{\infty} \leq \tau \tilde{C}_\kappa.
$$
\end{enumerate}
\end{proposition}
\begin{proof}
1. Let $\kappa >0$. Define
$$
D_\kappa := \inf \left\{ D^\prime>1 \middle\vert \inf_{\tau \in (0,1), \vert y-x \vert > \tau D^\prime, m \in \PP ( \T^d )} \frac{\LL_{\tau,m} (x,y) - \LL_{\tau,m} (y,y)}{\vert y- x \vert} > \kappa \right\}.
$$
Note that this constant is well-defined by Lemma \ref{222}.
Let $\varphi$ be a $\Z^d$-periodic Lipschitz function with $\Lip(\varphi) \leq \kappa$ and $\left\Vert \varphi \right\Vert_\infty \leq C / \lambda$. For any $y \in \R^d$ and any $x \in \argmin_{x \in \R^d} \left\{ (1- \tau \lambda) \varphi(x) + \LL_{\tau,m}(x,y) \right\}$, suppose that $\vert y-x \vert > \tau D_\kappa$. Then
$$
\LL_{\tau,m}(x,y) - \LL_{\tau,m}(y,y) > \kappa \vert y-x \vert.
$$
Since $(1- \tau \lambda) \varphi(x) + \LL_{\tau,m}(x,y) \leq (1- \tau \lambda) \varphi(y) + \LL_{\tau,m}(y,y)$, we obtain that
$$
\kappa \vert y-x \vert \geq (1- \tau \lambda) \kappa \vert y-x \vert \geq (1- \tau \lambda) \left( \varphi(y) - \varphi(x) \right) \geq \LL_{\tau,m}(x,y) - \LL_{\tau,m}(y,y) > \kappa \vert y-x \vert.
$$
Here comes a contradiction. Thus, $\vert y - x \vert \leq \tau D_\kappa$.
Moreover, 
\begin{align*}
\left\vert \inf_{z \in \R^d} \left( (1- \tau \lambda) \varphi(z) + \LL_{\tau,m}(z, y) \right) - \varphi(y) \right\vert
& = \left\vert  (1- \tau \lambda) \varphi(x) + \LL_{\tau,m}(x, y)  - \varphi(y) \right\vert \\
& \leq \kappa \vert x-y \vert + \tau \left( \max_{x \in \T^d, \vert v \vert \leq D_\kappa} \left\vert L(x,v) \right\vert+ F_\infty + C \right) \\
& \leq \tau \left( \kappa D_\kappa + \max_{x \in \T^d, \vert v \vert \leq D_\kappa} \left\vert L(x,v) \right\vert+ F_\infty + C \right) =: \tau \tilde{C}_\kappa.
\end{align*}

2. Let $\kappa >0$. Define
$$
D_\kappa := \inf \left\{ D^\prime>1 \middle\vert \inf_{\lambda \in (0,1], \tau \in (0,1), \vert y-x \vert > \tau D^\prime, m \in \PP \left( \T^d \right)} \frac{h_{\tau, \lambda}^m (x,y) - h_{\tau, \lambda}^m (y,y)}{\vert y- x \vert} > \kappa \right\}.
$$
Note that this constant is well-defined by Lemma \ref{222}.
Let $\varphi$ be a $\Z^d$-periodic Lipschitz function with $\Lip(\varphi) \leq \kappa$ and $\left\Vert \varphi \right\Vert_\infty \leq C / \lambda$. For any $y \in \R^d$ and any $x \in \argmin_{x \in \R^d} \left\{ e^{-\lambda \tau} \varphi(x) + h_{\tau, \lambda}^m(x,y) \right\}$, suppose that $\vert y-x \vert > \tau D_\kappa$. Then
$$
h_{\tau, \lambda}^m (x,y) - h_{\tau, \lambda}^m (y,y) > \kappa \vert y-x \vert.
$$
Since $e^{-\lambda \tau} \varphi(x) + h_{\tau, \lambda}^m(x,y) \leq e^{-\lambda \tau} \varphi(y) + h_{\tau, \lambda}^m(y,y)$, we obtain that
$$
\kappa \vert y-x \vert \geq \kappa e^{-\lambda \tau} \vert y-x \vert \geq e^{-\lambda \tau} \left( \varphi(y) - \varphi(x) \right) \geq h_{\tau, \lambda}^m (x,y) - h_{\tau, \lambda}^m (y,y) > \kappa \vert y-x \vert.
$$
Here comes a contradiction. Thus, $\vert y-x \vert \leq \tau D_\kappa$.
Moreover, 
\begin{align*}
	\left\vert \inf_{z \in \R^d} \left( e^{-\lambda \tau} \varphi(z) + h_{\tau, \lambda}^m(z, y) \right) - \varphi(y) \right\vert
	& = \left\vert e^{-\lambda \tau} \varphi(x) + h_{\tau, \lambda}^m(x, y) - \varphi(y) \right\vert \\
	& \leq \kappa \vert x-y \vert + \frac{1- e^{-\lambda \tau}}{\lambda} \left( \max_{x \in \T^d, \vert v \vert \leq C\left( D_\kappa \right)} \left\vert L(x,v) \right\vert+ F_\infty + C \right) \\
	& \leq \tau \left( \kappa D_\kappa + \max_{x \in \T^d, \vert v \vert \leq C\left( D_\kappa \right)} \left\vert L(x,v) \right\vert+ F_\infty + C \right) =: \tau \tilde{C}_\kappa,
\end{align*}
where $C\left( D_\kappa \right)$ is defined as in Lemma \ref{lemma 1}.
                              
\end{proof}

\subsection{Minimizing $\tau$-holonomic $\lambda$-measures} \label{Minimizing tau-holonomic lambda-measures for mean field games}
Firstly, we prove that for any $\tau$-holonomic measure $\mu$, the functional $\int_{\T^d \times \R^d} \left( L_m(x,v) - \lambda u_{\tau,m}^\lambda (x) \right)  d\mu$ is non-negative in Lemma \ref{need0818_1}, and establish the existence of a measure $\mu$ such that the functional attains the minimum and the support of it belongs to the discrete discounted Aubry set in Proposition \ref{non-empty of minimizing holonomic measure}.
\begin{lemma} \label{need0818_1}
	For any $\tau \in (0,1)$, any $\lambda \in (0,1]$, any $m \in \PP ( \T^d )$ and any $\mu \in \PP_\tau ( \T^d \times \R^d )$, 
	$$
	\int_{\T^d \times \R^d} \left( L_m(x,v) - \lambda u_{\tau,m}^\lambda (x) \right) d\mu  \geq 0.
	$$
\end{lemma}
\begin{proof}
	By \eqref{discrete_laxoleinik}, for any $(x,v) \in \T^d \times \R^d$, we have
	$$
	u_{\tau,m}^\lambda (x +\tau v) \leq (1-\tau \lambda) u_{\tau,m}^\lambda (x) + \tau L_m (x,v).
	$$
	Integrate both sides with respect to $\mu$. Since $\mu \in \PP_\tau ( \T^d \times \R^d )$, we obtain that
	$$
	\tau \int_{\T^d \times \R^d} \left( L_m(x,v) - \lambda u_{\tau,m} (x)\right)  d \mu \geq 0,
	$$
	which completes the proof.
\end{proof}

\begin{proposition} \label{non-empty of minimizing holonomic measure}
	For any $\tau \in (0,1)$, any $\lambda \in (0,1]$ and any $m \in \PP ( \T^d )$, there exists a minimizing $\tau$-holonomic $\lambda$-measure for $L_m$.
\end{proposition}
\begin{proof}
	Fix $([x^0], v^0) \in \tilde{\A}_{L_m \lambda}^\tau$. Let $\left\{ x_{-k} \right\}_{k=0}^\infty$ be a discounted calibrated configuration of $u_{\tau,m}^\lambda (x^0 + \tau v^0)$ with $\Pi \left( x^0 \right) = [x^0]$.
	For any $n \in \N$, define a probability measure $\mu_n \in \PP ( \T^d \times \R^d )$ by
	$$
	\mu_n := \frac{1}{n} \Sigma_{k=1}^{n} \mathbf{1}_{([x_{-k}],v_{-k})},
	$$
	where $v_{-k} = \frac{x_{-k+1} - x_{-k}}{\tau}$ and $\mathbf{1}_{([x_{-k}],v_{-k})}$ is an indicator function of the point $([x_{-k}],v_{-k})$.
	Since $\supp (\mu_n) \subset \left\{(x,v) \middle \vert x \in \T^d, \vert v \vert \leq R_\tau \right\}$ for any $n \in \N$, there exists a measure $\mu \in \PP ( \T^d \times \R^d )$ such that $\mu_n \stackrel{w^*}{\longrightarrow} \mu$.
	For any $\varphi \in \CC ( \T^d )$, since $\mu_n$ and $\mu$ are compactly supported, it follows that
	\begin{align*}
		\left\vert \int_{\T^d \times \R^d} \left(  \varphi (x+\tau v) - \varphi(x) \right) d \mu_n \right\vert
		& \leq \frac{1}{n} \left\vert \Sigma_{k=1}^{n} \left( \varphi \left( x_{-k+1} \right) - \varphi \left( x_{-k} \right) \right) \right\vert \\
		& = \frac{1}{n} \left\vert \varphi(x_{0}) - \varphi(x_{-n}) \right\vert \leq \frac{2}{n} \left\Vert \varphi \right\Vert_\infty.
	\end{align*}
	Let $n$ tend to infinity. It is clear that
	$$
	\int_{\T^d \times \R^d} \left( \varphi (x+\tau v) - \varphi(x) \right) d \mu = 0,
	$$
	indicating that $\mu \in \PP_\tau ( \T^d \times \R^d )$.
	By the definition of $\mu_n$, we have
	$$
	\int_{\T^d \times \R^d} \left( u_{\tau, m}^\lambda (x +\tau v) - (1-\tau \lambda) u_{\tau, m}^\lambda (x) - \tau L_m(x,v) \right)  d\mu_n = 0.
	$$
	Since $\mu_n \stackrel{w^*}{\longrightarrow} \mu$, we obtain that
	$$
	\int_{\T^d \times \R^d} \left( u_{\tau, m}^\lambda (x +\tau v) - (1-\tau \lambda) u_{\tau, m}^\lambda (x) - \tau L_m(x,v) \right) d\mu = 0.
	$$
	Since $\mu \in \PP_\tau ( \T^d \times \R^d )$ and $\tau>0$,
	$$
	\int_{\T^d \times \R^d} \left( L_m(x,v) - \lambda u_{\tau,m}^\lambda (x) \right) d\mu = 0,
	$$
	indicating that $\mu$ is a minimizing $\tau$-holonomic $\lambda$-measure for $L_m$.
\end{proof}

\begin{remarks}
	The word ``minimizing'' in Definition \ref{def of minimizng holonomic measure} originates from the property that these measures minimize the functional $\int_{\T^d \times \R^d} \left( L_m (x,v) - \lambda u_{\tau,m}^\lambda (x) \right) d\mu$ with the minimum value being precisely 0.
\end{remarks}

Then, we get the conclusion that all minimizing $\tau$-holonomic $\lambda$-measures for $L_m$ are supported on a fixed compact set $K$, as established by Lemma \ref{need0818_2} and Proposition \ref{compactness of minimizing holonomic measure}. This conclusion plays an important role in the proofs of Proposition \ref{existence of minimizing tau holonomic lambda measure for MFG}, Proposition \ref{u_0 is solution of HJE}, Proposition \ref{convergence2_prop2} and Proposition \ref{convergence1_prop2}.

\begin{definition}
For each $\tau \in (0,1)$, each $\lambda \in (0,1]$ and each $m \in \PP ( \T^d )$, define the $\tau$-$\lambda$-Mather set for $L_m$ by
$$
\MMM_\tau^\lambda (L_m) := \cl \left( \bigcup \left\{ \supp(\mu) \middle\vert \,\, \mu \,\, \text{is a minimizing $\tau$-holonomic $\lambda$-measure for}\,\, L_m \right\} \right).
$$
\end{definition}

Define the set 
$$
\NN_\tau^\lambda \left( L_m, u_{\tau,m}^\lambda \right):= \left\{ (x,v) \in \T^d \times \R^d \middle\vert \,\, \tau L_m(x,v) = u_{\tau,m}^\lambda (x+\tau v)- (1- \tau \lambda) u_{\tau,m}^\lambda(x) \right\}.
$$

\begin{lemma} \label{need0818_2}
	For any $\tau \in (0,1)$, any $\lambda \in (0,1]$ and any $m \in \PP ( \T^d )$, we have $\MMM_\tau^\lambda (L_m) \subset \NN_\tau^\lambda \left( L_m, u_{\tau,m}^\lambda \right)$.
\end{lemma}
\begin{proof}
According to the definition of $u_{\tau,m}^\lambda$, we obtain that
$$
u_{\tau,m}^\lambda (x+ \tau v) \leq (1- \tau \lambda) u_{\tau,m}^\lambda (x) + \tau L_m(x,v), \quad \forall (x,v) \in \T^d \times \R^d.
$$
For any minimizing $\tau$-holonomic $\lambda$-measure $\mu \in \PP_\tau ( \T^d \times \R^d )$,
$$
\int_{\T^d \times \R^d} \left( \tau L_m(x,v) - u_{\tau,m}^\lambda (x+\tau v) + (1- \tau \lambda) u_{\tau,m}^\lambda (x)\right)  d\mu = \tau \int_{\T^d \times \R^d} \left( L_m(x,v) - \lambda u_{\tau,m}(x) \right) d\mu = 0.
$$
Thus, $\tau L_m(x,v) = u_{\tau,m}^\lambda (x+\tau v)- (1- \tau \lambda) u_{\tau,m}^\lambda (x)$ holds everywhere on the support of $\mu$.
\end{proof}

\begin{proposition} \label{compactness of minimizing holonomic measure}
There exists a compact set $K \subset \T^d \times \R^d$ such that $\NN_\tau^\lambda \left( L_m, u_{\tau,m}^\lambda \right) \subset K$ for any $\tau \in (0, 1)$, any $\lambda \in (0,1]$ and any $m \in \PP ( \T^d )$.
\end{proposition}
\begin{proof}
For any $(x,v) \in \NN_\tau^\lambda \left( L_m, u_{\tau,m}^\lambda \right)$,
$$
\tau L_m(x,v) = u_{\tau,m}^\lambda (x+\tau v)- (1- \tau\lambda) u_{\tau,m}^\lambda (x) \leq C \tau \vert v \vert + \tau\lambda u_{\tau,m}^\lambda (x).
$$
Thus, $L(x,v)+F(x,m) \leq C \vert v \vert + \lambda u_{\tau,m}^\lambda (x) \leq C \vert v \vert + C_0$.
At the same time, by (L\ref{MFG_L3}), there exists a constant $\hat{C}$ such that
$$
L(x,v) \geq 2C \vert v \vert + \hat{C}.
$$
Thus, $2C \vert v \vert + \hat{C} - F_\infty \leq C \vert v \vert + C_0$, indicating that
$$
\vert v \vert \leq \frac{C_0 + F_{\infty} - \hat{C}}{C} < +\infty.
$$
Let $K:= \left\{ (x,v) \in \T^d \times \R^d \middle\vert \vert v \vert \leq \frac{C_0 + F_{\infty} - \hat{C}}{C} \right\}$.
\end{proof}

\begin{remarks} \label{remark 3.3}
There exists a compact set $K$ such that for any $\tau \in (0, 1)$, any $\lambda \in (0,1]$, any $m \in \PP ( \T^d )$ and any minimizing $\tau$-holonomic $\lambda$-measure $\mu$ for $L_m$, we always have $\supp (\mu) \subset K$.
\end{remarks}

\section{Proof of Theorem \ref{theorem 1}} \label{section 4}

\subsection{Proof of Theorem \ref{theorem 1} (1)(i)} \label{proof of theorem 1(1)(i)}
Theorem \ref{theorem 1} (1)(i) is proved in Proposition \ref{existence of minimizing tau holonomic lambda measure for MFG}. Before the proof, we establish the continuity of $u_{\tau,m}^\lambda$ with respect to the measure $m$.

\begin{lemma} \label{the only limit of solution of discrete lax-oleinik}
Let $\tau \in (0, 1)$ and $\lambda \in (0,1]$. If the sequence $\left\{m_i\right\}_{i=1}^\infty \subset \PP(\T^d)$ converges to some $m_0 \in \PP ( \T^d )$ weakly, then, by taking a subsequence if necessary, the related $\Z^d$-periodic solutions $\left\{u_{\tau, m_i}^\lambda \right\}_{i=1}^\infty$ of \eqref{discrete_laxoleinik} converge to $u_{\tau, m_0}^\lambda$ uniformly.
\end{lemma}
\begin{proof}
Since the family $\left\{ u^\lambda_{\tau,m} \middle \vert m \in \PP ( \T^d ) \right\}$ is uniformly bounded and equi-Lipschitz, without loss of generality, suppose there exists a function $u_0$ such that $u_{\tau, m_i}^\lambda$ converges to $u_0$ uniformly. Note that $u_0$ is $\Z^d$-periodic.
First of all, we claim that
$$
\lim_{i \rightarrow \infty} \left\Vert \TT_{\tau, \lambda}^{m_i} u_0 - \TT_{\tau, \lambda}^{m_0} u_0 \right\Vert_{\infty} = 0.
$$
For any $y \in \R^d$ and any $x \in \argmin_{ x \in \R^d} \left\{ (1- \tau\lambda) u_0(x) + \LL_{\tau,m_0} (x,y) \right\}$, we obtain that
\begin{align*}
\TT_{\tau, \lambda}^{m_i} u_0(y) - \TT_{\tau, \lambda}^{m_0} u_0(y)
& \leq (1- \tau\lambda) u_0(x) + \LL_{\tau,m_i} (x,y) - (1- \tau\lambda) u_0(x) - \LL_{\tau,m_0} (x,y) \\
& = \tau \left( F (x, m_i) - F(x, m_0) \right) \leq \tau \Lip(F) d_1 (m_i, m_0).
\end{align*}
Similarly, we can obtain that $\TT_{\tau, \lambda}^{m_0} u_0(y) - \TT_{\tau, \lambda}^{m_i} u_0(y) \leq \tau \Lip(F) d_1 (m_i, m_0)$. 
Thus, 
$$
\lim_{i \rightarrow \infty} \left\Vert \TT_{\tau, \lambda}^{m_i} u_0 - \TT_{\tau, \lambda}^{m_0} u_0 \right\Vert_{\infty} \leq \lim_{i \rightarrow \infty} \Lip(F) d_1 (m_i, m_0) = 0.
$$

Based on this, it is clear that
$$
\lim_{i \rightarrow \infty} \left\Vert \TT_{\tau, \lambda}^{m_i} u_{\tau,m_i}^\lambda - \TT_{\tau, \lambda}^{m_0} u_0 \right\Vert_{\infty} \leq \lim_{i \rightarrow \infty} (1- \tau\lambda) \left\Vert u_{\tau,m_i}^\lambda - u_0 \right\Vert_{\infty} = 0.
$$
Thus, for any $i \in \N$,
$$
\left\Vert \TT_{\tau, \lambda}^{m_0} u_0 - u_0 \right\Vert_{\infty} \leq \left\Vert \TT_{\tau, \lambda}^{m_0} u_0 - \TT_{\tau, \lambda}^{m_i} u_{\tau,m_i}^\lambda \right\Vert_{\infty} + \left\Vert \TT_{\tau, \lambda}^{m_i} u_{\tau,m_i}^\lambda - u_{\tau,m_i}^\lambda \right\Vert_{\infty} + \left\Vert u_{\tau,m_i}^\lambda - u_0 \right\Vert_{\infty}.
$$
Let $i$ tend to infinity. We obtain that $\TT_{\tau, \lambda}^{m_0} u_0 = u_0$. 
By Proposition \ref{existence of solution of discounted HJE}, it follows that $u_0 = u_{\tau, m_0}^\lambda$.
\end{proof}

\begin{proposition} \label{existence of minimizing tau holonomic lambda measure for MFG}
For each $\tau \in (0, 1)$ and each $\lambda \in (0,1]$, there is $m \in \PP ( \T^d )$ such that there exists a minimizing $\tau$-holonomic $\lambda$-measure $\mu_{\tau,m}^\lambda$ for the Lagrangian $L_m$ with
$$
m = \pi \sharp \mu_{\tau,m}^\lambda.
$$
We call such a measure $m$ the minimizing $\tau$-holonomic $\lambda$-measure for DMFGs \eqref{DMFG} and denote it by $m_\tau^\lambda$.
\end{proposition}
\begin{proof}
For each $\tau \in (0, 1)$ and each $\lambda \in (0,1]$, define a set-valued map by
\begin{align*}
\psi: \PP ( \T^d ) & \rightarrow 2^{\PP ( \T^d )} \\
m & \mapsto \psi(m) := \left\{ \pi \sharp \mu \middle\vert\,\, \mu \,\, \text{is a minimizing $\tau$-holonomic $\lambda$-measure for}\,\, L_m \right\}.
\end{align*}
By Proposition \ref{non-empty of minimizing holonomic measure}, $\psi(m)$ is non-empty for any $m \in \PP ( \T^d )$. Moreover, $\psi(m)$ is convex for any $m \in \PP ( \T^d )$ obviously.

If $m_i \stackrel{w^*}{\longrightarrow} m$, $\eta_i \in \psi(m_i)$ and $\eta_i \stackrel{w^*}{\longrightarrow} \eta$ as $i \rightarrow \infty$, we claim that $\eta \in \psi (m)$.
Since $\eta_i \in \psi(m_i)$, there exists $\mu_i$, which is a minimizing $\tau$-holonomic $\lambda$-measure for $L_{m_i}$, such that $\eta_i = \pi \sharp \mu_i$.
By Remark \ref{remark 3.3}, there exists a measure $\mu \in \PP(K)$ such that $\mu_i \stackrel{w^*}{\longrightarrow} \mu$.
It is clear that $\eta = \pi \sharp \mu$. Consider
\begin{align*}
& \left\vert \int_{\T^d \times \R^d} \left( L_m (x,v) - u_{\tau,m}^\lambda (x) \right) d \mu - \int_{\T^d \times \R^d} \left( L_{m_i} (x,v) - u_{\tau, m_i}^\lambda(x) \right)  d \mu_i \right\vert \\
\leq & \left\vert \int_{\T^d \times \R^d} L_m (x,v) d \mu - \int_{\T^d \times \R^d} L_{m_i} (x,v) d \mu_i \right\vert  + \left\vert \int_{\T^d \times \R^d} u_{\tau,m}^\lambda (x) d \mu - \int_{\T^d \times \R^d} u_{\tau, m_i}^\lambda(x) d \mu_i \right\vert \\
\triangleq & B_1 +B_2.
\end{align*}
First of all,
$$
B_1 \leq \left\vert \int_{\T^d \times \R^d} L_m (x,v) d \mu - \int_{\T^d \times \R^d} L_{m} (x,v) d \mu_i \right\vert + \left\vert \int_{\T^d \times \R^d} \left( L_m (x,v) - L_{m_i} (x,v) \right) d \mu_i \right\vert \triangleq B_3 + B_4.
$$
Since $\mu_i \stackrel{w^*}{\longrightarrow} \mu$ and $\mu_i, \mu \in \PP(K)$, we obtain that $\lim_{i \rightarrow \infty} B_3 = 0$. By (F\ref{MFG_F3}), it is clear that
$$
B_4 \leq \int_{\T^d \times \R^d} \Lip(F) d_1 \left( m, m_i \right) d \mu_i = \Lip(F) d_1 \left( m, m_i \right) \rightarrow 0, \quad \text{as}\,\, i \rightarrow \infty.
$$
Moreover, 
$$
B_2 \leq \left\vert \int_{\T^d} \left( u_{\tau,m_i}^\lambda (x) - u_{\tau, m}^\lambda(x) \right) d \eta_i \right\vert + \left\vert \int_{\T^d} u_{\tau,m}^\lambda (x) d \eta_i - \int_{\T^d} u_{\tau, m}^\lambda(x) d \eta \right\vert \triangleq B_5 +B_6.
$$
By Lemma \ref{the only limit of solution of discrete lax-oleinik}, $\lim_{i \rightarrow \infty} B_5 = 0$. Since $\eta_i \stackrel{w^*}{\longrightarrow} \eta$, $\lim_{i \rightarrow \infty} B_6 = 0$.
Above all, we obtain that
$$
\int_{\T^d \times \R^d} \left( L_m (x,v) - u_{\tau,m}^\lambda (x) \right) d \mu = 0.
$$
For any $\varphi \in \CC^1 ( \T^d )$, since $\mu_i \stackrel{w^*}{\longrightarrow} \mu$ and $\mu_i, \mu \in \PP(K)$, we obtain that
$$
\int_{\T^d \times \R^d} \left( \varphi (x+\tau v) - \varphi (x) \right) d \mu = \lim_{i \rightarrow \infty} \int_{\T^d \times \R^d} \left( \varphi (x+\tau v) - \varphi (x) \right) d \mu_i = 0.
$$
If there exists a sequence $([x_i], v_i) \in \tilde{\A}_{L_{m_i}, \lambda}^\tau$ converging to $([x], v) \in \T^d \times \R^d$ as $i \rightarrow \infty$, then for any $n \in \N$, we have $([x_i], v_i) \in \Psi_{L_{m_i},\lambda, \tau}^n \left( \tilde{\Sigma}_{L_{m_i}, \lambda}^\tau \right)$. Without loss of generality, assume that for any $i \in \N$, there exists a sequence of discounted calibrated configuration $\left\{ x^i_{-k} \right\}_{k=0}^\infty$ of $u^\lambda_{\tau, m_i} (x_0^i)$ such that $x^i_{-n-1} \in [0,1]^d$ and $\Pi \left( x^i_{-n-1} \right) = [x_i]$. By Proposition \ref{minimizer x is near}, the set
$$
\left\{ x^i_{-k} \middle\vert i \in \N, 0 \leq k \leq n+1 \right\}
$$
belongs to a compact set. Thus, there exists $x_{-k} \in \R^d$ such that $x^i_{-k} \rightarrow x_{-k}$ for any $0 \leq k \leq n+1$ as $i \rightarrow \infty$. Note that $\Pi (x_0) = [x]$.
Since
$$
u^\lambda_{\tau, m_i} (x^i_0) = \Sigma_{k =0}^n \left( 1-\tau\lambda \right)^k \LL_{\tau,m_i} \left( x^i_{-k-1}, x^i_{-k} \right) + \left( 1-\tau\lambda \right)^{n+1} u^\lambda_{\tau, m_i} (x_i),
$$
by Lemma \ref{the only limit of solution of discrete lax-oleinik}, (L\ref{MFG_L1}) and (F\ref{MFG_F2}), it follows that
$$
u^\lambda_{\tau, m} (x_0) = \Sigma_{k =0}^n \left( 1-\tau\lambda \right)^k \LL_{\tau,m} \left( x_{-k-1}, x_{-k} \right) + \left( 1-\tau\lambda \right)^{n+1} u^\lambda_{\tau, m} (x).
$$
Therefore, we can conclude that $([x], v) \in \Psi_{L_{m},\lambda, \tau}^n \left( \tilde{\Sigma}_{L_m, \lambda}^\tau \right)$. By the arbitrariness of $n$, we obtain that $([x], v) \in \tilde{\A}_{L_m, \lambda}^\tau$, and hence $\supp(\mu) \subset \tilde{\A}_{L_m, \lambda}^\tau$.

In conclusion, the measure $\mu$ is a minimizing $\tau$-holonomic $\lambda$-measure for $L_m$.
We conclude that the set-valued map $\psi$ has a closed graph, indicating that $\psi(m)$ is closed for any $m \in \PP ( \T^d )$.
By Kakutani's theorem, the proof is complete.
\end{proof}

\subsection{Proof of Theorem \ref{theorem 1} (1)(ii)} \label{proof of theorem 1(1)(ii)}
Theorem \ref{theorem 1} (1)(ii) is established with the help of Proposition \ref{convergence to find viscosity solution} and Proposition \ref{m_0 is solution of continuity equation}.

For each $\tau \in (0,1)$ and each $\lambda \in (0,1]$, consider solutions to the discrete Lax-Oleinik equation for (\ref{DMFG}a)
\begin{equation} \label{discrete back}
	u_{\tau, m_{\tau}^\lambda}^\lambda (y) = \inf_{x \in \R^d} \left( (1- \tau\lambda) u_{\tau, m_{\tau}^\lambda}^\lambda (x) + \LL_{\tau, m_\tau^\lambda} (x,y)\right) , \quad \forall y \in \R^d.
\end{equation}

\begin{proposition} \label{convergence to find viscosity solution}
Fix $\lambda \in (0,1]$. There is a sequence $\tau_i \rightarrow 0$ as $i \rightarrow \infty$ such that $m_{\tau_i}^\lambda \stackrel{w^*}{\longrightarrow} m_0^\lambda$ and $u_{\tau_i, m_{\tau_i}^\lambda}^\lambda$ converges to $u_0^\lambda$ uniformly, where $u_{\tau_i, m_{\tau_i}^\lambda}^\lambda$ is the solution to \eqref{discrete back}. Moreover, $u_0^\lambda$ is a viscosity solution to
\begin{equation} \label{has F HJE}
\lambda u+H(x,Du) = F(x,m_0^\lambda).
\end{equation}
\end{proposition}
\begin{proof}
Recall two Lax-Oleinik operators defined as in \eqref{DMFG laxoleinik} and \eqref{DMFG laxoleinik_discrete}.
By the results in Sect. \ref{Priori estimates}, there exist constants $C,D>0$ such that for any $\tau \in (0,1)$, any $\lambda \in (0,1]$ and any $m \in  \PP ( \T^d )$, we have
\begin{enumerate}[(i)]
\item $\Lip \left(u_{\tau,m}^\lambda \right) \leq C$, $\left\Vert u_{\tau,m}^\lambda \right\Vert_\infty \leq \frac{C}{\lambda}$.
\item For any $y \in \R^d$ and any $x \in \argmin_{x \in \R^d} \left\{ (1- \tau\lambda) u_{\tau,m}^\lambda(x) + \LL_{\tau,m}(x,y) \right\}$, we have $\vert x-y \vert \leq \tau D$. Moreover,
$$
\left\Vert \TT_{\tau,\lambda}^m u_{\tau,m}^\lambda - u_{\tau,m}^\lambda \right\Vert_{\infty} \leq \tau C.
$$
\item For any $y \in \R^d$ and any $x \in \argmin_{x \in \R^d} \left\{ e^{-\lambda \tau} u_{\tau,m}^\lambda (x) + h_{\tau,\lambda}^m(x,y) \right\}$, we have $\vert x-y \vert \leq \tau D$. Moreover,
$$
\left\Vert T_{\tau,\lambda}^m u_{\tau,m}^\lambda - u_{\tau,m}^\lambda \right\Vert_{\infty} \leq \tau C.
$$

\item For any $x,y \in \R^d$ with $\vert x-y \vert \leq \tau D$, we have
$$
\left\vert h_{\tau,\lambda}^m (x,y) - \LL_{\tau,m} (x,y) \right\vert \leq \tau^2 C.
$$
\end{enumerate}

For any $y \in \R^d$ and any $x \in \argmin_{x \in \R^d} \left\{ (1- \tau\lambda) u_{\tau,m}^\lambda (x) + \LL_{\tau,m}(x,y) \right\}$, we obtain that
\begin{align*}
T_{\tau,\lambda}^m u_{\tau,m}^\lambda(y) 
& \leq e^{-\lambda \tau} u_{\tau,m}^\lambda (x) + h_{\tau,\lambda}^m (x,y) \\
& \leq (1- \tau\lambda) u_{\tau,m}^\lambda(x) + \LL_{\tau,m}(x,y) + \tau^2 C + \left( e^{- \lambda \tau}-1+ \tau\lambda \right) u_{\tau,m}^\lambda(x) \\
& \leq \TT_{\tau,\lambda}^m u_{\tau,m}^\lambda(y) + \tau^2 C + \left( e^{-\lambda \tau}-1+ \tau\lambda \right) \frac{C}{\lambda} \leq \TT_{\tau,\lambda}^m u_{\tau,m}^\lambda(y) + 2 \tau^2 C.
\end{align*}
Similarly, for any $y \in \R^d$ and any $x \in \argmin_{x \in \R^d} \left\{ e^{- \lambda \tau} u_{\tau,m}^\lambda(x) + h_{\tau,\lambda}^m(x,y) \right\}$, we obtain that
\begin{align*}
\TT_{\tau,\lambda}^m u_{\tau,m}^\lambda(y)
& \leq (1- \tau\lambda) u_{\tau,m}^\lambda(x) + \LL_{\tau,m}(x,y) \\
& \leq e^{- \lambda\tau} u_{\tau,m}^\lambda(x) + h_{\tau,\lambda}^m (x,y) + \tau^2 C + (1- \tau\lambda -e^{- \lambda \tau}) u_{\tau,m}^\lambda(x) \\
& \leq T_{\tau,\lambda}^m u_{\tau,m}^\lambda(y) + 2 \tau^2 C.
\end{align*}
Thus, we conclude that $\left\Vert T_{\tau,\lambda}^m u_{\tau,m}^\lambda - \TT_{\tau,\lambda}^m u_{\tau,m}^\lambda \right\Vert_\infty \leq 2 \tau^2 C$.

For any $\tau \in (0, 1)$, any $\lambda \in (0,1]$ and any $m \in \PP ( \T^d )$, $u_{\tau,m}^\lambda$ is uniformly bounded by $\frac{C}{\lambda}$ and equi-Lipschitz with Lipschitz constant $C$. 
Since $\PP ( \T^d )$ is compact, there exists a sequence $\tau_i \rightarrow 0$ such that $m_{\tau_i}^\lambda \stackrel{w^*}{\longrightarrow} m_0^\lambda$ for some measure $m_0^\lambda$. By Arz\'ela-Ascoli theorem, taking a subsequence if necessary, there exists a Lipschitz function $u_0^\lambda$ such that $u_{\tau_i, m_{\tau_i}^\lambda}^\lambda \rightarrow u_0^\lambda$ uniformly.

For any $t>0$, choose $N_i \in \N$ such that $N_i \tau_i \leq t < (N_i +1) \tau_i$.
First of all, we claim that $\left\Vert T_{N_i \tau_i, \lambda}^{m_{\tau_i}^\lambda} u_{\tau_i, m_{\tau_i}^\lambda}^\lambda - T_{t, \lambda}^{m_0^\lambda} u_0^\lambda \right\Vert_{\infty} \rightarrow 0$ as $i \rightarrow \infty$.
\begin{align*}
& \left\Vert T_{N_i \tau_i, \lambda}^{m_{\tau_i}^\lambda} u_{\tau_i, m_{\tau_i}^\lambda}^\lambda - T_{t, \lambda}^{m_0^\lambda} u_0^\lambda \right\Vert_{\infty} \\
\leq & \left\Vert T_{N_i \tau_i, \lambda}^{m_{\tau_i}^\lambda} u_{\tau_i, m_{\tau_i}^\lambda}^\lambda - T_{N_i \tau_i, \lambda}^{m_0^\lambda} u_{\tau_i, m_{\tau_i}^\lambda}^\lambda \right\Vert_{\infty} + \left\Vert T_{N_i \tau_i, \lambda}^{m_0^\lambda} u_{\tau_i, m_{\tau_i}^\lambda}^\lambda - T_{t, \lambda}^{m_0^\lambda} u_0^\lambda \right\Vert_{\infty} \triangleq D_1+D_2.
\end{align*}
On the one hand,
\begin{align*}
D_2 
& \leq \left\Vert T_{t- N_i \tau_i, \lambda}^{m_0^\lambda} u_0^\lambda - u_{\tau_i, m_{\tau_i}^\lambda}^\lambda \right\Vert_{\infty} \leq \left\Vert T_{t- N_i \tau_i, \lambda}^{m_0^\lambda} u_0^\lambda - u_0^\lambda \right\Vert_{\infty} + \left\Vert u_0^\lambda - u_{\tau_i, m_{\tau_i}^\lambda}^\lambda \right\Vert_{\infty} \\
& \leq C \left\vert t- N_i \tau_i \right\vert + \left\Vert u_0^\lambda - u_{\tau_i, m_{\tau_i}^\lambda}^\lambda \right\Vert_{\infty} \rightarrow 0.
\end{align*}
On the other hand, for any $y \in \R^d$, any $x \in \argmin_{x \in \R^d} \left\{ e^{- \lambda N_i \tau_i} u_{\tau_i, m_{\tau_i}^\lambda}^\lambda (x) + h_{N_i \tau_i, \lambda}^{m_{\tau_i}^\lambda} (x,y) \right\}$, and any minimizer $\gamma$ for $h_{N_i \tau_i, \lambda}^{m_{\tau_i}^\lambda} (x,y)$, we obtain that
\begin{align*}
& T_{N_i \tau_i, \lambda}^{m_0^\lambda} u_{\tau_i, m_{\tau_i}^\lambda}^\lambda(y) - T_{N_i \tau_i, \lambda}^{m_{\tau_i}^\lambda} u_{\tau_i, m_{\tau_i}^\lambda}^\lambda (y) \\
\leq &  e^{-\lambda N_i \tau_i} u_{\tau_i, m_{\tau_i}^\lambda}^\lambda (x) + h_{N_i \tau_i, \lambda}^{m_0^\lambda} (x,y) - e^{-\lambda N_i \tau_i} u_{\tau_i, m_{\tau_i}^\lambda}^\lambda (x) - h_{N_i \tau_i, \lambda}^{m_{\tau_i}^\lambda} (x,y) \\
\leq & \int_{- N_i \tau_i}^0 e^{\lambda s} \left( L_{m_0^\lambda} \left( \gamma(s), \dot{\gamma} (s) \right) - L_{m_{\tau_i}^\lambda} \left( \gamma(s), \dot{\gamma} (s) \right) \right) ds \\
\leq & \frac{1- e^{- \lambda N_i \tau_i}}{\lambda} \Lip(F) d_1 \left( m_{\tau_i}^\lambda, m_0^\lambda \right).
\end{align*}
Similarly, we have
$$
T_{N_i \tau_i, \lambda}^{m_{\tau_i}^\lambda} u_{\tau_i, m_{\tau_i}^\lambda}^\lambda (y) - T_{N_i \tau_i, \lambda}^{m_0^\lambda} u_{\tau_i, m_{\tau_i}^\lambda}^\lambda(y) \leq \frac{1- e^{- \lambda N_i \tau_i}}{\lambda} \Lip(F) d_1 \left( m_{\tau_i^\lambda}, m_0^\lambda \right).
$$
Thus, we conclude that $\lim_{i \rightarrow \infty} D_1 =0$.
Above all, $\lim_{i \rightarrow \infty} \left\Vert T_{N_i \tau_i, \lambda}^{m_{\tau_i}^\lambda} u_{\tau_i, m_{\tau_i}^\lambda}^\lambda - T_{t, \lambda}^{m_0^\lambda} u_0^\lambda \right\Vert_{\infty} = 0$.

Furthermore, it is clear that
\begin{align*}
& \left\Vert T_{N_i \tau_i, \lambda}^{m_{\tau_i}^\lambda} u_{\tau_i, m_{\tau_i}^\lambda}^\lambda - u_{\tau_i, m_{\tau_i}^\lambda}^\lambda \right\Vert_{\infty} \\
\leq & \left\Vert T_{N_i \tau_i, \lambda}^{m_{\tau_i}^\lambda} u_{\tau_i, m_{\tau_i}^\lambda}^\lambda - T_{(N_i -1) \tau_i, \lambda}^{m_{\tau_i}^\lambda} u_{\tau_i, m_{\tau_i}^\lambda}^\lambda \right\Vert_{\infty} + \cdots + \left\Vert T_{\tau_i, \lambda}^{m_{\tau_i}^\lambda} u_{\tau_i, m_{\tau_i}^\lambda}^\lambda - u_{\tau_i, m_{\tau_i}^\lambda}^\lambda \right\Vert_{\infty} \\
\leq & N_i \left\Vert T_{\tau_i,\lambda}^{m_{\tau_i}^\lambda} u_{\tau_i, m_{\tau_i}^\lambda}^\lambda - u_{\tau_i, m_{\tau_i}^\lambda}^\lambda \right\Vert_{\infty} = N_i \left\Vert T_{\tau_i, \lambda}^{m_{\tau_i}^\lambda} u_{\tau_i, m_{\tau_i}^\lambda}^\lambda - \TT_{\tau_i, \lambda}^{m_{\tau_i}^\lambda} u_{\tau_i, m_{\tau_i}^\lambda}^\lambda \right\Vert_{\infty} \leq 2 N_i \tau_i^2 C \leq 2 t \tau_i C,
\end{align*}
indicating that $\lim_{i \rightarrow \infty} \left\Vert T_{N_i \tau_i, \lambda}^{m_{\tau_i}^\lambda} u_{\tau_i, m_{\tau_i}^\lambda}^\lambda - u_{\tau_i, m_{\tau_i}^\lambda}^\lambda \right\Vert_{\infty} =0$.

In conclusion, for any $i \in \N$, we have
$$
\left\Vert T_{t, \lambda}^{m_0^\lambda} u_0^\lambda -  u_0^\lambda \right\Vert_{\infty} \leq  \left\Vert T_{N_i \tau_i, \lambda}^{m_{\tau_i}^\lambda} u_{\tau_i, m_{\tau_i}^\lambda}^\lambda - T_{t, \lambda}^{m_0^\lambda} u_0^\lambda \right\Vert_{\infty} + \left\Vert T_{N_i \tau_i, \lambda}^{m_{\tau_i}^\lambda} u_{\tau_i, m_{\tau_i}^\lambda}^\lambda - u_{\tau_i, m_{\tau_i}^\lambda}^\lambda \right\Vert_{\infty} + \left\Vert u_{\tau_i, m_{\tau_i}^\lambda}^\lambda -u_0^\lambda \right\Vert_{\infty}.
$$
Let $i$ tend to infinity. We conclude that $T_{t, \lambda}^{m_0^\lambda} u_0^\lambda = u_0^\lambda$, which implies that $u_0^\lambda$ is the viscosity solution to \eqref{has F HJE}.
\end{proof}

The following lemma reveals the relation between discrete discounted Aubry sets and discounted Aubry sets. Since the proof is quite long, we leave the proof in Appendix \ref{proof of Lemma 4.1}.
\begin{lemma} \label{convergence of aubry set}
	Fix $\lambda \in (0,1]$. Define $\tau_i$, $m_{\tau_i}^\lambda$, $u_{\tau_i, m_{\tau_i}^\lambda}^\lambda$, $m_0^\lambda$ and $u_0^\lambda$ as in Proposition \ref{convergence to find viscosity solution}.
	For any sequence $x_i \in \R^d$ converging to some $x_0 \in \R^d$, denote by the sequence $\left\{x_{n, x_i}^{\tau_i, \lambda ,m_{\tau_i}^\lambda }\right\}_{n\leq 0}$ with $x_{0,x_i}^{\tau_i, \lambda, m_{\tau_i}^\lambda} = x_i$ the discounted calibrated configuration of $u_{\tau_i,m_{\tau_i}^\lambda}^\lambda (x_i)$.
	Let $\gamma_{\tau_i,m_{\tau_i}^\lambda}^{\lambda,x_i} (t)$ be the piecewise linear approximation satisfying $\gamma_{\tau_i,m_{\tau_i}^\lambda}^{\lambda,x_i} (n \tau_i) = x_{n,x_i}^{\tau_i, \lambda, m_{\tau_i}^\lambda}$. Then, 
	\begin{enumerate}[(i)]
		\item there exists a curve $\gamma_0: (-\infty, 0] \rightarrow \R^d$ such that $\gamma_{\tau_i , m_{\tau_i}^\lambda}^{\lambda,x_i} \rightarrow \gamma_0$ uniformly and $\dot{\gamma}_{\tau_i, m_{\tau_i}^\lambda}^{\lambda, x_i} \rightarrow \dot{\gamma}_0$ in $L^1$-norm on every compact subset of $(-\infty, 0]$,
		\item there exists a constant $C$ such that the curve $\gamma_0 \in \CC^2 \left( (-\infty, 0], \R^d \right)$ and satisfies $\left\Vert \dot{\gamma}_0 \right\Vert_\infty \leq C$, $\Lip \left( \dot{\gamma}_0 \right) \leq C$,
		\item for any $t \geq 0$, we obtain that $u^\lambda_{m_0^\lambda} (x_0) = e^{- \lambda t} u^\lambda_{m_0^\lambda} \left( \gamma_0 (-t) \right) + \int_{-t}^0 e^{\lambda s} L_{m_0^\lambda} \left( \gamma_0 (s), \dot{\gamma}_0 (s) \right) ds$, where $u^\lambda_{m_0^\lambda}$ is defined as in \eqref{equation of solution of HJE_MFG}.
	\end{enumerate}
\end{lemma}

\begin{remarks} \label{remark after convergence of aubry set}
Fix $\lambda \in (0,1]$. Define a sequence $\tau_i \rightarrow 0$ as in Proposition \ref{convergence to find viscosity solution}.
Then Lemma \ref{convergence of aubry set} implies that
$$
\limsup_{i \rightarrow \infty} \tilde{\A}_{L_{m_{\tau_i}^\lambda, \lambda}}^{\tau_i} := \bigcap_{i=0}^\infty \bigcup_{k=i}^\infty \tilde{\A}_{L_{m_{\tau_k}^\lambda, \lambda}}^{\tau_k} \subset \tilde{\A}_{L_{m_0^\lambda}, \lambda}. 
$$
\end{remarks}

\begin{proposition} \label{u_0 is solution of HJE}
Fix $\lambda \in (0,1]$. There is a measure $\mu_0^\lambda \in \PP ( \T^d \times \R^d )$ such that $\mu_{\tau_i, m_{\tau_i}^\lambda}^\lambda$ converges to $\mu_0^\lambda$ weakly with $m_0^\lambda = \pi \sharp \mu_0^\lambda$, where $\tau_i$, $m_{\tau_i}^\lambda$ and $m_0^\lambda$ are defined as in Proposition \ref{convergence to find viscosity solution}, and $\mu_{\tau_i, m_{\tau_i}^\lambda}^\lambda$ is defined as in Proposition \ref{existence of minimizing tau holonomic lambda measure for MFG}.
Moreover, $\mu_0^\lambda$ is a minimizing $\lambda$-measure for $L_{m_0^\lambda}$.
\end{proposition}
\begin{proof}
By Remark \ref{remark 3.3}, $\mu_{\tau_i, m_{\tau_i}^\lambda}^\lambda \in \PP(K)$ for any $i \in \N$. Since $\PP(K)$ is compact, there exists a measure $\mu_0^\lambda \in \PP ( \T^d \times \R^d )$ such that $\mu_{\tau_i, m_{\tau_i}^\lambda}^\lambda \stackrel{w^*}{\longrightarrow} \mu_0^\lambda$ as $i \rightarrow +\infty$. Clearly, $m_0^\lambda = \pi \sharp \mu_0^\lambda$.

Since both $u_0^\lambda$ and $u_{m_0^\lambda}^\lambda$, which is defined as in \eqref{equation of solution of HJE_MFG}, are viscosity solutions to \eqref{has F HJE}, it follows that $u_0^\lambda = u_{m_0^\lambda}^\lambda$ by the uniqueness of viscosity solutions.
Thus,
\begin{align*}
& \left\vert \int_{\T^d \times \R^d} \left( L_{m_{\tau_i}^\lambda}(x,v) - \lambda u_{\tau_i, m_{\tau_i}^\lambda}^\lambda(x) \right) d \mu_{\tau_i, m_{\tau_i}^\lambda}^\lambda - \int_{\T^d \times \R^d} \left( L_{m_0^\lambda} (x,v) - \lambda u_0^\lambda (x) \right) d\mu_0^\lambda \right\vert \\
\leq & \int_{\T^d \times \R^d} \left\vert L_{m_{\tau_i}^\lambda}(x,v) - \lambda u_{\tau_i, m_{\tau_i}^\lambda}^\lambda(x) - L_{m_0^\lambda} (x,v) + \lambda u_0^\lambda (x) \right\vert  d \mu_{\tau_i, m_{\tau_i}^\lambda}^\lambda \\
& + \left\vert \int_{\T^d \times \R^d} \left(  L_{m_0} (x,v) - \lambda u_0^\lambda (x) \right)  d \mu_{\tau_i, m_{\tau_i}^\lambda}^\lambda - \int_{\T^d \times \R^d} \left( L_{m_0^\lambda} (x,v) - \lambda u_0^\lambda (x) \right)  d\mu_0^\lambda \right\vert \\
\triangleq & E_1 + E_2.
\end{align*}
By Remark \ref{remark 3.3} and $\mu_{\tau_i, m_{\tau_i}^\lambda}^\lambda \stackrel{w^*}{\longrightarrow} \mu_0^\lambda$, we have $\lim_{i \rightarrow \infty} E_2 = 0$. Since
\begin{align*}
& \left\vert L_{m_{\tau_i}^\lambda}(x,v) - \lambda u_{\tau_i, m_{\tau_i}^\lambda}^\lambda (x) - L_{m_0^\lambda} (x,v) + \lambda u_0^\lambda (x) \right\vert \\
\leq & \left\vert F \left( x, m_{\tau_i}^\lambda \right) - F \left( x, m_0^\lambda \right) \right\vert + \lambda \left\vert u_{\tau_i, m_{\tau_i}^\lambda}^\lambda (x) - u_0^\lambda (x) \right\vert \\
\leq & \Lip(F) d_1 \left( m_{\tau_i}^\lambda, m_0^\lambda \right) + \lambda \left\Vert u_{\tau_i, m_{\tau_i}^\lambda}^\lambda - u_0^\lambda \right\Vert_\infty,
\end{align*}
we obtain that
$$
\lim_{i \rightarrow \infty} E_1 = \lim_{i \rightarrow \infty} \Lip(F) d_1 \left( m_{\tau_i}^\lambda, m_0^\lambda \right) + \lambda \left\Vert u_{\tau_i, m_{\tau_i}^\lambda}^\lambda - u_0^\lambda \right\Vert_\infty = 0.
$$
Above all, we conclude that $\int_{\T^d \times \R^d} \left(  L_{m_0^\lambda}(x,v) -\lambda u_0^\lambda(x) \right) d \mu_0^\lambda = 0$.

For any $\tau>0$ and any $\varphi \in \CC^1 ( \T^d )$, define
$$
\Delta \varphi_\tau (x,v) := \frac{\varphi(x+ \tau v) - \varphi(x)}{\tau}.
$$
It is clear that
$$
\lim_{\tau \rightarrow 0} \Delta \varphi_\tau (x,v) = v D\varphi(x), \quad \forall (x,v) \in \T^d \times \R^d.
$$
In fact, for any compact subset $K^\prime \subset \T^d \times \R^d$, one can deduce that 
\begin{equation}
\lim_{\tau \rightarrow 0} \Delta \varphi_\tau (x,v) = v D\varphi(x)
\end{equation}
uniformly on $K^\prime$.
Note that
\begin{align*}
& \left\vert \int_{\T^d \times \R^d} v D\varphi(x) d \mu_0^\lambda \right\vert = \left\vert \int_{\T^d \times \R^d} \Delta \varphi_{\tau_i}(x,v) d \mu_{\tau_i, m_{\tau_i}}^\lambda - \int_{\T^d \times \R^d} v D\varphi(x) d \mu_0^\lambda \right\vert \\
\leq & \left\vert \int_{\T^d \times \R^d} \left(  \Delta \varphi_{\tau_i}(x,v) - v D\varphi(x) \right)  d \mu_{\tau_i, m_{\tau_i}^\lambda}^\lambda \right\vert + \left\vert \int_{\T^d \times \R^d} v D\varphi(x) d \mu_{\tau_i, m_{\tau_i}^\lambda}^\lambda - \int_{\T^d \times \R^d} v D\varphi(x) d \mu_0^\lambda \right\vert.
\end{align*}
By Remark \ref{remark 3.3} and $\mu_{\tau_i, m_{\tau_i}^\lambda}^\lambda \stackrel{w^*}{\longrightarrow} \mu_0^\lambda$, we obtain that
$$
\int_{\T^d \times \R^d} v D\varphi(x) d \mu_0^\lambda = 0.
$$
Finally, by Remark \ref{remark after convergence of aubry set} and $d_1 \left( \mu_{\tau_i, m_{\tau_i}^\lambda}^\lambda, \mu_0^\lambda \right) \rightarrow 0$, it follows that the measure $\mu_0^\lambda$ is a minimizing $\lambda$-measure for $L_{m_0^\lambda}$.
\end{proof}

\begin{proposition} \label{m_0 is solution of continuity equation}
For any $\lambda \in (0,1]$, the measure $m_0^\lambda$ obtained in Proposition \ref{convergence to find viscosity solution} is a solution to the continuity equation (\ref{DMFG}b) in the sense of distributions.
\end{proposition}
\begin{proof}
For any $x \in \supp (m_0^\lambda)$ and any $v \in \R^d$, we have
$$
L_{m_0^\lambda} (x,v) + H_{m_0^\lambda} \left(x, Du_0^\lambda(x) \right) = L_{m_0^\lambda} (x,v) - \lambda u_0^\lambda (x). 
$$
Integrating both sides with respect to $\mu_0^\lambda$, we obtain
$$
\int_{\T^d \times \R^d} \left( L_{m_0^\lambda} (x,v) + H_{m_0^\lambda} \left(x, Du_0^\lambda(x) \right) \right) d \mu_0^\lambda = 0.
$$

Since $\supp(\mu_0^\lambda) \subset \tilde{\A}_{L_{m_0^\lambda}, \lambda}$, by \cite[Section 4]{MR3663623}, it is clear that
$$
\int_{\T^d \times \R^d} \left( L_{m_0^\lambda} (x,v) + H_{m_0^\lambda} \left(x, Du_0^\lambda(x) \right) - \left\langle Du_0^\lambda(x) , v \right\rangle \right) d \mu_0^\lambda = 0.
$$
Adding with Fenchel inequality
$$
L_{m_0^\lambda} (x,v) + H_{m_0^\lambda} \left(x, Du_0^\lambda(x) \right) - \left\langle Du_0^\lambda(x), v \right\rangle \geq 0, \quad \forall x \in \T^d, \forall v \in \R^d,
$$
we obtain that for any $(x,v) \in \supp(\mu_0^\lambda)$,
$$
L_{m_0^\lambda} (x,v) + H_{m_0^\lambda} \left(x, Du_0^\lambda(x) \right) = \left\langle Du_0^\lambda(x), v \right\rangle.
$$
Thus, by Legendre transform, for any $x \in \T^d$, there exists a unique $v = \frac{\partial H_{m_0^\lambda}}{\partial p} \left(x, Du_0^\lambda(x) \right)$ such that $(x,v) \in \supp(\mu_0^\lambda)$.

Thus, for any $f \in \CC^\infty ( \T^d )$, we have
\begin{align*}
0 & = \int_{\T^d \times \R^d} \left\langle Df(x), v \right\rangle d \mu_0^\lambda = \int_{\T^d} \left\langle Df(x), \frac{\partial H_{m_0^\lambda}}{\partial p} \left(x, Du_0^\lambda(x) \right) \right\rangle d m_0^\lambda \\
& = \int_{\T^d} \left\langle Df(x), \frac{\partial H}{\partial p} \left(x, Du_0^\lambda(x) \right) \right\rangle d m_0^\lambda,
\end{align*}
indicating that $m_0^\lambda$ is a weak solution to the continuity equation (\ref{DMFG}b).
\end{proof}

\subsection{Proof of Theorem \ref{theorem 1} (2)} \label{proof of theorem 1(2)}
Theorem \ref{theorem 1} (2) is established with the help of Proposition \ref{convergence2_prop1} and Proposition \ref{need0806}. Before the proof, it is necessary to show that the term $\left\Vert u^\lambda_m + \frac{c(m)}{\lambda} \right\Vert_\infty$ is uniformly bounded, which allows us to extract a convergent subsequence via Ascoli-Arz\'ela theorem.

\begin{proposition} \label{need0813}
There exists a constant $C$ such that for any $\lambda \in (0,1]$ and any $m \in \PP ( \T^d )$, we have
$$
\left\Vert u^\lambda_m + \frac{c(m)}{\lambda} \right\Vert_\infty \leq C.
$$
\end{proposition}
\begin{proof}
Denote $\bar{u}^\lambda_m = u^\lambda_m + \frac{c(m)}{\lambda}$. Then $\bar{u}^\lambda_m$ is a viscosity solution to
$$
\lambda u + H_m (x, Du) = c(m), \quad \forall x \in \T^d.
$$
Let $\mu_0 \in \K ( \T^d \times \R^d )$ and $c[0] \in \R$ such that
$$
- c[0] = \int_{\T^d \times \R^d} L(x,v) d \mu_0.
$$
Then we obtain that
$$
- c(m) \leq \int_{\T^d \times \R^d} L_m(x,v) d \mu_0 \leq -c[0] + F_\infty.
$$
By (L\ref{MFG_L3}), there exists a constant $\hat{C}$ such that $L(x,v) \geq \hat{C}$ for any $(x,v) \in \T^d \times \R^d$.
Thus, for any $\mu \in \K ( \T^d \times \R^d )$, it is clear that
$$
\int_{\T^d \times \R^d} L_m(x,v) d \mu \geq \hat{C}-F_\infty.
$$
Thus, $- c(m) \geq \hat{C}-F_\infty$, and we conclude that $c(m)$ is bounded.
Let $u_m$ be a viscosity solution to
\begin{equation} \label{need0106_2}
	H_m (x, Du) = c(m), \quad \forall x \in \T^d
\end{equation}
with $u_m(0) = 0$.
By \cite[Proposition 4.2.1]{dashu}, since all viscosity solutions to (\ref{need0106_2}) are equi-Lipschitz with Lipschitz constant $\sup_{x \in \T^d, \vert v \vert = 1} \left\vert L(x,v) \right\vert + F_\infty + c(m)$ and $c(m)$ is bounded, we can obtain that there exists a constant $\tilde{C}$ such that $\left\Vert u_m \right\Vert_\infty \leq \tilde{C}$ for any $m \in \PP ( \T^d )$.
By adding suitable constants, we obtain two viscosity solutions $u_m^1$ and $u_m^2$ to \eqref{need0106_2} satisfying $-2 \tilde{C} \leq u_m^1 \leq 0 \leq u_m^2 \leq 2 \tilde{C}$.
Thus, for any $\lambda \in (0,1]$,
$$
\lambda u_m^1 + H_m (x, D u_m^1) \leq c(m),
$$
$$
\lambda u_m^2 + H_m (x, D u_m^2) \geq c(m).
$$
By the comparison principle, it follows that $u_m^1 \leq \bar{u}_m^\lambda \leq u_m^2$ for any $\lambda \in (0,1]$. Therefore, $\left\Vert \bar{u}_m^\lambda \right\Vert_\infty \leq 2\tilde{C}$, where the constant $\tilde{C}$ is independent of $\lambda$ and $m$.
\end{proof}

\begin{proposition} \label{convergence2_prop1}
	Define $m_0^\lambda$ as in Proposition \ref{u_0 is solution of HJE} for any $\lambda \in (0,1]$. There exists a sequence $\lambda_i \rightarrow 0$ such that $m_0^{\lambda_i} \stackrel{w^*}{\longrightarrow} m_0$ for some measure $m_0 \in \PP ( \T^d )$ and $u_{m_0^{\lambda_i}}^{\lambda_i} + \frac{c \left( m_0^{\lambda_i} \right)}{\lambda_i} \rightarrow u_0$ uniformly for some Lipschitz function $u_0$. Moreover, $u_0$ is a viscosity solution to
	\begin{equation} \label{need0705_1}
		H_{m_0} (x,Du) = -c(m_0), \quad \forall x \in \T^d.
	\end{equation}
\end{proposition}
\begin{proof}
	By Proposition \ref{convergence to find viscosity solution}, for any $\lambda \in (0,1]$, $u_{m_0^\lambda}^{\lambda}$ is equi-Lipschitz.	
	Since $\PP ( \T^d )$ is compact, there exists a sequence $\lambda_i \rightarrow 0$ such that $m_0^{\lambda_i} \stackrel{w^*}{\longrightarrow} m_0$ for some measure $m_0 \in \PP ( \T^d )$.
	Moreover, by Ascoli-Arz\'ela theorem and Proposition \ref{need0813}, taking a subsequence if necessary, we can assume there is a function $u_0$ such that $u_{m_0^{\lambda_i}}^{\lambda_i} + \frac{c \left(m_0^{\lambda_i} \right)}{\lambda_i} \rightarrow u_0$ uniformly.

	For any $\mu \in \K ( \T^d \times \R^d )$ such that $- c(m_0) = \int_{\T^d \times \R^d} L_{m_0} (x,v) d \mu$, we obtain that
	$$
	c(m_0) - c(m_0^{\lambda_i}) \leq \int_{\T^d \times \R^d} \left( L_{m_0^{\lambda_i}} (x,v) - L_{m_0} (x,v) \right) d \mu \leq \Lip (F) d_1 \left(  m_0^{\lambda_i}, m_0 \right).
	$$
	Similarly, we can derive $c(m_0^{\lambda_i}) - c(m_0) \leq \Lip (F) d_1 \left(  m_0^{\lambda_i}, m_0 \right)$. Thus, $c(m_0^{\lambda_i}) \rightarrow c(m_0)$ as $i \rightarrow \infty$.

	For any $t >0$, any $m \in \PP ( \T^d )$, any $u \in \CC ( \T^d )$ and any $\lambda \in (0,1]$, define two Lax-Oleinik operators by
	\begin{align*}
		& \tilde{T}_t^m u(y) := \inf_{x \in \R^d} \left( u(x) + h_t^m (x,y) + c(m)t\right) ,\\
		& \tilde{T}_{t, \lambda}^m u(y) := \inf_{ x \in \R^d} \left( e^{-\lambda t} u(x) + h_{t, \lambda}^m (x,y) + c(m) \frac{1- e^{-\lambda t}}{\lambda}\right) .
	\end{align*}

	Fix $\tau \in (0,1)$. Since $\tilde{T}_{\tau, \lambda_i}^{m_0^{\lambda_i}} \bar{u}_{m_0^{\lambda_i}}^{\lambda_i} = \bar{u}_{m_0^{\lambda_i}}^{\lambda_i}$, we obtain that
	\begin{align*}
		& \left\Vert \tilde{T}_\tau^{m_0} u_0 - u_0 \right\Vert_\infty \\
		\leq & \left\Vert \tilde{T}_\tau^{m_0} u_0 - \tilde{T}_{\tau, \lambda_i}^{m_0^{\lambda_i}} u_0 \right\Vert_\infty + \left\Vert \tilde{T}_{\tau, \lambda_i}^{m_0^{\lambda_i}} u_0 - \tilde{T}_{\tau, \lambda_i}^{m_0^{\lambda_i}} \bar{u}_{m_0^{\lambda_i}}^{\lambda_i} \right\Vert_\infty + \left\Vert \bar{u}_{m_0^{\lambda_i}}^{\lambda_i} -u_0 \right\Vert_\infty \\
		\leq & \left\Vert \tilde{T}_\tau^{m_0} u_0 - \tilde{T}_{\tau, \lambda_i}^{m_0^{\lambda_i}} u_0 \right\Vert_\infty + 2 \left\Vert \bar{u}_{m_0^{\lambda_i}}^{\lambda_i} -u_0 \right\Vert_\infty \\
		\leq & \left\Vert \tilde{T}_\tau^{m_0} u_0 - \tilde{T}_{\tau, \lambda_i}^{m_0} u_0 \right\Vert_\infty + \left\Vert \tilde{T}_{\tau, \lambda_i}^{m_0} u_0 - \tilde{T}_{\tau, \lambda_i}^{m_0^{\lambda_i}} u_0 \right\Vert_\infty + 2 \left\Vert \bar{u}_{m_0^{\lambda_i}}^{\lambda_i} -u_0 \right\Vert_\infty.
	\end{align*}
	For any $y \in \R^d$, any minimizer $x$ of $\tilde{T}_{\tau, \lambda_i}^{m_0^{\lambda_i}} u_0(y)$ and any minimizer $\gamma$ of $h_{\tau, \lambda_i}^{m_0^{\lambda_i}} (x,y)$, we obtain that
	\begin{align*}
		& \tilde{T}_{\tau, \lambda_i}^{m_0} u_0 (y) - \tilde{T}_{\tau, \lambda_i}^{m_0^{\lambda_i}} u_0 (y) \\
		\leq & e^{-\lambda_i \tau} u_0(x) + h_{\tau, \lambda_i}^{m_0} (x,y) - e^{-\lambda_i \tau} u_0(x) - h_{\tau, \lambda_i}^{m_0^{\lambda_i}} (x,y) + \left( c(m_0) - c(m_0^{\lambda_i}) \right) \frac{1- e^{-\lambda_i \tau}}{\lambda_i} \\
		\leq & \int_{-\tau}^0 e^{\lambda_i s} \left( L_{m_0} \left( \gamma(s), \dot{\gamma}(s) \right) - L_{m_0^{\lambda_i}} \left( \gamma(s), \dot{\gamma}(s) \right) \right) ds + \left( c(m_0) - c(m_0^{\lambda_i})\right) \frac{1- e^{-\lambda_i \tau}}{\lambda_i} \\
		\leq & \frac{1- e^{-\lambda_i \tau}}{\lambda_i} \left( \Lip(F) d_1 \left(  m_0^{\lambda_i}, m_0 \right) + c(m_0) - c(m_0^{\lambda_i}) \right).
	\end{align*}
	Similarly, we have
	$$
	\tilde{T}_{\tau, \lambda_i}^{m_0^{\lambda_i}} u_0 (y) - \tilde{T}_{\tau, \lambda_i}^{m_0} u_0 (y) \leq \frac{1- e^{-\lambda_i \tau}}{\lambda_i} \left( \Lip(F) d_1 \left(  m_0^{\lambda_i}, m_0 \right) + c(m_0) - c(m_0^{\lambda_i}) \right).
	$$
	Thus, $\lim_{i \rightarrow \infty} \left\Vert \tilde{T}_{\tau, \lambda_i}^{m_0} u_0 - \tilde{T}_{\tau, \lambda_i}^{m_0^{\lambda_i}} u_0 \right\Vert_\infty = 0$.
	
	The only thing left is to prove
	$$
	\lim_{i \rightarrow \infty} \left\Vert \tilde{T}_\tau^{m_0} u_0 - \tilde{T}_{\tau, \lambda_i}^{m_0} u_0 \right\Vert_\infty = 0.
	$$
	For any $y \in \R^d$ and any minimizer $x$ of $\tilde{T}_{\tau, \lambda_i}^{m_0} u_0 (y)$, there exists a constant $D$ such that $\vert x-y \vert \leq \tau D$ by Proposition \ref{minimizer x is near}. Moreover, 
	\begin{align*}
		\tilde{T}_\tau^{m_0} u_0 (y) - \tilde{T}_{\tau, \lambda_i}^{m_0} u_0 (y) 
		& \leq u_0(x)+h_\tau^{m_0} (x,y) + c(m_0) \tau - e^{-\lambda_i \tau} u_0(x) - h_{\tau, \lambda_i}^{m_0} (x,y) - c(m_0) \frac{1- e^{- \lambda_i \tau}}{\lambda_i} \\
		& = (1- e^{-\lambda_i \tau}) u_0(x) + h_\tau^{m_0} (x,y) - h_{\tau, \lambda_i}^{m_0} (x,y) + c(m_0) \left( \tau - \frac{1- e^{- \lambda_i \tau}}{\lambda_i} \right).
	\end{align*}
	Since $u_0$ and $c(m_0)$ are bounded, we only need to prove:
	$$
	A:= h_\tau^{m_0} (x,y) - h_{\tau, \lambda_i}^{m_0} (x,y) \rightarrow 0.
	$$
	For any absolutely continuous curve $\gamma$ that minimizes $h_{\tau, \lambda_i}^{m_0} (x,y)$, we have
	$$
	A \leq \int_{-\tau}^0 L_{m_0} \left( \gamma(s), \dot{\gamma}(s) \right) ds - \int_{-\tau}^0 e^{\lambda_i s} L_{m_0} \left( \gamma(s), \dot{\gamma}(s) \right) ds.
	$$
	By Lemma \ref{lemma 1}, there exists a constant $C$, independent of $\lambda_i$, such that $\left\Vert \dot{\gamma} \right\Vert_{\infty} \leq C$. It follows that
	$$
	\left\vert L_{m_0} \left( \gamma(s), \dot{\gamma}(s) \right) \right\vert \leq \sup_{x \in \T^d, \vert v \vert \leq C} \left\vert L(x,v) \right\vert +F_\infty.
	$$
	Thus
	$$
	A \leq \left( \sup_{x \in \T^d, \vert v \vert \leq C} \left\vert L(x,v) \right\vert +F_\infty \right) \left( \tau - \frac{1- e^{-\lambda_i \tau}}{\lambda_i} \right) \rightarrow 0.
	$$
	Building on similar results in \cite[Proposition 5 and Lemma 1]{MR4605206}, we conclude that $\tilde{T}_\tau^{m_0} u_0 = u_0$.
	
	For any $t>0$, there exist $n \in \mathbb{N}$ and $\tau \in (0,1)$ such that $t = n\tau$. Thus,
	\begin{align*}
		\left\Vert \tilde{T}_t^{m_0} u_0 - u_0 \right\Vert_\infty 
		& \leq \left\Vert \tilde{T}_{n\tau}^{m_0} u_0 - \tilde{T}_{(n-1)\tau}^{m_0} u_0 \right\Vert_\infty + \cdots +\left\Vert \tilde{T}_\tau^{m_0} u_0 - u_0 \right\Vert_\infty \\
		& \leq n \left\Vert \tilde{T}_\tau^{m_0} u_0 - u_0 \right\Vert_\infty = 0.
	\end{align*}
	Therefore, we can conclude that $u_0$ is a viscosity solution to \eqref{need0705_1}.
\end{proof}

\begin{proposition} \label{convergence2_prop2}
	There exists a measure $\mu_0$ such that $\mu_0^{\lambda_i} \stackrel{w^*}{\longrightarrow} \mu_0$ with $m_0 = \pi \sharp \mu_0$, where $\lambda_i$, $m_0$ are defined as in Proposition \ref{convergence2_prop1}, and $\mu_0^{\lambda_i}$ is defined as in Proposition \ref{u_0 is solution of HJE}.
	Moreover, $\mu_0$ is a Mather measure for $L_{m_0}$.
\end{proposition}
\begin{proof}
	By Remark \ref{remark 3.3} and Proposition \ref{u_0 is solution of HJE}, there exists a minimizing $\lambda_i$-measure $\mu_0^{\lambda_i} \in \PP (K)$ for $L_{m_0^{\lambda_i}}$, where $K \subset \T^d \times \R^d$ is a compact set. Consequently, there exists a measure $\mu_0 \in \PP (K)$ such that $\mu_0^{\lambda_i} \stackrel{w^*}{\longrightarrow} \mu_0$, and clearly $\pi \sharp \mu_0 = m_0$.
	For any $\varphi \in \CC^1 ( \T^d )$, since $\supp \left( \mu_0^{\lambda_i} \right) \subset K$ and $\supp \left( \mu_0 \right) \subset K$,
	it follows that
	$$
	\int_{\T^d \times \R^d} v D \varphi (x) d \mu_0 = \lim_{i \rightarrow \infty} \int_{\T^d \times \R^d} v D \varphi (x) d \mu_0^{\lambda_i} = 0.
	$$
	Moreover, by Proposition \ref{need0813}, we have
	\begin{align*}
		& \left\vert \int_{\T^d \times \R^d} L_{m_0} (x,v) d \mu_0 + c(m_0) \right\vert \\
		\leq & \left\vert \int_{\T^d \times \R^d} \left( L_{m_0} (x,v) - L_{m_0^{\lambda_i}} (x,v) \right) d \mu_0^{\lambda_i} \right\vert  + \left\vert \int_{\T^d \times \R^d} L_{m_0} (x,v) d \mu_0^{\lambda_i} - \int_{\T^d \times \R^d} L_{m_0} (x,v) d \mu_0 \right\vert \\
		& + \left\vert \int_{\T^d \times \R^d} \left( L_{m_0^{\lambda_i}} (x,v) - \lambda_i u_{m_0^{\lambda_i}}^{\lambda_i} (x) \right) d \mu_0^{\lambda_i} \right\vert + \left\vert \int_{\T^d \times \R^d} \lambda_i u_{m_0^{\lambda_i}}^{\lambda_i} (x) d \mu_0^{\lambda_i} + c(m_0^{\lambda_i}) \right\vert + \left\vert c(m_0^{\lambda_i}) - c(m_0) \right\vert \\
		\leq & \Lip (F)  d_1 \left(  m_0^{\lambda_i}, m_0 \right) + \left\vert \int_{\T^d \times \R^d} L_{m_0} (x,v) d \mu_0^{\lambda_i} - \int_{\T^d \times \R^d} L_{m_0} (x,v) d \mu_0 \right\vert  + \lambda_i C + \left\vert c(m_0^{\lambda_i}) - c(m_0) \right\vert.
	\end{align*}
	Let $i \rightarrow \infty$. We can conclude that
	$$
	\int_{\T^d \times \R^d} L_{m_0} (x,v) d \mu_0 = - c(m_0),
	$$
	indicating that $\mu_0$ is a Mather measure for $L_{m_0}$.
\end{proof}

By the same method used in \cite[Proposition 11]{MR4605206}, we can obtain the following result.
\begin{proposition} \label{need0806}
	The measure $m_0$ defined as in Proposition \ref{convergence2_prop1} is a solution to (\ref{MFG}b) in the sense of distributions.
\end{proposition}

\section{Proof of Theorem \ref{theorem 2}} \label{section5}

Theorem \ref{theorem 2} (1) is established by Proposition \ref{convergence1_prop1} and Proposition \ref{convergence1_prop2}, while Theorem \ref{theorem 2} (2) is established with the help of Proposition \ref{main theorem in classical case}.

Define the constant $\tau_0$ as in Proposition \ref{main theorem in classical case}. 
For any $\tau \in (0,\tau_0)$ and any $m \in \PP ( \T^d )$, denote by $u_{\tau, m}$ the solution to the discrete Lax-Oleinik equation \eqref{discrete_laxoleinik_classical}, define $m_\tau$ and $\mu_{\tau,m}$ as in Proposition \ref{main theorem in classical case}.

\begin{proposition} \label{convergence1_prop1}
For any $\tau \in (0,\tau_0)$, there exists a sequence $\lambda_i \rightarrow 0$ as $i \rightarrow \infty$ such that $m_\tau^{\lambda_i} \stackrel{w^*}{\longrightarrow} m_0^\tau$ and
$$
u^{\lambda_i}_{\tau, m_\tau^{\lambda_i}} - \frac{\bar{L} \left( \tau, m_\tau^{\lambda_i} \right)}{\lambda_i} \rightarrow u_0^\tau
$$
uniformly for some measure $m_0^\tau \in \PP ( \T^d )$ and some function $u_0^\tau$. Moreover, $u_0^\tau$ is the solution to the discrete Lax-Oleinik equation \eqref{discrete_laxoleinik_classical} with respect to the measure $m_0^\tau$.
\end{proposition}
\begin{proof}
Fix $\tau \in (0,\tau_0)$ and $\lambda \in (0,1]$. For any $m \in \PP ( \T^d )$ and any $y \in \argmax_{y \in \R^d} \left\{ u_{\tau,m}^\lambda (y) - \frac{\bar{L} (\tau,m)}{\lambda} \right.$ $ - u_{\tau,m}(y) \Big{ \}}$, we have
\begin{align*}
	u_{\tau,m}^\lambda (y) - \frac{\bar{L} (\tau,m)}{\lambda} - u_{\tau,m}(y)
	\leq & (1- \tau \lambda) \left( u_{\tau,m}^\lambda (x) - \frac{\bar{L} (\tau,m)}{\lambda} - u_{\tau,m}(x) \right) \\
	& + \LL_{\tau,m} (x,y) - u_{\tau,m}(y) + u_{\tau,m}(x) - \tau \bar{L}(\tau,m) - \tau \lambda u_{\tau,m}(x), \quad \forall x \in \R^d.
\end{align*}
Let $x \in \argmin_{x \in \R^d} \left\{u_{\tau,m}(x) + \LL_{\tau,m} (x,y)\right\}$. Since
$$
u_{\tau,m}^\lambda (x) - \frac{\bar{L} (\tau,m)}{\lambda} - u_{\tau,m}(x) \leq u_{\tau,m}^\lambda (y) - \frac{\bar{L} (\tau,m)}{\lambda} - u_{\tau,m}(y),
$$
we obtain that
$$
u_{\tau,m}^\lambda (y) - \frac{\bar{L} (\tau,m)}{\lambda} - u_{\tau,m}(y) \leq - u_{\tau,m}(x).
$$
Thus, for any $z \in \R^d$, it follows that
$$
u_{\tau,m}^\lambda (z) - \frac{\bar{L} (\tau,m)}{\lambda} - u_{\tau,m}(z) \leq u_{\tau,m}^\lambda (y) - \frac{\bar{L} (\tau,m)}{\lambda} - u_{\tau,m}(y) \leq - u_{\tau,m}(x).
$$
By Proposition \ref{equi-lipschitz of classical solution}, there exists a constant $C$ such that $u_{\tau,m}^\lambda (z) - \frac{\bar{L} (\tau,m)}{\lambda} \leq C$.
Similarly, for any $y \in \argmin_{y \in \R^d} \left\{u_{\tau,m}^\lambda (y) - \frac{\bar{L} (\tau,m)}{\lambda} - u_{\tau,m}(y)\right\}$ and any $x \in \argmin_{x \in \R^d} \left\{(1-\tau\lambda) u_{\tau,m}^\lambda (x) + \LL_{\tau,m} (x,y)\right\}$, we have
\begin{align*}
	u_{\tau,m}^\lambda (y) - \frac{\bar{L} (\tau,m)}{\lambda} - u_{\tau,m}(y)
	= & (1- \tau \lambda) \left( u_{\tau,m}^\lambda (x) - \frac{\bar{L} (\tau,m)}{\lambda} - u_{\tau,m}(x) \right) \\
	& + \LL_{\tau,m} (x,y) - u_{\tau,m}(y) + u_{\tau,m}(x) - \tau \bar{L}(\tau,m) - \tau \lambda u_{\tau,m}(x).
\end{align*}
Since
$$
u_{\tau,m}(y) + \tau \bar{L} (\tau,m) \leq u_{\tau,m}(x) + \LL_{\tau,m}(x,y),
$$
we obtain that
$$
u_{\tau,m}^\lambda (y) - \frac{\bar{L} (\tau,m)}{\lambda} - u_{\tau,m}(y) \geq - u_{\tau,m}(x).
$$
Thus, for any $z \in \R^d$, it follows that
$$
u_{\tau,m}^\lambda (z) - \frac{\bar{L} (\tau,m)}{\lambda} - u_{\tau,m}(z) \geq u_{\tau,m}^\lambda (y) - \frac{\bar{L} (\tau,m)}{\lambda} - u_{\tau,m}(y) \geq - u_{\tau,m}(x).
$$
By Proposition \ref{equi-lipschitz of classical solution}, there exists a constant $C$ such that $u_{\tau,m}^\lambda (z) - \frac{\bar{L} (\tau,m)}{\lambda} \geq -C$.
Thus,
$$
\left\Vert u_{\tau,m}^\lambda - \frac{\bar{L}(\tau,m)}{\lambda} \right\Vert_\infty \leq C.
$$

By Proposition \ref{equi-Lipschitz of discounted HJE}, $u_{\tau,m}^\lambda - \frac{\bar{L}(\tau,m)}{\lambda}$ is equi-Lipschitz. Since $\PP ( \T^d )$ is compact, there exists a sequence $\lambda_i \rightarrow 0$ such that $m_\tau^{\lambda_i} \stackrel{w^*}{\longrightarrow} m_0^\tau$ for some measure $m_0^\tau$. By Ascoli-Arz\'ela theorem, taking a subsequence if necessary, there exists a Lipschitz function $u_0^\tau$ such that
$$
u_{\tau, m_\tau^{\lambda_i}}^{\lambda_i} - \frac{\bar{L} \left( \tau, m_\tau^{\lambda_i} \right)}{\lambda_i} \rightarrow u_0^\tau
$$
uniformly.
For any $x,y \in \R^d$, we have
$$
u_{\tau, m_\tau^{\lambda_i}}^{\lambda_i} (y) - \frac{\bar{L} \left( \tau, m_\tau^{\lambda_i} \right)}{\lambda_i} \leq (1- \tau \lambda_i) \left( u_{\tau, m_\tau^{\lambda_i}}^{\lambda_i} (x) - \frac{\bar{L} \left( \tau, m_\tau^{\lambda_i} \right)}{\lambda_i} \right) + \LL_{\tau, m_\tau^{\lambda_i}} (x,y) - \tau \bar{L} \left( \tau, m_\tau^{\lambda_i} \right).
$$
Let $i$ tend to infinity. We obtain that
$$
u_0^\tau(y) \leq u_0^\tau(x) + \LL_{\tau, m_0^\tau}(x,y) - \tau \bar{L} \left( \tau, m_0^\tau \right).
$$
Moreover, for any $y \in \R^d$, there exists a sequence $x_i \in \R^d$ such that
$$
u_{\tau, m_\tau^{\lambda_i}}^{\lambda_i} (y) = (1- \tau \lambda_i) u_{\tau, m_\tau^{\lambda_i}}^{\lambda_i} (x_i) + \LL_{\tau, m_\tau^{\lambda_i}} (x_i, y).
$$
By Proposition \ref{minimizer x is near}, there exists a constant $D$ such that $\vert x_i - y \vert \leq \tau D$. Since $\left\{ x \in \R^d \middle\vert \,\, \vert y-x\vert \leq \tau D \right\}$ is compact, there exists $x_0 \in \R^d$ such that $x_i \rightarrow x_0$.
Since
$$
u_{\tau, m_\tau^{\lambda_i}}^{\lambda_i} (y) - \frac{\bar{L} \left( \tau, m_\tau^{\lambda_i} \right)}{\lambda_i} = (1- \tau \lambda_i) \left( u_{\tau, m_\tau^{\lambda_i}}^{\lambda_i} (x_i) - \frac{\bar{L} \left( \tau, m_\tau^{\lambda_i} \right)}{\lambda_i} \right) + \LL_{\tau, m_\tau^{\lambda_i}} (x_i,y) - \tau \bar{L} \left( \tau, m_\tau^{\lambda_i} \right),
$$
we obtain $u_0^\tau(y) = u_0^\tau(x_0) + \LL_{\tau, m_0^\tau}(x_0 , y) - \tau \bar{L} \left( \tau, m_0^\tau \right)$ as $i \rightarrow \infty$.
Thus,
$$
u_0^\tau(y) + \tau \bar{L} \left( \tau, m_0^\tau \right) = \inf_{ x \in \R^d} \left( u_0^\tau(x) + \LL_{\tau, m_0^\tau}(x , y)\right) ,
$$
indicating that the limit $u_0^\tau$ is a solution to \eqref{discrete_laxoleinik_classical} with respect to the measure $m_0^\tau$.
\end{proof}

\begin{proposition} \label{convergence1_prop2}
For any $\tau \in (0,\tau_0)$, there exists a measure $\mu_0^\tau$ such that $\mu_{\tau, m_\tau^{\lambda_i}}^{\lambda_i} \stackrel{w^*}{\longrightarrow} \mu_0^\tau$ with $m_0^\tau = \pi \sharp \mu_0^\tau$, where $\lambda_i$, $m_0^\tau$ are defined as in Proposition \ref{convergence1_prop1}, and $\mu_{\tau, m_\tau^{\lambda_i}}^{\lambda_i}$, $m_\tau^{\lambda_i}$ are defined as in Proposition \ref{existence of minimizing tau holonomic lambda measure for MFG}. Moreover, $\mu_0^\tau$ is a minimizing $\tau$-holonomic measure for $L_{m_0^\tau}$, which implies that $m_0^\tau$ is a minimizing $\tau$-holonomic measure for MFGs \eqref{MFG}.
\end{proposition}
\begin{proof}
By the definition of $m_\tau^{\lambda_i}$, there exists a minimizing $\tau$-holonomic $\lambda_i$-measure $\mu_{\tau, m_\tau^{\lambda_i}}^{\lambda_i}$ with $m_\tau^{\lambda_i} = \pi \sharp \mu_{\tau, m_\tau^{\lambda_i}}^{\lambda_i}$. By Remark \ref{remark 3.3}, there exists a compactly supported measure $\mu_0^\tau$ such that $\mu_{\tau, m_\tau^{\lambda_i}}^{\lambda_i} \stackrel{w^*}{\longrightarrow} \mu_0^\tau$ as $i \rightarrow \infty$. Clearly, $m_0^\tau = \pi \sharp \mu_0^\tau$.

By Remark \ref{remark 3.3}, for any $\varphi \in \CC ( \T^d )$, we always have
$$
\int_{\T^d \times \R^d} \left( \varphi(x+\tau v) - \varphi(x) \right) d \mu_0^\tau = \lim_{i \rightarrow \infty} \int_{\T^d \times \R^d} \left( \varphi(x+\tau v) - \varphi(x) \right) d \mu_{\tau, m_\tau^{\lambda_i}}^{\lambda_i} = 0,
$$
indicating that $\mu_0^\tau \in \PP_\tau ( \T^d \times \R^d )$.
Moreover,
\begin{align*}
	& \left\vert \int_{\T^d \times \R^d} L_{m_0^\tau} (x,v) d \mu_0^\tau - \bar{L} \left( \tau, m_0^\tau \right) \right\vert \\
	\leq & \left\vert \int_{\T^d \times \R^d} L_{m_0^\tau} (x,v) d \mu_0^\tau - \int_{\T^d \times \R^d} L_{m_\tau^{\lambda_i}} (x,v) d \mu_{\tau, m_\tau^{\lambda_i}}^{\lambda_i} \right\vert + \left\vert \int_{\T^d \times \R^d} \left( L_{m_\tau^{\lambda_i}} (x,v) - \lambda_i u_{\tau, m_\tau^{\lambda_i}}^{\lambda_i} \right) d \mu_{\tau, m_\tau^{\lambda_i}}^{\lambda_i} \right\vert \\
	& + \left\vert \int_{\T^d \times \R^d} \lambda_i u_{\tau, m_\tau^{\lambda_i}}^{\lambda_i} d \mu_{\tau, m_\tau^{\lambda_i}}^{\lambda_i} - \bar{L} \left( \tau, m_\tau^{\lambda_i} \right) \right\vert + \left\vert \bar{L} \left( \tau, m_\tau^{\lambda_i} \right) - \bar{L} \left( \tau, m_0^\tau \right) \right\vert \\
	\triangleq & G_1 +G_2+G_3+G_4.
\end{align*}
First of all, since $\mu_{\tau, m_\tau^{\lambda_i}}^{\lambda_i}$ is a minimizing $\tau$-holonomic $\lambda_i$-measure, $G_2 = 0$.
Since $\mu_{\tau, m_\tau^{\lambda_i}}^{\lambda_i} \stackrel{w^*}{\longrightarrow} \mu_0^\tau$,
\begin{align*}
	G_1 
	& \leq \int_{\T^d \times \R^d} \left\vert L_{m_\tau^{\lambda_i}} (x,v) - L_{m_0^\tau} (x,v) \right\vert d \mu_{\tau, m_\tau^{\lambda_i}}^{\lambda_i} + \left\vert \int_{\T^d \times \R^d} L_{m_0^\tau} (x,v) d \mu_{\tau, m_\tau^{\lambda_i}}^{\lambda_i} - \int_{\T^d \times \R^d} L_{m_0^\tau} (x,v) d \mu_0^\tau \right\vert \\
	& \leq \Lip(F) d_1 \left( m_\tau^{\lambda_i}, m_0^\tau \right) + \left\vert \int_{\T^d \times \R^d} L_{m_0^\tau} (x,v) d \mu_{\tau, m_\tau^{\lambda_i}}^{\lambda_i} - \int_{\T^d \times \R^d} L_{m_0^\tau} (x,v) d \mu_0^\tau \right\vert \rightarrow 0.
\end{align*}
%Thus, $\lim_{i \rightarrow +\infty} A_1 = 0$.
Moreover, by Proposition \ref{need0813},
$$
\lim_{i \rightarrow \infty} G_3 = \lim_{i \rightarrow \infty} \lambda_i \left\vert \int_{\T^d \times \R^d} \left( u_{\tau, m_\tau^{\lambda_i}}^{\lambda_i}  - \frac{\bar{L} \left( \tau, m_\tau^{\lambda_i} \right)}{\lambda_i} \right) d \mu_{\tau, m_\tau^{\lambda_i}}^{\lambda_i} \right\vert \leq \lim_{i \rightarrow \infty} \lambda_i C = 0.
$$
Finally,
$$
\lim_{i \rightarrow \infty} G_4 \leq \lim_{i \rightarrow \infty} \Lip (F) d_1 \left( m_\tau^{\lambda_i}, m_0 \right) = 0.
$$
Thus, $\mu_0^\tau$ is a minimizing $\tau$-holonomic measure for $L_{m_0^\tau}$ with $m_0^\tau = \pi \sharp \mu_0^\tau$, implying that $m_0^\tau$ is a minimizing $\tau$-holonomic measures for MFGs \eqref{MFG}.
\end{proof}

\bigskip

\section*{Acknowledgement}
\addcontentsline{toc}{section}{Acknowledgements}
Kaizhi Wang gratefully acknowledges the hospitality and support provided by the Tianyuan Mathematical Center in Northeast China during his visit to Changchun in July 2025.
%Kaizhi Wang is supported by NSFC Grant No. 12171315, 11931016, and Natural Science Foundation of Shanghai No. 22ZR1433100. 
\bigskip

\section*{Statements and Declarations}
\addcontentsline{toc}{section}{Statements and Declarations}
\begin{itemize}
	\item {\bf The Data Availability Statement}. No datasets were generated or analysed
	during the current study.
	\item {\bf The Conflict of Interest Statement}. We have no conflicts of interest to disclose.
	\item {\bf Funding}. 
	Renato Iturriaga was partially supported by Conacyt Mexico grant A1-S-33854.
	Cristian Mendico was partially supported by Istituto Nazionale di Alta Matematica, INdAM-GNAMPA project 2023, and the King Abdullah University of Science and Technology (KAUST) project CRG2021-4674 ``Mean-Field Games: models, theory and computational aspects''. He also acknowledges the MIUR Excellence Department Project MatMod@TOV awarded to the Department of Mathematics, University of Rome Tor Vergata, CUP E83C23000330006. 
	Kaizhi Wang is partially supported by National Natural Science Foundation of China (Grant Nos. 12525107, 12171315).
\end{itemize}

\bigskip
\section*{Appendix}
\addcontentsline{toc}{section}{Appendix}
%\appendix

\subsection*{Appendix A. Proof of Lemma \ref{convergence of aubry set}} \label{proof of Lemma 4.1}
\addcontentsline{toc}{subsection}{Appendix A. Proof of Lemma \ref{convergence of aubry set}}

	Fix $\tau \in (0,1)$ and $\lambda \in (0,1]$. Denote $v_{n, x_i}^{\tau, \lambda, m_\tau^\lambda} := \frac{1}{\tau} \left( x_{n+1, x_i}^{\tau, \lambda, m_\tau^\lambda} - x_{n, x_i}^{\tau, \lambda, m_\tau^\lambda} \right)$ for any integer $n \leq 0$. By Proposition \ref{minimizer x is near}, there exists a constant $D$ such that
	$$
	\left\vert v_{n, x_i}^{\tau, \lambda, m_\tau^\lambda} \right\vert \leq D.
	$$
	We first prove that there exists a constant $C$ such that
	$$
	\left\vert v_{n, x_i}^{\tau, \lambda, m_\tau^\lambda} - v_{n-1, x_i}^{\tau, \lambda, m_\tau^\lambda} \right\vert \leq C\tau, \quad \forall n \leq 0.
	$$
	For any $n \leq -1$, we obtain that
	\begin{align*}
		u_{\tau,m_\tau^\lambda}^\lambda \left( x_{n+1, x_i}^{\tau, \lambda, m_\tau^\lambda} \right)
		& = (1-\tau\lambda) u_{\tau,m_\tau^\lambda}^\lambda \left( x_{n, x_i}^{\tau, \lambda, m_\tau^\lambda} \right) +\LL_{\tau,m_\tau^\lambda} \left( x_{n, x_i}^{\tau, \lambda, m_\tau^\lambda},  x_{n+1, x_i}^{\tau, \lambda, m_\tau^\lambda} \right) \\
		& = (1-\tau\lambda)^2 u_{\tau,m_\tau^\lambda}^\lambda \left( x_{n-1, x_i}^{\tau, \lambda, m_\tau^\lambda} \right) +(1-\tau\lambda) \LL_{\tau,m_\tau^\lambda} \left( x_{n-1, x_i}^{\tau, \lambda, m_\tau^\lambda}, x_{n, x_i}^{\tau, \lambda, m_\tau^\lambda} \right) + \LL_{\tau,m_\tau^\lambda} \left( x_{n, x_i}^{\tau, \lambda, m_\tau^\lambda},  x_{n+1, x_i}^{\tau, \lambda, m_\tau^\lambda} \right) \\
		& \leq (1-\tau\lambda) u_{\tau,m_\tau^\lambda}^\lambda \left( x \right) +\LL_{\tau,m_\tau^\lambda} \left( x,  x_{n+1, x_i}^{\tau, \lambda, m_\tau^\lambda} \right) \\
		& \leq (1-\tau\lambda)^2 u_{\tau,m_\tau^\lambda}^\lambda \left( x_{n-1, x_i}^{\tau, \lambda, m_\tau^\lambda} \right) +(1-\tau\lambda) \LL_{\tau,m_\tau^\lambda} \left( x_{n-1, x_i}^{\tau, \lambda, m_\tau^\lambda}, x \right) + \LL_{\tau,m_\tau^\lambda} \left( x,  x_{n+1, x_i}^{\tau, \lambda, m_\tau^\lambda} \right), \quad \forall x \in \R^d,
	\end{align*}
	indicating that $x_{n, x_i}^{\tau, \lambda, m_\tau^\lambda}$ satisfies
	$$
	x_{n, x_i}^{\tau, \lambda, m_\tau^\lambda} \in \argmin_{x \in \R^d} \left\{ (1-\tau\lambda) \LL_{\tau,m_\tau^\lambda} \left( x_{n-1, x_i}^{\tau, \lambda, m_\tau^\lambda}, x \right) + \LL_{\tau,m_\tau^\lambda} \left( x,  x_{n+1, x_i}^{\tau, \lambda, m_\tau^\lambda} \right)\right\}.
	$$
	Thus, we obtain that
	$$
	(1-\tau\lambda) \frac{\partial \LL_{\tau,m_\tau^\lambda}}{\partial y} \left( x_{n-1, x_i}^{\tau, \lambda, m_\tau^\lambda}, x_{n, x_i}^{\tau, \lambda, m_\tau^\lambda} \right) + \frac{\partial \LL_{\tau,m_\tau^\lambda}}{\partial x} \left( x_{n, x_i}^{\tau, \lambda, m_\tau^\lambda},  x_{n+1, x_i}^{\tau, \lambda, m_\tau^\lambda} \right) = 0,
	$$
	$$
	(1-\tau\lambda) \frac{\partial L_{m_\tau^\lambda}}{\partial v} \left( x_{n-1, x_i}^{\tau, \lambda, m_\tau^\lambda}, v_{n-1, x_i}^{\tau, \lambda, m_\tau^\lambda} \right) + \tau \frac{\partial L_{m_\tau^\lambda}}{\partial x} \left( x_{n, x_i}^{\tau, \lambda, m_\tau^\lambda}, v_{n, x_i}^{\tau, \lambda, m_\tau^\lambda} \right) - \frac{\partial L_{m_\tau^\lambda}}{\partial v}\left( x_{n, x_i}^{\tau, \lambda, m_\tau^\lambda}, v_{n, x_i}^{\tau, \lambda, m_\tau^\lambda} \right) = 0,
	$$
	$$
	\frac{1}{\tau} \left( \frac{\partial L_{m_\tau^\lambda}}{\partial v}\left( x_{n, x_i}^{\tau, \lambda, m_\tau^\lambda}, v_{n, x_i}^{\tau, \lambda, m_\tau^\lambda} \right) - \frac{\partial L_{m_\tau^\lambda}}{\partial v} \left( x_{n-1, x_i}^{\tau, \lambda, m_\tau^\lambda}, v_{n-1, x_i}^{\tau, \lambda, m_\tau^\lambda} \right) \right) = \frac{\partial L_{m_\tau^\lambda}}{\partial x} \left( x_{n, x_i}^{\tau, \lambda, m_\tau^\lambda}, v_{n, x_i}^{\tau, \lambda, m_\tau^\lambda} \right) - \lambda \frac{\partial L_{m_\tau^\lambda}}{\partial v} \left( x_{n-1, x_i}^{\tau, \lambda, m_\tau^\lambda}, v_{n-1, x_i}^{\tau, \lambda, m_\tau^\lambda} \right).
	$$
	By (L\ref{MFG_L2}), for any $x \in \R^d$ and any $\vert v \vert \leq D$, there exists a constant $\alpha (D) >0$ such that
	$$
	h^T \cdot \frac{\partial^2 L}{\partial v^2} (x,v) \cdot h \geq \alpha (D) \vert h \vert^2,\quad \forall h \in \R^d.
	$$
	On the one hand, we obtain that
	\begin{align*}
		& \int_0^1 \frac{d}{dt} \left( \frac{\partial L_{m_\tau^\lambda}}{\partial v} \left( x_{n-1, x_i}^{\tau, \lambda, m_\tau^\lambda} + t \left( x_{n, x_i}^{\tau, \lambda, m_\tau^\lambda} - x_{n-1, x_i}^{\tau, \lambda, m_\tau^\lambda} \right), v_{n-1, x_i}^{\tau, \lambda, m_\tau^\lambda} + t \left( v_{n, x_i}^{\tau, \lambda, m_\tau^\lambda} - v_{n-1, x_i}^{\tau, \lambda, m_\tau^\lambda} \right) \right) \right) dt \\
		= & \frac{\partial L_{m_\tau^\lambda}}{\partial v} \left( x_{n, x_i}^{\tau, \lambda, m_\tau^\lambda}, v_{n, x_i}^{\tau, \lambda, m_\tau^\lambda} \right) - \frac{\partial L_{m_\tau^\lambda}}{\partial v} \left( x_{n-1, x_i}^{\tau, \lambda, m_\tau^\lambda}, v_{n-1, x_i}^{\tau, \lambda, m_\tau^\lambda} \right) \\
		= & \tau \frac{\partial L_{m_\tau^\lambda}}{\partial x} \left( x_{n, x_i}^{\tau, \lambda, m_\tau^\lambda}, v_{n, x_i}^{\tau, \lambda, m_\tau^\lambda} \right) - \tau \lambda \frac{\partial L_{m_\tau^\lambda}}{\partial v} \left( x_{n-1, x_i}^{\tau, \lambda, m_\tau^\lambda}, v_{n-1, x_i}^{\tau, \lambda, m_\tau^\lambda} \right).
	\end{align*}
	On the other hand,
	\begin{align*}
		& \int_0^1 \frac{d}{dt} \left( \frac{\partial L_{m_\tau^\lambda}}{\partial v} \left( x_{n-1, x_i}^{\tau, \lambda, m_\tau^\lambda} + t \left( x_{n, x_i}^{\tau, \lambda, m_\tau^\lambda} - x_{n-1, x_i}^{\tau, \lambda, m_\tau^\lambda} \right), v_{n-1, x_i}^{\tau, \lambda, m_\tau^\lambda} + t \left( v_{n, x_i}^{\tau, \lambda, m_\tau^\lambda} - v_{n-1, x_i}^{\tau, \lambda, m_\tau^\lambda} \right) \right) \right) dt \\
		= & \int_0^1 \frac{\partial^2 L_{m_\tau^\lambda}}{\partial x \partial v} \left( x_{n-1, x_i}^{\tau, \lambda, m_\tau^\lambda} + t \left( x_{n, x_i}^{\tau, \lambda, m_\tau^\lambda} - x_{n-1, x_i}^{\tau, \lambda, m_\tau^\lambda} \right), v_{n-1, x_i}^{\tau, \lambda, m_\tau^\lambda} + t \left( v_{n, x_i}^{\tau, \lambda, m_\tau^\lambda} - v_{n-1, x_i}^{\tau, \lambda, m_\tau^\lambda} \right) \right) \cdot \left( x_{n, x_i}^{\tau, \lambda, m_\tau^\lambda} - x_{n-1, x_i}^{\tau, \lambda, m_\tau^\lambda} \right) \\
		& + \frac{\partial^2 L_{m_\tau^\lambda}}{\partial v^2} \left( x_{n-1, x_i}^{\tau, \lambda, m_\tau^\lambda} + t \left( x_{n, x_i}^{\tau, \lambda, m_\tau^\lambda} - x_{n-1, x_i}^{\tau, \lambda, m_\tau^\lambda} \right), v_{n-1, x_i}^{\tau, \lambda, m_\tau^\lambda} + t \left( v_{n, x_i}^{\tau, \lambda, m_\tau^\lambda} - v_{n-1, x_i}^{\tau, \lambda, m_\tau^\lambda} \right) \right) \cdot \left( v_{n, x_i}^{\tau, \lambda, m_\tau^\lambda} - v_{n-1, x_i}^{\tau, \lambda, m_\tau^\lambda} \right) dt \\
		\geq & \left( \inf_{x \in \T^d, \vert v \vert \leq D} \frac{\partial^2 L_{m_\tau^\lambda}}{\partial x \partial v}(x,v) \right) \cdot \left( x_{n, x_i}^{\tau, \lambda, m_\tau^\lambda} - x_{n-1, x_i}^{\tau, \lambda, m_\tau^\lambda} \right) + \left( \inf_{x \in \T^d, \vert v \vert \leq D} \frac{\partial^2 L_{m_\tau^\lambda}}{\partial v^2}(x,v) \right) \cdot \left( v_{n, x_i}^{\tau, \lambda, m_\tau^\lambda} - v_{n-1, x_i}^{\tau, \lambda, m_\tau^\lambda} \right) .
	\end{align*}
	By $\frac{\partial L_{m_\tau^\lambda}}{\partial v} = \frac{\partial L}{\partial v}$, we obtain that
	\begin{align*}
		\alpha (D) \left\vert v_{n, x_i}^{\tau, \lambda, m_\tau^\lambda} - v_{n-1, x_i}^{\tau, \lambda, m_\tau^\lambda} \right\vert
		\leq & \sup_{x \in \T^d, \vert v \vert \leq D} \left\vert \frac{\partial^2 L}{\partial x \partial v} (x,v) \right\vert \cdot \left\vert x_{n, x_i}^{\tau, \lambda, m_\tau^\lambda} - x_{n-1, x_i}^{\tau, \lambda, m_\tau^\lambda} \right\vert \\
		& +\tau \sup_{x \in \T^d, \vert v \vert \leq D} \left\vert \frac{\partial L}{\partial x} (x,v) \right\vert + \tau F_\infty + \tau \lambda \sup_{x \in \T^d, \vert v \vert \leq D} \left\vert \frac{\partial L}{\partial v} (x,v) \right\vert.
	\end{align*}
	By $\left\vert x_{n, x_i}^{\tau, \lambda, m_\tau^\lambda} - x_{n-1, x_i}^{\tau, \lambda, m_\tau^\lambda} \right\vert \leq \tau D$ and (L\ref{MFG_L1}), there exists a constant $C$ such that
	$$
	\left\vert v_{n, x_i}^{\tau, \lambda, m_\tau^\lambda} - v_{n-1, x_i}^{\tau, \lambda, m_\tau^\lambda} \right\vert \leq \tau C, \quad \forall n \leq 0.
	$$

	Let $\gamma_{\tau,m_\tau^\lambda}^{\lambda, x_i} : (-\infty, 0] \rightarrow \R^d$ be the piecewise affine path interpolating the points $x_{n, x_i}^{\tau, \lambda, m_\tau^\lambda}$ at time $n \tau$.
	For any $s < t \leq 0$, if there exists $n \leq 0$ such that $s,t \in ((n-1)\tau, n \tau]$, then $\left\vert \gamma_{\tau,m_\tau^\lambda}^{\lambda, x_i} (t)- \gamma_{\tau,m_\tau^\lambda}^{\lambda, x_i}(s) \right\vert \leq (t-s) D$. Otherwise, there exist integers $n_1 \leq n_2$ such that $s \in ((n_1-1)\tau, n_1 \tau]$ and $t \in ((n_2-1)\tau, n_2 \tau]$. In this case,
	\begin{align*}
		\left\vert \gamma_{\tau,m_\tau^\lambda}^{\lambda, x_i} (t)- \gamma_{\tau,m_\tau^\lambda}^{\lambda, x_i}(s) \right\vert
		& \leq \left\vert \gamma_{\tau,m_\tau^\lambda}^{\lambda, x_i} (t)- x_{n_2 - 1, x_i}^{\tau, \lambda, m_\tau^\lambda} \right\vert + \left\vert x_{n_2 - 1, x_i}^{\tau, \lambda, m_\tau^\lambda} - x_{n_2 - 2, x_i}^{\tau, \lambda, m_\tau^\lambda} \right\vert + \cdots \left\vert x_{n_1, x_i}^{\tau, \lambda, m_\tau^\lambda} - \gamma_{\tau,m_\tau^\lambda}^{\lambda, x_i}(s) \right\vert \\
		& \leq D \left( t - (n_2 - 1)\tau + (n_2 - 1)\tau - (n_2 - 2)\tau + \cdots + n_1\tau - s \right) = D(t-s).
	\end{align*}
	Above all, the curve $\gamma_{\tau,m_\tau^\lambda}^{\lambda, x_i}$ is Lipschitz with Lipschitz constant $D$.
	Thus, taking a subsequence if necessary, we obtain that $\gamma_{\tau_i ,m_{\tau_i}^\lambda}^{\lambda, x_i} \rightarrow \gamma_0$ uniformly on any compact interval of $(-\infty, 0]$ for some curve $\gamma_0$ satisfying $\gamma_0(0)=x_0$.
	
	We claim that there exists a Lipschitz function $V : (-\infty, 0] \rightarrow \R^d$ such that
	$$
	\int_t^0 V(s) ds = x_0 - \gamma_0 (t), \quad \forall t \leq 0.
	$$
	Let $T \subset (-\infty, 0]$ be a countable dense subset. Define a map $V_i : (-\infty, 0) \rightarrow \R^d$ by
	$$
	V_i (t) := \frac{1}{\tau_i} \left( x_{n, x_i}^{\tau_i, \lambda, m_{\tau_i}^\lambda} - x_{n-1, x_i}^{\tau_i, \lambda, m_{\tau_i}^\lambda} \right),\quad \forall t \in [(n-1)\tau_i, n \tau_i),\,\, \forall n\leq 0.
	$$
	Since $\left\vert V_i (t) \right\vert \leq D$ for any $i \in \N$ and any $t \leq 0$, by taking a subsequence if necessary, we define $\tilde{V}$ by $V_i (t) \rightarrow \tilde{V}(t)$ for any $t \in T$.
	For any $s < t <0$, there exist integers $n_1 < n_2$ such that $s \in [(n_1-1)\tau_i, n_1 \tau_i)$ and $t \in [(n_2-1)\tau_i, n_2 \tau_i)$. Thus,
	$$
	\left\vert V_i (t) - V_i (s) \right\vert = \left\vert v_{n_2 -1, x_i}^{\tau_i, \lambda, m_{\tau_i}^\lambda} - v_{n_1 -1, x_i}^{\tau_i, \lambda, m_{\tau_i}^\lambda} \right\vert \leq (n_2 - n_1)\tau_i C \leq (t-s) C + \tau_i C.
	$$
	Let $i$ tend to infinity. We obtain that
	$$
	\left\vert \tilde{V} (t) - \tilde{V} (s) \right\vert \leq (t-s) C, \quad \forall s,t \in T.
	$$
	Let $V : (-\infty, 0) \rightarrow \R^d$ be the unique Lipschitz extension of $\tilde{V}: T \rightarrow \R^d$. Then $V_i(t) \rightarrow V (t)$ for every $t \in (-\infty, 0)$. Since
	$$
	\int_t^0 V_i(s) ds = x_i - \gamma_{\tau_i , m_{\tau_i}^\lambda}^{\lambda, x_i} (t), \quad \forall t<0,
	$$
	the claim is established. Then the curve $\gamma_0$ is of class $\CC^1$, and its derivative is Lipschitz with Lipschitz constant $C$. Moreover, by the dominated convergence theorem, it is clear that $\dot{\gamma}_{\tau_i, m_{\tau_i}^\lambda}^{\lambda, x_i} \rightarrow \dot{\gamma}_0$ in the $L^1$-norm on every compact subset of $(-\infty, 0]$.

We claim that
\begin{equation} \label{need0705}
	u_0^\lambda (x_0) - e^{\lambda t} u_0^\lambda \left( \gamma_0 (t) \right) = \int_t^0 e^{\lambda s} L_{m_0^\lambda} \left( \gamma_0 (s), \dot{\gamma}_0 (s) \right) ds, \quad \forall t \leq 0.
\end{equation}
For any $n \leq 0$, it is clear that
\begin{equation} \label{need0105}
	u_{\tau_i,m_{\tau_i}^\lambda}^\lambda (x_i) = \left( 1 - \tau_i \lambda \right)^{-n} u_{\tau_i,m_{\tau_i}^\lambda}^\lambda \left( \gamma_{\tau_i,m_{\tau_i}^\lambda}^{\lambda, x_i} (n \tau_i) \right) + \Sigma_{k = n}^{-1} \left( 1- \tau_i \lambda \right)^{-k-1} \tau_i L_{m_{\tau_i}^\lambda} \left( \gamma_{\tau_i, m_{\tau_i}^\lambda}^{\lambda, x_i} (k \tau_i), V_i (k \tau_i) \right).
\end{equation}
For any $t <0$, there exists an integer $n \leq 0$ such that $(n-1) \tau_i \leq t < n \tau_i$. Then
$$
I := \left\vert \Sigma_{k=n}^{-1} \left( 1 - \tau_i \lambda \right)^{-k-1} \tau_i L_{m_{\tau_i}^\lambda} \left( \gamma_{\tau_i, m_{\tau_i}^\lambda}^{\lambda, x_i} (k \tau_i), V_i (k \tau_i) \right) - \int_{n \tau_i} ^0 e^{\lambda s} L_{m_{\tau_i}^\lambda} \left( \gamma_{\tau_i, m_{\tau_i}^\lambda}^{\lambda, x_i} (s), \dot{\gamma}_{\tau_i, m_{\tau_i}^\lambda}^{\lambda, x_i} (s) \right) ds \right\vert \leq I_1 + I_2 + I_3,
$$
where
\begin{align*}
	I_1
	& = \Sigma_{k=n}^{-1} \left( 1 - \tau_i \lambda \right)^{-k-1} \int_{k \tau_i}^{(k+1) \tau_i} \left\vert L_{m_{\tau_i}^\lambda} \left( \gamma_{\tau_i, m_{\tau_i}^\lambda}^{\lambda, x_i} (k \tau_i), V_i (k \tau_i) \right) - L_{m_{\tau_i}^\lambda} \left( \gamma_{\tau_i, m_{\tau_i}^\lambda}^{\lambda, x_i} (s), \dot{\gamma}_{\tau_i, m_{\tau_i}^\lambda}^{\lambda, x_i} (s) \right) \right\vert ds \\
	& \leq \left( \sup_{x \in \T^d, \vert v \vert \leq D} \left\vert \frac{\partial L}{\partial x} (x,v) \right\vert + F_\infty \right) D \tau_i \frac{1- \left( 1 - \tau_i \lambda \right)^{-n}}{1-1+ \tau_i \lambda} \tau_i \\
	& \leq \left( \sup_{x \in \T^d, \vert v \vert \leq D} \left\vert \frac{\partial L}{\partial x} (x,v) \right\vert + F_\infty \right) D \frac{\tau_i}{\lambda}, \\
	I_2
	& \leq \Sigma_{k=n}^{-1} \left( \left( 1 - \tau_i \lambda \right)^{-k-1} - \left( 1 - \tau_i \lambda \right)^{-k} \right) \int_{k \tau_i}^{(k+1) \tau_i} \left\vert L_{m_{\tau_i}^\lambda} \left( \gamma_{\tau_i, m_{\tau_i}^\lambda}^{\lambda, x_i} (s), \dot{\gamma}_{\tau_i, m_{\tau_i}^\lambda}^{\lambda, x_i} (s) \right) \right\vert ds \\
	& \leq \left( \sup_{x \in \T^d, \vert v \vert \leq D} \left\vert L (x,v) \right\vert + F_\infty \right) \tau_i \left( 1 - \left( 1 - \tau_i \lambda \right)^{-n} \right) \leq \tau_i \left( \sup_{x \in \T^d, \vert v \vert \leq D} \left\vert L (x,v) \right\vert + F_\infty \right), \\
	I_3
	& = \Sigma_{k=n}^{-1} \int_{k \tau_i}^{(k+1) \tau_i} \left( e^{\lambda s} - \left( 1 - \tau_i \lambda \right)^{-k} \right) \left\vert L_{m_{\tau_i}^\lambda} \left( \gamma_{\tau_i, m_{\tau_i}^\lambda}^{\lambda, x_i} (s), \dot{\gamma}_{\tau_i, m_{\tau_i}^\lambda}^{\lambda, x_i} (s) \right) \right\vert ds \\
	& \leq \left( \sup_{x \in \T^d, \vert v \vert \leq D} \left\vert L (x,v) \right\vert + F_\infty \right) \left( \int_{n \tau_i}^0 e^{\lambda s} ds - \tau_i \Sigma_{k=n}^{-1} \left( 1 - \tau_i \lambda \right)^{-k} \right) \\
	& \leq \left( \sup_{x \in \T^d, \vert v \vert \leq D} \left\vert L (x,v) \right\vert + F_\infty \right) \tau_i.
\end{align*}
In conclusion, we obtain that
$$
I \leq D \left( \sup_{x \in \T^d, \vert v \vert \leq D} \left\vert \frac{\partial L}{\partial x} (x,v) \right\vert + F_\infty \right) \frac{\tau_i}{\lambda} + 2 \left( \sup_{x \in \T^d, \vert v \vert \leq D} \left\vert L (x,v) \right\vert + F_\infty \right) \tau_i,
$$
indicating that $\lim_{i \rightarrow \infty} I = 0$.
Besides, for any $t \leq 0$,
\begin{align*}
	& \left\vert \int_t^0 e^{\lambda s} \left( L_{m_0^\lambda} \left( \gamma_0(s), \dot{\gamma}_0(s) \right) - L_{m_{\tau_i}^\lambda} \left( \gamma_{\tau_i, m_{\tau_i}^\lambda}^{\lambda, x_i} (s), \dot{\gamma}_{\tau_i, m_{\tau_i}^\lambda}^{\lambda, x_i} (s) \right) \right) ds \right\vert \\
	\leq & \int_t^0 e^{\lambda s} \left\vert L_{m_0^\lambda} \left( \gamma_0(s), \dot{\gamma}_0(s) \right) - L_{m_0^\lambda} \left( \gamma_{\tau_i, m_{\tau_i}^\lambda}^{\lambda, x_i} (s), \dot{\gamma}_{\tau_i, m_{\tau_i}^\lambda}^{\lambda, x_i} (s) \right) \right\vert ds \\
	& + \int_t^0 e^{\lambda s} \left\vert L_{m_0^\lambda} \left( \gamma_{\tau_i, m_{\tau_i}^\lambda}^{\lambda, x_i} (s), \dot{\gamma}_{\tau_i, m_{\tau_i}^\lambda}^{\lambda, x_i} (s) \right) - L_{m_{\tau_i}^\lambda} \left( \gamma_{\tau_i, m_{\tau_i}^\lambda}^{\lambda, x_i} (s), \dot{\gamma}_{\tau_i, m_{\tau_i}^\lambda}^{\lambda, x_i} (s) \right) \right\vert ds \\
	& \leq -t \left( \sup_{x \in \T^d, \vert v \vert \leq C+ D} \left\vert \frac{\partial L}{\partial x} (x,v) \right\vert + F_\infty \right) \left( \sup_{s \in (t,0)} \left\vert \gamma_0 (s) - \gamma_{\tau_i , m_{\tau_i}^\lambda}^{\lambda, x_i} (s) \right\vert \right) \\
	&  + \left( \sup_{x \in \T^d, \vert v \vert \leq C+D} \left\vert \frac{\partial L}{\partial v} (x,v) \right\vert \right) \int_t^0  \left\vert \dot{\gamma}_0 (s) - \dot{\gamma}_{\tau_i , m_{\tau_i}^\lambda}^{\lambda, x_i} (s) \right\vert ds - t \Lip(F) d_1 (m_0^\lambda, m_{\tau_i}^\lambda) \rightarrow 0.
\end{align*}
Thus, \eqref{need0705} is established by (\ref{need0105}).
Thus, the minimizing curve $\gamma_0$ is of class $\CC^2$. The proof is finished by the uniqueness of the solution to (\ref{DMFG}a).

\subsection*{Appendix B. An example: the non-uniqueness of solutions to \eqref{DMFG}}
\addcontentsline{toc}{subsection}{Appendix B. An example: the non-uniqueness of solutions to \eqref{DMFG}}

Fix $\lambda>0$. Consider the following example
\begin{equation} \label{counterexample}
	\begin{cases}
		\lambda u + \frac{1}{2} \left\vert Du(x) \right\vert ^2= f(m) + g(x), & x \in \T^d
		\\
		\ddiv \left( m  Du(x) \right) = 0, & x \in \T^d 
		\\
		\int_{\T^d} m \, dx =1, 
	\end{cases}
\end{equation}
where  $f: \PP \left( \T^d \right) \rightarrow \R$ is Lipschitz and $g: \T^d \rightarrow \R$ is of class $\CC^2$. It is clear that $\argmin_{x \in \T^d} g(x)$ is non-empty.
Assume that $\argmin_{x \in \T^d} g(x)$ is not a singleton.

For any $m\in\PP(\T^d)$,
$$
L_m (x,v) = \frac{1}{2} \left\vert v \right\vert ^2 + f(m) + g(x), \quad \forall (x,v)\in \mathbb{T}^d\times\mathbb{R}^d.
$$
Take a point $x_0 \in \argmin_{x \in \T^d} g(x)$. Note that the curve $\gamma_0: t \mapsto x_0$ for $t \in \R$ satisfies the Euler-Lagrange equation \eqref{lagrangian flow of DMFG}
$$
\frac{d}{dt} \left( e^{\lambda t} \frac{\partial L_m}{\partial v} \left( \gamma_0(t), \dot{\gamma}_0(t) \right)  \right) = e^{\lambda t} \frac{\partial L_m}{\partial x}\left( \gamma_0 (t), \dot{\gamma}_0 (t) \right),
$$
i.e., $\phi^t_{L_m, \lambda} (x_0, 0) = (x_0, 0)$ for all $t \in \R$.
Let $\mu_0 := \delta_{(x_0, 0)}$, where $\delta_{(x_0, 0)}$ is the atomic measure supported on $(x_0, 0)$. 

 By the definition of $\mu_0$, it is clear that $\mu_0 \in \K \left( \T^d \times \R^d \right)$. % and $\supp (\mu_0) \subset \A_{L_m, \lambda}$. 
By \eqref{equation of solution of HJE_MFG}, we obtain that
$$
u^\lambda_m (x_0) \leq \int_{-\infty}^0 e^{\lambda s} L_m \left( \gamma_0(s), \dot{\gamma}_0 (s) \right) ds = \frac{f(m) + g(x_0)}{\lambda}.
$$
On the other hand, for any absolutely continuous curve $\gamma: (-\infty, 0] \rightarrow \T^d$ with $\gamma(0) = x_0$, there holds
$$
L_m (\gamma(s), \dot{\gamma}(s)) = \frac{1}{2} \left\vert \dot{\gamma} (s) \right\vert^2 +f(m) + g \left( \gamma(s) \right)  \geq f(m) + g(x_0) , \quad \forall s \leq 0.
$$
Therefore, 
$$
u^\lambda_m (x_0) \geq \left( \int_{-\infty}^0 e^{\lambda s} ds\right)  \left( f(m) + g(x_0) \right)  = \frac{f(m) + g(x_0) }{\lambda}.
$$
Consequently, we have $u^\lambda_m (x_0) = \frac{f(m) + g(x_0) }{\lambda}$.
Thus, for any $t, t^\prime \in \R$, there holds
$$
e^{\lambda t} u^\lambda_m \left( \gamma_0 (t) \right) - e^{\lambda t^\prime} u^\lambda_m \left( \gamma_0 (t^\prime) \right) = \frac{e^{\lambda t} - e^{\lambda t^\prime}}{\lambda} \left( f(m) + g(x_0) \right)  = \int_{t^\prime}^t e^{\lambda s} L_m \left( \gamma_0(s), \dot{\gamma}_0(s) \right) ds.
$$
Thus, the curve $\gamma_0$ is a $(u^\lambda_m, L_m)$-$\lambda$-calibrated curve. We conclude that $\supp(\mu_0) = \{(x_0, 0)\} \subset \tilde{\A}_{L_m, \lambda}$.
So far, we have proved that $\mu_0$ is a minimizing $\lambda$-measure for $L_m$. 

Let $m_0 = \delta_{x_0}$. By the same argument we used for $m$, one can deduce that $\mu_0$ is also a minimizing $\lambda$-measure for $L_{m_0}$. 
Define $u_0 = u^\lambda_{m_0}$ as in \eqref{equation of solution of HJE_MFG}.
Therefore, $(u_0,m_0)$ is a solution to \eqref{counterexample} by the definition of $u_0$ and the same method used in the proof of  Proposition \ref{m_0 is solution of continuity equation}.

More precisely, for any $x \in \A_{L_m, \lambda}$ and any $v \in \R^d$, we have
$$
L_{m_0} (x, v) + H_{m_0} \left( x, D u_0 (x)\right) = L_{m_0} (x, v) - \lambda u_0 (x).
$$
Integrating both sides with respect to $\mu_0$, we obtain that
$$
\int_{\T^d \times \R^d} \left(  L_{m_0} (x, v) + H_{m_0} \left( x, D u_0 (x)\right) \right) d \mu_0 =0.
$$
Thus, we obtain that
$$
L_{m_0} (x_0, 0) + H_{m_0} \left( x_0, D u_0 (x_0)\right) =0,
$$
indicating that $D u_0 (x_0) = 0$. Therefore,
$m_0$ is a solution of $\ddiv\left( m Du_0(x) \right) = 0$ in the sense of distributions. And thus $(u_0,m_0)$ is a solution to \eqref{counterexample}.

Take $x_1$, $x_2 \in \argmin_{x \in \T^d} g(x)$ with $x_1 \neq x_2$. Let $m_1 = \delta_{x_1}$ and $m_2 = \delta_{x_2}$. Define $u_1 = u^\lambda_{m_1}$ and $u_2 = u^\lambda_{m_2}$ as in \eqref{equation of solution of HJE_MFG}. Then we conclude that $(u_1, m_1)$ and $(u_2, m_2)$ are two distinct solutions to \eqref{counterexample}.

\bigskip

%%%%%%%%%%%%%%%%%%%%%%%%%%%%%%%%%%%%
%\noindent {\bf Statements and Declarations.} 
%This paper is the authors' original work and has not been published or submitted simultaneously elsewhere.

%\medskip

%\noindent {\bf Competing Interests.}
%Authors have no financial interests that are directly related to the work submitted for publication. We have no conflicts of interest to disclose.

%\medskip

%\noindent {\bf The Data Availability Statement.}
%No datasets were generated or analysed during the current study.

%\medskip

%\medskip
%%%%%%%%%%%%%%%%%%%%%%%%%%%%%%%%%%%%
\bibliographystyle{plain}
\bibliography{references}

\end{document}